\pdfoutput=1
\documentclass[reqno]{amsart}

\usepackage{amssymb}
\usepackage{graphicx}
\usepackage{tikz-cd}
\usepackage{enumitem}
\usepackage[bookmarks,hyperfootnotes=false, psdextra=true]{hyperref}
\usepackage[nameinlink, capitalise, noabbrev]{cleveref}
\crefformat{enumi}{#2#1#3}
\crefformat{equation}{#2(#1)#3}
\usepackage[msc-links]{amsrefs} 

\usepackage{geometry}
\let\setminus=\smallsetminus
\renewcommand{\leq}{\leqslant}
\renewcommand{\geq}{\geqslant}
\renewcommand{\ge}{\geq}
\renewcommand{\le}{\leq}

\let\rho=\varrho
\let\phi=\varphi

\newcommand{\id}{\normalfont\text{id}}

\DeclareMathOperator{\Aut}{Aut}

\renewcommand{\subset}{\subseteq}
\renewcommand{\supset}{\supseteq}

\newcommand{ \N } { \mathbb{N} }
\newcommand{ \Z } { \mathbb{Z} }

\makeatletter

\def\calCommandfactory#1{%
   \expandafter\def\csname c#1\endcsname{\mathcal{#1}}}
\def\frakCommandfactory#1{%
   \expandafter\def\csname frak#1\endcsname{\mathfrak{#1}}}

\newcounter{ctr}
\loop
  \stepcounter{ctr}
  \edef\X{\@Alph\c@ctr}
  \expandafter\calCommandfactory\X
  \expandafter\frakCommandfactory\X
  \edef\Y{\@alph\c@ctr}
  \expandafter\frakCommandfactory\Y
\ifnum\thectr<26 \repeat

\setenumerate{label={\normalfont (\roman*)}}

 \def\lowfwd #1#2#3{{\mathop{\kern0pt #1}\limits^{\kern#2pt\raise.#3ex
 \vbox to 0pt{\hbox{$\scriptscriptstyle\rightarrow$}\vss}}}}
 \def\lowbkwd #1#2#3{{\mathop{\kern0pt #1}\limits^{\kern#2pt\raise.#3ex
 \vbox to 0pt{\hbox{$\scriptscriptstyle\leftarrow$}\vss}}}}
 \def\fwd #1#2{{\lowfwd{#1}{#2}{15}}}
 \def\vS{{\vec S}_{\aleph_0}}

 \def\vS{{\hskip-1pt{\fwd S3}\hskip-1pt}}

 \def\vM{{\hskip-1pt{\fwd{M}{3.5}}\hskip-1pt}}
 \def\vN{{\hskip-1pt{\fwd{N}{3.5}}\hskip-1pt}}
 \def\vNdash{{\hskip-1pt{\fwd{N'}{3.5}}\hskip-1pt}}

 \def\ve{\kern-1.5pt\lowfwd e{1.5}2\kern-1pt}
 \def\ev{\kern-1pt\lowbkwd e{0.5}2\kern-1pt}
 \def\vf{\kern-2pt\lowfwd f{2.5}2\kern-1pt}

 \def\vs{\lowfwd s{1.5}1}
 \def\sv{\lowbkwd s{0}1}
 \def\vsi{\lowfwd {s_i}11}
 
 \def\vsiplus{\lowfwd {s_{i+1}}{-7}1}
 \def\vsdash{{\mathop{\kern0pt s\lower.5pt\hbox{${}
     \scriptstyle'$}}\limits^{\kern0pt\raise.02ex
     \vbox to 0pt{\hbox{$\scriptscriptstyle\rightarrow$}\vss}}}}
 \def\svdash{{\mathop{\kern0pt s\lower.5pt\hbox{${}
     \scriptstyle'$}}\limits^{\kern0pt\raise.02ex
     \vbox to 0pt{\hbox{$\scriptscriptstyle\leftarrow$}\vss}}}}
 \def\vr{\lowfwd r{1.5}2}
 \def\rv{\lowbkwd r02}
 \def\vri{\lowfwd {r_i}11}
 \def\rvi{\lowbkwd {r_i}11}
 
 \def\vrdash{{\mathop{\kern0pt r\lower.5pt\hbox{${}
     \scriptstyle'$}}\limits^{\kern0pt\raise.02ex
     \vbox to 0pt{\hbox{$\scriptscriptstyle\rightarrow$}\vss}}}}
 \def\rvdash{{\mathop{\kern0pt r\lower.5pt\hbox{${}
     \scriptstyle'$}}\limits^{\kern0pt\raise.02ex
     \vbox to 0pt{\hbox{$\scriptscriptstyle\leftarrow$}\vss}}}}

 \def\vSd{{\mathop{\kern0pt S\lower-1pt\hbox{${}
      \scriptstyle'$}}\limits^{\kern2pt\raise.1ex
      \vbox to 0pt{\hbox{$\scriptscriptstyle\rightarrow$}\vss}}}}

\newtheorem{theorem}{Theorem}[section] 
\newtheorem{proposition}[theorem]{Proposition}
\newtheorem{corollary}[theorem]{Corollary}
\newtheorem{lemma}[theorem]{Lemma}

\newtheorem{mainresult}{Theorem}

\newtheorem{reduction}{Reduction}

\newtheorem{innercustomthm}{}

\newenvironment{customthm}[1]
  {\renewcommand\theinnercustomthm{#1}\innercustomthm}
  {\endinnercustomthm}

\theoremstyle{definition}
\newtheorem{example}[theorem]{Example}

\newtheorem{definition}[theorem]{Definition}
\newtheorem{construction}[theorem]{Construction}

\theoremstyle{remark}

\graphicspath{ {./figures/} }

\newcommand{\proper}{regular}

\newcommand{\CS}{\cZ}
\newcommand{\CSr}{\cZ_r}

\newcommand{\rh}{r/2}
\newcommand{\rhoh}{\varrho/2}

\newcommand{\act}{\cdot}
\newcommand{\loc}{{G_r}}

\renewcommand{\O}{O}
\newcommand{\minor}{\preccurlyeq}
\newcommand{\ballpres}{-ball-pre\-serv\-ing}

\newcommand{\accessible}{{accessible}}
\newcommand{\deckr}{{\cD_r}}
\newcommand{\normcl}[2]{\langle #1\, \rangle_{#2}^\triangleleft}
\newcommand{\pres}[2]{\langle\, #1 \mid #2 \, \rangle}
\newcommand{\free}[1]{F(#1)}

\DeclareMathOperator{\cay}{Cay}
\newcommand{\cayley}[2]{\cay(#1, #2)}
\newcommand{\can}{canonical}
\newcommand{\dcan}{$\cD$-canon\-ic\-al}
\newcommand{\dpcan}{$\cD(p)$-canon\-ic\-al}

\crefrangeformat{section}{Sections~#3#1#4--#5#2#6}

\begin{document}
	\setlength{\fboxsep}{0pt}
	\setlength{\fboxrule}{.1pt}
	\vspace*{-24mm} 
	\title{Canonical Graph Decompositions via Coverings
	\ifArXiv\\[\smallskipamount] \rm\smaller Extended version of~\cite{GraphDec}, ArXiv only\fi}
	
\newif\ifArXiv \ArXivfalse 

	\ifArXiv\vspace*{-9pt}\fi 

\def\?#1{\vadjust{\vbox to 0pt{\vss\vskip-8pt\leftline{%
     \llap{\hbox{\vbox{\pretolerance=-1
     \doublehyphendemerits=0\finalhyphendemerits=0
     \hsize25truemm\tolerance=10000\small
     \lineskip=0pt\lineskiplimit=0pt
     \rightskip=0pt plus25truemm\baselineskip8pt\noindent
     \hskip0pt        
     #1\endgraf}\hskip0.1truemm}}}\vss}}}
	\newcount\commentno
\def\COMMENT#{}%
\def\td{tree-decom\-pos\-ition}
\def\gd{graph-decom\-pos\-ition}
\def\|{\!\!\restriction\!\!}
\def\phiT{\phi}

\newtheorem*{example*}{Example}

	\author{Reinhard Diestel}
	
	\author{Raphael~W.\ Jacobs}
	
	\author{Paul Knappe}
	
	\author{Jan Kurkofka}
	
	\keywords{Graph, local, global, connectivity, covering space, tangle, canonical \td, Cayley graph, finitely generated group, accessible group}
	\@namedef{subjclassname@2020}{\textup{2020} Mathematics Subject Classification}
	\subjclass[2020]{05C83, 57M15, 05C25, 05C40, 05C63, 05C38}
	
	\begin{abstract}
		We present a canonical way to decompose finite graphs into highly connected local parts.
		The decomposition depends only on an integer parameter whose choice sets the intended degree of locality.
		The global structure of the graph, as determined by the relative position of these parts, is described by a coarser \emph{model}.
		This is a simpler graph determined by the decomposition, not imposed.
		
		The model and decomposition are obtained as projections of the tangle-tree structure of a covering of the given graph that reflects its local structure while unfolding its global structure. In this way, the tangle theory from graph minors is brought to bear canonically on arbitrary graphs, which need not be tree-like.
		
		Our theorem extends to locally finite quasi-transitive graphs, and in particular to locally finite Cayley graphs.
		It thereby yields a canonical decomposition of finitely generated groups into local parts, whose relative structure is displayed by a graph.
	\end{abstract}
	
	\maketitle
	
	\vspace{-3mm}

\section{\boldmath Introduction}
	
\noindent
	The question to what extent graph invariants---the chromatic number, say, or connectivity---are of local or global character, and how their local and global aspects interact, drives much of the research in graph theory both structural and extremal~\cite{ScottICM}.
	In this paper we offer such studies a possible formal basis.
	
	We show that every finite graph~$G$ decomposes canonically into local parts which, between them, form a global structure displayed by a simpler graph~$H$.
	Both $H$ and the decomposition $(H, (G_h)_{h\in H})$ of~$G$ into parts indexed by the nodes of~$H$ are unique once we have set an integer parameter~$r>0$ to define our desired threshold between `local' and `global':
	cycles in~$G$ of length~$\le r$ are considered as local, and are reflected inside the parts~$G_h$ of the decomposition, while the cycles of $H$ reflect only global cycles of~$G$, longer cycles that are not generated by the short ones.
	
	\vskip-3pt

	\begin{figure}[ht]
		\center
		\includegraphics[scale=4]{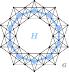}
	\vskip-8pt
		\caption{The global structure of $G$ is displayed by a cycle~$H$. Its local parts are~$K^5$s.}
		\label{fig:CycleDef}
	\end{figure}
	\vskip-5pt

	Our main result reads as follows:
	

	\begin{mainresult} \label{main:KeyTheorem}
		Let $G$ be any finite graph, and $r>0$ an integer.
		Then $G$ has a unique canonical decomposition modelled on another finite graph~$H=H(G,r)$ that displays its $r$-global structure.
	\end{mainresult}
	
	\noindent
	See \cref{subsec:Background,subsec:ConstructionH,subsec:MainThmPrecise} below for a summary of the definitions needed here. More details are given in  \cref{def:r-globalstructure} and the text following it.%
	\COMMENT{}

\medbreak
	
	Since $H=H(G,r)$ and the associated decomposition of~$G$ are canonical, they are invariants of~$G$, and we can study how they interact with other graph invariants as $r$ ranges between~$1$ and~$|G|$.
	For example, we may study how the chromatic number, connectivity, or edge density of $G$ is reflected in~$H$, and is in this sense global, or is reflected in the parts~$G_h$ of the decomposition of~$G$, and is in that sense local.

	\cref{main:KeyTheorem} extends to connected Cayley graphs even when these are only locally finite.
	It thus implies a canonical decomposition theorem for finitely generated groups into local parts, whose relative structure is displayed by a graph:
	
	\begin{mainresult} \label{main:CayleyThm}
		Let $\Gamma$ be a group given with a finite set~$S$ of generators.
		For every integer $r>0$, the Cayley graph~$\cayley{\Gamma}{S}$ has a unique canonical decomposition modelled on a graph $H=H(\Gamma,S,r)$ that displays the $r$-global structure of the group~$\Gamma$ as presented by~$S$.
	\end{mainresult}

Our results have further group-theoretical implications. We shall see that finitely generated groups~$\Gamma$ extend to%
 	\COMMENT{}
	finitely presented groups~$\Gamma_r$ that describe their $r$-{\em local\/} structure, the aspect of their structure that complements their $r$-global structure given by \cref{main:CayleyThm}. Similarly, we shall see that the $r$-local structure of any finite graph~$G$ that complements its $r$-global structure as described in \cref{main:KeyTheorem} is displayed by a graph~$\loc$ modelled on the Cayley graph of a finitely presented group.%
	\COMMENT{}
	This graph~$\loc$ is typically infinite. It describes the local structure of~$G$ by way of a canonical covering $\loc\to G$. See \cref{sec:LocalCoverings} for details. Just recently, similar graph-decomposition methods have been used by Carmesin, Kontogeorgiou, Kurkofka and Turner to prove a low-order Stallings-type theorem for finite nilpotent groups~\cite{carmesin2024stallings}. More group-theoretical applications of this approach have been obtained by Carmesin~\cite{Localkseps}.

\medbreak
	
	The idea for our construction of $H=H(G,r)$ was inspired by the aims of Carmesin~\cite{Local2seps}.
	It is to apply the theory of tangles from graph minor theory not to $G$ itself, but to the covering space~$\loc$ of~$G$ that reflects its short cycles, those of length~$\le r$, while unfolding in a tree-like way any longer cycles not generated by the short ones.
	This tree-likeness of this graph~$\loc$ enables us to use tangles to structure~$\loc$ as customary in graph minor theory: by canonical tangle-distinguishing \td s.
	Projecting this structure back to~$G$ defines~$H$ and the decom\-po\-si\-tion of~$G$ modelled on~$H$.
	
	Our original proof of \cref{main:KeyTheorem} also used group-theoretic tools in an essential way.
	In order to apply tangle-distinguishing \td s to our covering graph~$\loc$, we needed that $\loc$ resembles a Cayley graph of some finitely presented group that we could decompose using Stallings' theorem~\cite{stallings1972group}, and then use~\cite{ThomassenWoess} to analyse its ends.
	We formalised this resemblance by showing that $\loc$ had a decomposition into bounded-sized finite parts modelled on a Cayley graph of its group of deck transformations over~$G$.
	This group is finitely presented not only when $G$ is finite (\cref{main:KeyTheorem}), but also when it is the Cayley graph of an arbitrary finitely generated group (\cref{main:CayleyThm}).
	We have since been able to replace these group-theoretical tools by purely combinatorial arguments, which is how we present our proof here.
	
	In the remainder of this introduction we indicate just enough background to be able to make the statement of \cref{main:KeyTheorem} precise, and to give a brief sketch of its proof.

	\vskip 1.5 \smallskipamount\subsection{\boldmath Tree-decom\-pos\-itions vs.\ \gd s} \label{subsec:Background}
	
	Early in their seminal work on graph minors, Robertson and Seymour~\cite{GM} introduced \emph{\td s} of graphs as a way to display their global structure when this is tree-like.
	However, not all graphs are globally tree-like.
	For those that are not, \td s are a poor fit to display their global structure.
	
	Yet, conversely, few graphs are totally homogeneous at all levels of local or global focus: most have some clusters that are particularly highly connected.
	These clusters then form a global structure between them, even if this is not tree-like.
	\cref{main:KeyTheorem} makes this global structure visible in the form of a graph $H$ that need not be a tree.
	Indeed, $H$~will change with the graph~$G$ being decomposed and with the locality parameter~$r$ we choose:
	the global structure of~$G$ is \emph{found} by our theorem, not imposed as in the case of \td s.
	
	Let us make this precise.
	Let $H$ be a graph.
	A~\emph{decomposition} of a graph~$G$ \emph{modelled on}~$H$, or an {\em $H$-decom\-pos\-ition\/} of~$G$, is a pair $(H,\cV)$ consisting of $H$ and a family $\cV = (V_h)_{h\in H}$ of sets of vertices of~$G$ that are associated with the nodes $h$ of~$H$ in such a way that
	\begin{itemize}
		\item $\bigcup_{h\in H} G[V_h] = G$;
		\item for every vertex $v\in G$, the set~$W_v$ of nodes~$h$ of~$H$ whose $V_h$ contains~$v$ is connected in~$H$.
	\end{itemize}
	The decompositions of a graph that are modelled on trees are thus precisely its \td s~\cite{DiestelBook25}.

Given a decomposition as above, we may choose spanning subgraphs $G_h \subseteq G[V_h]$ and ${H_v\subseteq H[W_v]}$ so that the~$G_h$ still cover all the edges of~$G$ and the $H_v$ are still connected. We then call these~$G_h$ the {\it parts\/} of our decomposition, and the~$H_v$ its {\it co-parts\/}; see \cref{def:GrDec} for details.
	
	Our decompositions and choice of their parts will be \emph{canonical} in that they commute with graph isomorphisms:
	if $(G_h)_{h\in H}$ and~$(G'_{h'}{)}_{h'\in H'}$ are the families of parts of decompositions of graphs $G$ and~$G'$ that witness%
	\COMMENT{}
	\cref{main:KeyTheorem}, then any graph isomorphism $\sigma\colon G\to G'$ maps the parts~$G_h$ of~$G$ to the parts~$G'_{h'}$ of~$G'$ in such a way that $h\mapsto h'$ is a graph isomorphism~$H\to H'$.
	In particular, the graph~$H$ on which \hbox{\cref{main:KeyTheorem}} models a given graph~$G$ is unique, up to isomorphism, for every choice of~$r$.

\vskip 1.5 \smallskipamount\subsection{\boldmath Construction of \texorpdfstring{\boldmath $H=H(G,r)$}{H=H(G,r)}} \label{subsec:ConstructionH}
	
	The construction of $H=H(G,r)$ is not hard to describe.
	It works as follows.
	
	Given $G$ and~$r$, let $\pi_1^r(G,v_0)$ denote the subgroup of the fundamental group $\pi_1(G,v_0)$ of~$G$, based at a vertex~$v_0$, that is generated by the elements represented by a walk in~$G$ from $v_0$ to a cycle of length at most~$r$, round it, and back along the access path.
	This~$\pi_1^r(G,v_0)$ is a normal subgroup of the fundamental group of~$G$, independent of the choice of basepoint;%
   \COMMENT{}
   let $p_r\colon\loc\to G$ denote the normal covering of~$G$ with $\pi_1^r(G,v_0)$ as characteristic subgroup.
	
	By our choice of~$r$, the cover~$\loc$ reflects all the short cycles of~$G$, as well as those of its longer cycles that are generated in~$\pi_1(G)$ by the short ones.
	The other longer cycles of~$G$ are usually unfolded to 2-way infinite paths, or \emph{double rays}.
	(One can construct examples where $\loc$ covers~$G$ with finitely many sheets, but those are rare.)
	Thus, $\loc$ mirrors~$G$ in its `$r$-local' aspects, but not in its `$r$-global' aspects, where it is simply tree-like.
	This is why we call $\loc$ the \emph{$r$-local\/} cover of $G$.
	
	As a consequence of this tree-likeness of the global aspects of~$\loc$, tree-decomposi\-tions will be a better fit for~$\loc$ than they were for~$G$, whose global aspects could include long cycles that would not fittingly be captured by \td s.
	We exploit this by applying to~$\loc$, not to~$G$, the tree-decompo\-sitions that represent the state of the art from the theory of graph minors:
	those that distinguish all the maximal blocks, tangles and ends of~$G$ efficiently, and are canonical in the sense described earlier~\cites{ProfilesNew,InfiniteSplinters}.
	We fix one such canonical \td\ $\cT_r = (T_r,(V_t)_{t\in T_r})$ that is particularly natural, and use it to define
	the $r$-global structure of~$G$.
	More on this in \cref{subsec:MainThmPrecise}.
	
	Since our \td\ $\cT_r$ of~$\loc$ is canonical, the group $\deckr$ of deck-transforma\-tions of~$\loc$ over~$G$ acts on its model, the tree~$T_r$.
	The canonical model~$H$ for our $r$-local decomposition~$\cG_r$ of~$G$ is then obtained as the orbit space~$T_r/\deckr$, which is a graph.
	The parts $G_h$ of~$\cG_r$ are the projections to~$G$ of the parts~$G[V_t]$ of~$\cT_r$ whenever $h$ is the $\deckr$-orbit of~$t$.%
   \COMMENT{}

\vskip 1.5 \smallskipamount\subsection{\boldmath Details of \cref{main:KeyTheorem}}\label{subsec:MainThmPrecise}
	
	In order to make \cref{main:KeyTheorem} precise we have to define what it means for a graph~$H$ to `display the $r$-global structure' of a given graph~$G$.
	We shall first give a more%
	\COMMENT{}
	formal definition of the `$r$-global structure' of~$G$, and then define what it means that $H$ displays it.
	Let $r > 1$ be given.
	
	Consider the $r$-local cover~$\loc$ of~$G$ from \cref{subsec:ConstructionH}.
	This is unique up to isomorphisms (of coverings of~$G$), and it is clearly a graph.
	The $r$-global structure of~$G$ is formally defined as a particular canonical {\em tree of tangles\/} of~$\loc$%
   \COMMENT{}
   along with the group $\deckr$ of deck transformations of $\loc$ over~$G$.
	Such a `tree of tangles' is not a graph.
	It is a central notion in the theory of graph minors; let us indicate briefly what it means.
	
	Our covering graph~$\loc$ may have highly cohesive regions, or `clusters'.
	These will include%
	\COMMENT{}
	lifts of the local clusters in~$G$ that our $H$-decom\-pos\-ition of~$G$ is meant to distinguish, in that they should live in distinct parts~$G_h$.
	There are various notions of such clusters in graph theory.
	The most general of these are known as \emph{blocks}~\cites{confing,CG14:isolatingblocks} and as \emph{tangles}~\cite{GMX}.
	When $\loc$ is infinite, as it usually will be, it will also have \emph{ends}.
	Ends can be formalised as infinite tangles, and all blocks except some irrelevant small ones also induce tangles.
	We shall therefore use the term `tangle' from now on to cover all three of these: blocks, ends, and traditional tangles in the sense of Robertson and Seymour~\cite{GM}.
	Details will be given in \cref{subsec:TanglesBackground}.

	All the tangles of our $r$-local cover~$G_r$ of~$G$ can be `distinguished' by a nested set~$N$ of separations of~$\loc$, in the following sense.
	Given any separation of~$\loc$, of order~$k$ say, every tangle `of order~${>k}$' in~$\loc$ will lie essentially on one of the two sides of this separation, and thereby \emph{orient} it towards that side.
	A~set $N$ of separations is said to \emph{distinguish} the tangles of~$\loc$ if for any two of them that orient some separation of $\loc$ differently there is such a separation even in~$N$. If $N$ is nested and even contains, for every pair of tangles, such a separation of minimum order (among all the separations of~$\loc$ that distinguish this pair of tangles), we say that $N$ distinguishes the tangles of~$\loc$ {\em efficiently\/} and call it a \emph{tree of tangles} for~$\loc$.
	
	While $\loc$ can have many trees of tangles, one of these stands out.
	Let $D$ be the set of all separations of~$\loc$ that distinguish some pair of tangles efficiently.
	Now given any pair of tangles that can be distinguished by some separation of~$\loc$, associate with this pair  the set of all those separations in $D$ which distinguish it efficiently and cross as few other separations in $D$ as possible.
	Let $N$ be the union of all these sets.
	This set $N$ clearly distinguishes every pair of tangles efficiently, and it is known to be nested~\cite{carmesin2022entanglements}.%
	\COMMENT{}
		
	Thus, $N$~is a tree of tangles for~$\loc$. It is canonical in the sense mentioned in \cref{subsec:Background}, since it is defined purely in terms of invariants of~$\loc$.
	In particular, it is invariant under the deck transformations of~$\loc$ over~$G$ (whose group we denote by~$\deckr$).
	It is thus a well-defined set of separations of~$\loc$ that is unique, given $G$ and~$r$.
	In recognition of the fact that this tree of tangles of $\loc$ distinguishes the lifts of all the local clusters in~$G$ and structures them in~$\loc$, we call it, together with~$\deckr$ (which reflects this global structure of $\loc$ back to~$G$), the \emph{$r$-global structure} of~$G$.
	
	This $r$-global structure of $G$ appears at the level of $G$ as the quotient $N/\deckr$.%
   \COMMENT{}
	This is not a graph.
	But we show that it can be displayed by a graph, the graph $H$ from \cref{main:KeyTheorem}, as follows.%
	\COMMENT{}
	
	Every nested set $N$ of separations of a finite graph can be converted into a \td\ of that graph so that the separations of that graph which correspond to edges of the decomposition tree (the model of the \td) are precisely those in~$N$.
	Infinite graphs in general can have trees of tangles that cannot be represented by a \td\ in this way, because they can have limits under the natural partial ordering of graph separations~\cite{ASS}. 
	For our particular $\loc$, however, we shall be able to prove that~$N$, our canonical tree of tangles for $\loc$, does not in fact have problematic limits,\vadjust{\penalty-200} and thus gives rise to a canonical \td\ $\cT_r = (T_r,(V_t)_{t\in T_r})$ of~$\loc$.

The orbit space $H = T_r/\deckr$ of the tree~$T_r$ is a graph. We say that~$H$, and the $H$-decom\-pos\-ition $\cG_r = (H,(V_h)_{h\in H})$ of~$G$ such that $V_h = p_r(V_t)$ whenever $h$ is the $\deckr$-orbit of~$t$, {\em display\/} the $r$-global structure~$(N, \deckr)$ of~$G$.%
   \COMMENT{}
	As parts~$G_h$ of~$\cG_r$ we choose the projections to~$G$ of the parts~$G_r[V_t]$ of~$\cT_r$ with $t\in h$. These~$G_h$ need not be induced in~$G$ although the~$G_r[V_t]$ are induced in~$G_r$, since distinct lifts~$V_t$ of~$V_h$ can be adjacent in~$G_r$. We remark that $H$ is finite if $G$ is finite.

	\vskip 1.5 \smallskipamount\subsection{\boldmath Overview of the proofs of \cref{main:KeyTheorem} and \cref{main:CayleyThm}} \label{subsec:IntroOverviewProof}
	
	We will prove the following common generalisation of these two theorems: every connected, locally finite, and quasi-transitive graph $G$ has a \gd\ that displays its $r$-global structure.
	(A~graph is \emph{quasi-transitive} if its automorphism group has finitely many vertex-orbits.)
	We will see that, if $G$ is quasi-transitive, then $\loc$ is quasi-transitive too.\looseness=-1
	
	Our key task will be to prove that the quasi-transitive graph $\loc$ is \emph{tangle-accessible}; this means that all its tangles are pairwise distinguished by separations of some finitely bounded order.
	For such tangle-accessible~$\loc$ our tree of tangles~$N$ exists~\cite{carmesin2022entanglements}, and the separations in $N$ also have finitely bounded order.
	This implies that $N$ has no problematic limits.
	In particular, $N$~gives rise to a canonical \td\ of $\loc$, which will in turn define a canonical \gd\ of $G$ that displays its $r$-global structure.
	
	To show that the quasi-transitive graph $\loc$ is tangle-accessible, we will first use a result of Hamann~\cite{hamann2018accessibility} to see that $\loc$ is \emph{accessible}: that the ends of $\loc$ can be pairwise distinguished by separations of finitely bounded order. By our above considerations there will then exist a canonical \td\ of~$\loc$, of finitely bounded adhesion, that distinguishes all the ends of~$\loc$.
	We will find a specific such \td, $\cT_{\rm end}$, whose infinite parts are 1-ended and themselves quasi-transitive, even just under automorphisms that extend to automorphisms of~$\loc$.
	
	This will allow us to refine $\cT_{\rm end}$ in each of its parts so that the arising canonical \td~$\cT$ distinguishes all the tangles of $\loc$, not just its ends.
	This $\cT$ will again have finitely bounded adhesion.%
   \COMMENT{}
	It is not the canonical \td\ of~$\loc$ we ultimately seek (the one whose projection to~$G$ defines our desired $H$-decom\-pos\-ition of~$G$), but it witnesses that $\loc$ is tangle-accessible and thus completes our proof.\looseness=-1

\vskip 1.5 \smallskipamount\subsection{\boldmath A windfall from the proof}\label{subsec:windfall}
	
	The concept of accessibility for locally finite graphs was introduced by Thomassen and Woess~\cite{ThomassenWoess}. It follows their graph-theoretic characterisation of the finitely generated groups that are accessible, in the sense of Wall~\cite{WallAccessibility}, by Stallings's group splitting theorem~\cite{stallings1972group}.%
	\COMMENT{}
	Hamann, Lehner, Miraftab, and R\"{u}hmann~\cite{hamann2018stallings} extended this characterization to quasi-transitive graphs that need not be Cayley graphs.%
	\COMMENT{}
	
	In the course of our proof we show that not only the ends of an accessible quasi-transitive graph can be distinguished by separations of bounded order, but all its other tangles can be too:
			
\begin{mainresult}\label{maincor:TangleAccessibility}
	Accessible locally finite, quasi-transitive graphs are in fact%
	\COMMENT{}
	tangle-accessible.
	\end{mainresult}
	
\noindent
	Applied to Cayley graphs of groups, \cref{maincor:TangleAccessibility} yields a structural description of all accessible groups that refines their structural decomposition by Stallings's theorem.%
   \COMMENT{}

\vskip 1.5 \smallskipamount\subsection{\boldmath Applications}\label{subsec:applications}
	
We think of the decomposition theory advanced in this paper, and in particular of \cref{main:KeyTheorem}, not so much as a tool with which to attack existing problems in graph theory, but as a natural way to view graphs and to analyse their structure from first principles. We do believe that our decompositions have the potential to interact with other graph invariants, by splitting them into a local and a global aspect as indicated after the statement of \cref{main:KeyTheorem}. But how exactly this can happen may not be straightforward, and may require non-trivial structural analysis. This will be worthwhile if, but also only if, our new invariant of $H(G,r)$ and the associated decomposition of~$G$ are considered as a natural lens through which graph structure may be viewed. We believe they are.

Still, let us sketch one particular potential application of graph decompositions of this type, just to indicate the idea. \cref{main:KeyTheorem} includes, as the special case of $r=|G|$, the canonical strengthening of the so-called {\em tree of tangles\/} theorem of Robertson and Seymour~\cite{GMX}. This proves the existence of a \td\ that separates out the most highly connected regions of a graph~$G$ in a tree-like way. If $G$ is not globally tree-like, this decomposition may be trivial or otherwise fail to capture the global structure of~$G$.
If $G$ is globally tree-like, which happens when the graph does not contain some other fixed graph as a minor but is much larger than it, the parts of such a \td\ can be fleshed out in a way that Robertson and Seymour~\cite{GMXVI} have described in detail: they are each `nearly' embeddable in one of a bounded number of surfaces. Both this bound and the notion of `nearly' depend only on the excluded minor, not on the graph being decomposed.

For smaller~$r$ than $r=|G|$, our \cref{main:KeyTheorem} no longer returns a canonical \td\ of~$G$ as in the Robertson-Seymour result. Instead, it returns the projection to~$G$ of that \td\ applied to its $r$-local cover. The parts of that projection, those of our canonical $H$-decomposition of~$G$ for $H=H(G,r)$, are not yet fleshed out by structural descriptions in terms of surface embeddings. But it should be possible to do that if it can be done for the \td\ of the cover~--%
   \COMMENT{}
   which Esperet, Giocanti and Legrand-Duchesne~\cite{esperet2023structure} have indeed done recently, albeit with a slightly different canonical \td\ than ours. The result would be a structure theorem, for arbitrary finite graphs~$G$ and a given locality parameter~$r$, that $G$ has a canonical $H$-decomposition for $H=H(G,r)$ whose parts we expect to be%
   \COMMENT{}
   nearly embeddable in surfaces in which~$G$, as a whole, does not embed.\looseness=-1

Pursuing applications such as this, which seek to apply global structural descriptions of tree-like graphs to the (tree-like) coverings of~$G$ in order to then translate them into self-contained theorems about $G$ itself that describe its local structure, we in particular have to be able to describe the decomposition of~$G$ and its parts without reference to~$G_r$ and the covering map, as happens in its definition. We have indeed shown in a sequel to this paper~\cite{LocalSeps} that such a finite description always exists, under only mild assumptions about $G$ and~$r$. This can then form the basis for further structural analysis of the parts. Carmesin and Frenkel have used such a finite description of low-order graph-decompositions~\cite{Local2seps} to develop an efficient algorithm for computing them~\cite{carmesin2023apply}. They have implemented their algorithm and tested it successfully on road networks.

\vskip 1.5 \smallskipamount\subsection{\boldmath Structure of the paper}
	
	In \cref{sec:ToolsAndTerminology} we collect together notation and basic facts with an emphasis on coverings.
	In~\cref{sec:GraphDecomps} we introduce \gd s, and show how we can obtain them from canonical \td s via coverings.
	Local coverings are introduced in \cref{sec:LocalCoverings}, where we also investigate their properties as well as local coverings of some special classes of graphs.
	In \cref{sec:OverviewProof} we define formally what it means for a \gd\ to `display the $r$-global structure' of a graph, and outline what remains to be done to prove \cref{main:KeyTheorem}.
	The proof is then completed in \cref{sec:LocalCoveringsAreAccessible,sec:TangleAccessibility,sec:ProofMainTheorem}.

\section*{\boldmath Acknowledgement}

A~referee suggested that we point out the following more topological way to define the $r$-local cover of~$G$. Rather than taking the covering space of~$G$ with $\pi_1^r(G,v_0)$ as characteristic subgroup, we could glue a disc on every cycle of length at most~$r$ in~$G$ and then work with the universal cover of the resulting 2-complex. We are happy to invite the reader to keep this alternative perspective in mind as the paper proceeds.

\bigbreak\section{\boldmath Terminology} \label{sec:ToolsAndTerminology}

\noindent
	The purpose of this section is to introduce our terminology wherever this may not be obvious or differ from standard usage.
	\ifArXiv\else
	Formal definitions of terms not defined here can be found in the extended version of this paper~\cite{GraphDecArXiv}, which also contains supplementary proofs and examples.%
	\fi

\vskip 1.5 \smallskipamount\subsection{\boldmath Graphs}
	Our graph-theoretic terminology follows~\cite{DiestelBook25}, with the following adaptations.

	A {\em graph\/}~$G$ consists of a set $V(G)$ of {\em vertices\/} or {\em nodes\/}, a~set $E(G)$ of {\em edges\/} that is disjoint from~$V(G)$, and an {\em incidence\/} map that assigns to every edge~$e$ an unordered pair $uv=\{u,v\}$ of vertices, its {\em endvertices\/}. Note that our graphs may have loops and parallel edges: we allow~$u=v$, and distinct edges may have the same pair of endvertices.

The triples $(e,u,v)$ and~$(e,v,u)$ are the two {\em orientations\/} of an edge~$e$; we think of $(e,u,v)$ as {\em $e$~oriented from~$u$ to~$v$}. When $G$ has no parallel edges, e.g.\ when $G$ is a tree, we abbreviate $(e,u,v)$ to~$(u,v)$.%
   \COMMENT{}

A graph is \emph{locally finite} if each of its vertices is incident with only finitely many edges, and {\em finite\/} if its vertex and edge sets are both finite.%
   \COMMENT{}
   	An \emph{isomorphism} between graphs $G$ and~$H$ is a pair $\phi = (\phi_V, \phi_E)$ of bijections $\phi_V \colon V(G) \to V(H)$ and $\phi_E \colon E(G) \to E(H)$ that commute with the incidence maps in $G$ and~$H$. We abbreviate both $\phi_V$ and~$\phi_E$ to~$\varphi$.

\vskip 1.5 \smallskipamount\subsection{\boldmath Groups}
	Our group-theoretic terminology largely follows~\cite{loh2017geometric}.%
	\COMMENT{}
   	Given sets $R\subset S$ we write $\free{S}$ for the free group that $S$ generates, $\normcl{R}{\free{S}}$~for its normal subgroup generated by~$R$, and $\pres{S}{R}$ for the quotient group $\free{S} / \normcl{R}{\free{S}}$ presented with generators~$S$ and relators~$R$. Given a group~$\Gamma$ with a set $S$ of generators, we write~$\phi_{\Gamma, S}$ for the canonical homomorphism $\free{S} \to \Gamma$ that extends the identity on~$S$.

\vskip 1.5 \smallskipamount\subsection{\boldmath Interplay of graphs and groups}\label{sec:interplayofgraphsandgroups}

When a group $\Gamma$ acts on a graph~$G$, we denote the action of an element $g\in\Gamma$ on a vertex or edge $x\in G$ by $x\mapsto g\act x$. We write $G/\Gamma$ for the graph with vertex set $\{\Gamma \cdot v \colon v \in V(G) \}$ and edge set $\{\Gamma \cdot e \colon e \in E(G)\}$ in which the edge $\Gamma\cdot e$ has endvertices $\Gamma\cdot u$ and $\Gamma\cdot v$, where $u$ and $v$ are the endvertices of~$e$ in~$G$.
	Note that $G/\Gamma$ can have loops and parallel edges even if $G$ does not.
	We call $G/\Gamma$ the \emph{orbit graph} of the action of~$\Gamma$ on~$G$.
	
	We say that $\Gamma$ acts \emph{transitively} on~$G$ if it does so with only one vertex orbit, and \emph{quasi-transitively} if it does so with only finitely many vertex orbits.%
	\COMMENT{}
	The graph~$G$ itself is \emph{(quasi-) transitive} if its full automorphism group acts (quasi-) transitively on its vertices.
	
	Let $\Gamma$ be a group and $S\subset\Gamma$, usually a set of generators.
	The \emph{Cayley graph} $\cayley{\Gamma}{S}$ is the graph with vertex set~$\Gamma$, edge set $\{(g,s) : g \in \Gamma,\ s \in S\}$ and incidence map $(g,s)\mapsto \{g,gs\}$. We allow $S$ to contain inverse pairs of elements. If it contains both $s$ and~$s^{-1}$, say, and $gs=h$ in~$\Gamma$, then $\cayley{\Gamma}{S}$ has two edges incident with $g$ and~$h$: the edge~$(g,s)$ and the edge~$(h,s^{-1})$.%
   \COMMENT{}
   Every group acts on its Cayley graphs by left-multiplication, as $h \act g := hg$ and $h \act (g,s) := (hg,s)$.

\vskip 1.5 \smallskipamount\subsection{\boldmath Topological and covering spaces}\label{subsec:coverings} 

In our topological notation we follow Hatcher~\cite{Hatcher}. In particular, when $X$ is a path-connected topological space and $\alpha\colon [0,1]\to X$ a path, we write $\alpha^{-}\colon t\mapsto \alpha(1-t)$ for the path that traverses $\alpha$ backwards.
	Given two paths $\alpha$ and $\beta$ with $\alpha(1)=\beta(0)$ we write $\alpha\beta$ for their concatenation, the path which traverses $\alpha$ first and then $\beta$, both at double speed. Multiplication in the fundamental group is denoted by $[\alpha][\beta]:=[\alpha\beta]$.%
   \COMMENT{}

For a normal subgroup~$S$ of $\pi_1(X,x_0)$ and any $x \in X$ we write $S_{x}$ for the subgroup $\{\,[\alpha\gamma\alpha^-]:[\gamma]\in S\,\}$ of $\pi_1(X,x)$, where $\alpha$ is any path from $x$ to~$x_0$ in~$X$. Since $S$ is normal, this definition is independent of the choice of~$\alpha$;%
	\COMMENT{}
	in particular,%
	\COMMENT{}
	$S_{x}\,$is again normal. When $X$ is a graph~$G$, we call~$S$%
	\COMMENT{}
	\emph{\can} if for every vertex $x \in G$ and for every automorphism $\phi$ of~$G$ we have~$\phi_\ast(S_x) = S_{\phi(x)}$.%
	\COMMENT{}

When we view a graph~$G$ as a 1-complex, as we freely will, we use the term `path' for graph-theoretical paths in the sense of~\cite{DiestelBook25}, and say {\em topological path\/} for continuous maps $[0,1]\to G$. 
	
In addition to Hatcher's~\cite{Hatcher}*{Ch.1.3} we use standard covering space terminology as in J\"anich~\cite{Janich}*{Ch.9}.
We refer to covering projections $p \colon (C,\hat x_0) \to (X,x_0)$ as {\em coverings\/}, reserving the term {\em covering space\/}, or simply {\em cover\/}, for~$C$. All covering spaces we consider are connected.
	The group of deck transformations of a covering~$p$ is denoted by~$\cD(p)$. As generally with automorphisms, we write multiplication in~$\cD(p)$ as $\varphi\circ\psi =: \varphi\psi$.
A~normal covering is {\em canonical\/} if its characteristic subgroup is \can.%
	\COMMENT{}

\vskip 1.5 \smallskipamount\subsection{\boldmath Topological paths versus graph-theoretic walks}
	
	In this section we provide a graph-theoretical description of the fundamental group of a graph, following Stallings~\cite{stallings1983topology}, on which our terminology is based.
	
	A \emph{walk} (of \emph{length}~$k$) in a graph~$G$ is a non-empty alternating sequence $v_0 e_0 v_1 \ldots  v_{k-1} e_{k-1} v_k$ of vertices and edges in~$G$ such that $e_i$ has endvertices $v_i$ and~$v_{i+1}$ for all~$i<k$.
	A walk $W$ is \emph{closed} with \emph{base vertex}~$v_0$ if $v_0 = v_k$. We say that $W$ is \emph{based} at~$v_0$, or simply a closed walk {\em at}~$v_0$.
	A walk is \emph{trivial} if it has length~0.
	A~walk~$W'$ is a \emph{subwalk} of $W$ if $W' = v_i e_i v_{i+1} \ldots v_{j-1} e_{j-1} v_j$ for some indices $i,j$ with~$0 \leq i \leq j \leq k$. The inverse sequence of a walk~$W$ is its {\em reverse\/}, or $W$ {\em traversed backwards}; we denote this walk by~$W^-$.
	
	Let $G$ be any connected graph.
	A walk in~$G$ is \emph{reduced} if it contains no subwalk of the form $u e v e u$ where $u,v$ are vertices and $e$ is an edge.
	Note that every walk $W$ can be turned into a reduced walk~$W'$ based at the same vertex%
	\COMMENT{}
	by iteratively replacing subwalks of the form $u e v e u$ with the trivial walk~$u$.
	We call $W'$ the \emph{reduction} of~$W$ and remark that this is well-defined.
	We call two walks (\emph{combinatorially}) \emph{homotopic} if their reductions are equal.
	This defines an equivalence relation~${\sim}$ on the set of all walks in~$G$.
	
	Now let $x_0\in G$ be any vertex.
	Let $\cW(G,x_0)$ be the set of all closed walks in~$G$ that are based at~$x_0$.
	Then $\pi'_1(G,x_0):=(\cW(G,x_0)/{\sim},\,{\cdot}\,)$ is a group, where $[W_1]\cdot [W_2]:=[W_1 W_2]$ and $W_1 W_2$ is the concatenation of the two walks~$W_1$ and~$W_2$.
	Any walk in~$G$ defines (up to re-parametrisation) a topological path in the $1$-complex~$G$ which traverses the vertices and edges of the walk in the same order and direction.
	This defines a group isomorphism between $\pi'_1(G,x_0)$ and the fundamental group~$\pi_1(G,x_0)$ of the $1$-complex~$G$~\cite{stallings1983topology}.
	For notational simplicity we will work with combinatorial walks instead of topological paths, and do not always distinguish between walks and the homotopy classes they represent. For example, we may say that a closed walk lies in a subgroup $S$ of $\pi_1(G, x_0)$ when we mean that its combinatorial homotopy class maps to an element of~$S$ under the above group isomorphism.

\bigbreak\section{\boldmath Graph-decom\-pos\-itions} \label{sec:GraphDecomps}
	
\noindent
	In this section we introduce \gd s. We show how \gd s of a graph~$G$ can be obtained canonically from canonical \td s of coverings graphs of~$G$ (\cref{main:TreeDecToGraphDec}).
	
	\vskip 1.5 \smallskipamount\subsection{\boldmath Decompositions}
	Since \gd s generalise \td s, we recall the notion of \td s before we introduce the new notion of \gd s.
	
	\begin{definition}[\bf Tree-decom\-pos\-ition]
		Let $G$ be a graph, $T$ a tree, and $\cV=(V_t)_{t\in T}$ be a family of vertex sets $V_t\subset V(G)$ indexed by the nodes $t$ of~$T$. The pair $(T,\cV)$ is called a \emph{\td} of~$G$ if it satisfies the following two conditions:
		\begin{enumerate}[label=(T\arabic*)]
			\item\label{ax:TreeDecomp1} $G=\bigcup_{t\in T}G[V_t]$;
			\item\label{ax:TreeDecomp2} for every vertex~$v\in G$, the set of nodes~$t$ of~$T$ such that $v\in V_t$ is connected in~$T$.
		\end{enumerate}
		The sets~$V_t$ of vertices are the {\em bags\/} of this decomposition. We sometimes call the subgraphs~$G[V_t]$ they induce the decomposition's \emph{parts}, and $T$ its \emph{decomposition tree}.
	\end{definition}

	\noindent In this paper, all bags of \td s are required to be non-empty.
	If $G$ is connected, this implies that  for edges $tt'$ of~$T$ the sets $V_t \cap V_{t'}$ are non-empty too~\cite{DiestelBook25}*{Lemma 12.3.1}.
	
	A~precursor to our notion of \gd s was suggested by Diestel and K\"uhn~\cite{GMhierarchies}. But their notion is not quite the same as ours, and we shall not rely on~\cite{GMhierarchies}.%
	\COMMENT{}
	
	\begin{definition}[\bf Graph-decom\-pos\-ition]%
   \label{def:GrDec}
		Let $G$ and $H$ be graphs, and let $\cV=(V_h)_{h\in H}$ be a family of sets~$V_h$ of vertices of~$G$ indexed by the nodes $h$ of~$H$.
		The pair $(H,\cV)$ is called a \emph{\gd\ with model~$H$}, or \emph{$H$-decom\-pos\-ition}, of~$G$ if it satisfies the following two conditions:
		\begin{enumerate}[label=(H\arabic*)]
			\item $G=\bigcup_{h\in H}G[V_h]$; \label{ax:GraphDecomp1}
		\item for every vertex $v\in G$, the subgraph~$H[W_v]$ of~$H$ induced by $W_v := \{\,h\in H\mid v\in V_h\,\}$ is connected.\hskip-0pt
\label{ax:GraphDecomp2}
		\end{enumerate}
		The decomposition is {\em honest\/} if $V_h\ne\emptyset$ for all nodes~$h$ and $V_h\cap V_{h'}\ne\emptyset$ for all edges $hh'$ of~$H$.

The sets~$V_h$ in our decomposition are its {\em bags\/}; the sets~$W_v$ its {\em co-bags\/}. The maximum of its bag sizes~$|V_h|$, if it exists, is the {\em width\/} of our decomposition,%
   \footnote{When $H$ is a tree, this differs from the traditional width parameter for \td s by~1.}
   the maximum of its co-bag sizes~$|W_v|$ its  {\em co-width\/}, or {\em spread\/}. If all the bags are finite but their sizes are not necessarily bounded, we say that the decomposition has {\em finite width\/}, and similarly for spread.

Note the duality of\vskip-18pt\par
 $$v\in V_h\Leftrightarrow h\in W_v.$$
This enables us to reconstruct the families $\cV=(V_h)_{h\in H}$ and $\cW=(W_v)_{v\in G}$ from each other when $G$ and~$H$ are given.%
   \COMMENT{}
    We call them {\em dual\/} to each other.

We often flesh out the bags and co-bags by choosing subgraphs~$G_h$ of~$G$ on~$V_h$ and $H_v$ of~$H$ on~$W_v$ in such a way that \cref{ax:GraphDecomp1} and \cref{ax:GraphDecomp2} still hold, as
		\begin{enumerate}[label={\rm(H\arabic*$'$)}]
			\item $G=\bigcup_{h\in H}G_h$; \label{ax:GraphDecomp1parts}
		\item for every vertex $v\in G$, the subgraph~$H_v$ of~$H$ is connected.
\label{ax:GraphDecomp2coparts}
		\end{enumerate}

\noindent
These~$G_h$, then, are our chosen \emph{parts} of the decomposition~$(H,\cV)$; the~$H_v$ our chosen \emph{co-parts}. We may then refer to the decomposition with these choices simply by citing the families $\cG = (G_h)_{h\in H}$ and $\cH = (H_v)_{v\in G}$ together with $G$ and~$H$.%
   \COMMENT{}
   Note that $\cV$ can be reconstructed from either of these families, and that it satisfies \cref{ax:GraphDecomp1} and \cref{ax:GraphDecomp2} if $\cG$ and~$\cH$ satisfy \ref{ax:GraphDecomp1parts} and~\ref{ax:GraphDecomp2coparts}.
	\end{definition}

Honest \gd s into connected parts are closed under duality:

\ifArXiv

	\begin{lemma} \label{lem:DualGraphDecomp}
		If $(H,\cV)$ is an honest \gd\ of~$G$ that has%
   \COMMENT{}
   connected parts, then its dual $(G,\cW)$ is an honest \gd\ of~$H$ that has connected parts.
	\end{lemma}

	\begin{proof}
		The decomposition~$(G,\cW)$ satisfies~\cref{ax:GraphDecomp1}, because $(H,\cV)$ is honest.%
   \COMMENT{}
		It satisfies~\cref{ax:GraphDecomp2}, since $(H,\cV)$ has connected parts~$G_h$, so the $G[V_h]\supseteq G_h$ from \cref{ax:GraphDecomp2} for~$(G,\cW)$ are connected too. Thus, $(G,\cW)$ is a \gd. We can choose connected parts $H_v := H[W_v]$ for it by \cref{ax:GraphDecomp2} for~$(H,\cV)$.
		
		Let us show that $(G,\cW)$ is honest. Its bags~$W_v$ are non-empty by \cref{ax:GraphDecomp1} for $(H,\cV)$, which ensures that $v$ lies in some~$V_h$ so that $W_v\owns h$. Similarly, every edge $e = vv'$ of~$G$ lies in some~$G[V_h]$ by~\cref{ax:GraphDecomp1}, so that $h\in W_v \cap W_{v'}\ne\emptyset$.
	\end{proof}

\else

	\begin{lemma}[\cite{GraphDecArXiv}] \label{lem:DualGraphDecomp}
		If $(H,\cV)$ is an honest \gd\ of~$G$ that has connected parts, then its dual $(G,\cW)$ is an honest \gd\ of~$H$ that has connected parts.
	\end{lemma}

\fi
	
\noindent
Within the realm of honest \gd s into connected parts, dualising is clearly an involution.

\medbreak
	
	An \emph{isomorphism} between \gd s $(H,\cV)$ and $(H',\cV')$ of graphs $G$ and~$G'$, where $\cV = (V_h)_{h\in H}$ and $\cV' = (V'_{h'}{)}_{h'\in H'}$,%
   \COMMENT{}
   is a pair $(\phi, \psi)$ of isomorphisms $\phi\colon G\to G'$ and $\psi\colon H\to H'$ such that $\varphi(V_h)=V'_{\psi(h)}$ for every node $h\in H$.%
   \COMMENT{}
If we have chosen families $\cG = (G_h)_{h\in H}$ and $\cG' = (G'_{h'}{)}_{h'\in H'}$ of parts $G_h\subset G[V_h]$ and $G'_{h'}\subset G'[V'_{h'}]$ for these decompositions, we call our pair $(\phi, \psi)$%
   \COMMENT{}
   an \emph{isomorphism} between $(H,\cG)$ and $(H',\cG')$ if $\varphi(G_h)=G'_{\psi(h)}$ for every node $h\in H$.%
   \COMMENT{}

\vskip 1.5 \smallskipamount\subsection{\boldmath Properties of \gd s} \label{subsec:GraphDecProperties}
	
	A~\emph{separation} of a set $X$ is an unordered pair $\{A,B\}$ of subsets $A,B$ of $X$ such that $A \cup B = X$. The sets $A$ and $B$ are the \emph{sides} of this separation, their intersection $A\cap B$ is the associated {\em separator\/} of~$X$, and $|A \cap B|$ is its \emph{order}. The separation $\{A,B\}$ is {\em proper\/} if neither $A$ nor~$B$ equals~$X$. The \emph{orientations} of $\{A, B\}$ are the two ordered pairs $(A,B)$ and~$(B,A)$; these are \emph{oriented separations} of~$X$. We denote the two orientations of a separation~$s$ as $\vs$ and~$\sv$. There are no default orientations: once we have called one of the two orientations~$\vs$, the other will be~$\sv$, and conversely. A~\emph{separation} of a graph~$G$ is a separation $\{A,B\}$ of its vertex set~$V(G)$ such that $G$ has no edge between $A \setminus B$ and~$B \setminus A$.

	Let $\cT = (T, \cV)$ be a \td\ of a graph~$G$.
	Every orientation $(t,t')$ of an edge $e \in T$ \emph{induces} an oriented separation $\alpha_\cT(t,t')$ of $G$, as follows. 
	Let $T_t$ and $T_{t'}$ be the components of $T - e$ containing $t$ and~$t'$, respectively.%
   \COMMENT{}
	By \cref{ax:TreeDecomp1}, $\alpha_\cT(t,t') := (A_t, A_{t'})$ is an oriented separation of $G$ with sides $A_t = \bigcup_{s \in T_t} V_s$ and~$A_{t'} = \bigcup_{s \in T_{t'}} V_{s}$.
	Similarly, $\alpha_\cT(e):=\{A_t, A_{t'}\}$ is an unoriented separation of~$G$.%
   \COMMENT{}
	By~\cref{ax:TreeDecomp2}, the separator $A_t\cap A_{t'}$ of $\{A_t, A_{t'}\}$ equals~$V_t \cap V_{t'}$~\cite{DiestelBook25}*{Lemma 12.3.1}; this is the \emph{adhesion set} of~$(T, \cV)$ at the edge~$e\in T$.
	A~\td\ has \emph{finitely bounded adhesion} if there exists $K \in \N$ such that all adhesion sets of $(T, \cV)$ have order at most~$K$.
	
Notice that, in the above example, the vertex sets of the components $T_t$ and $T_{t'}$ of~$T-e$ formed a separation of the set~$V(T)$. It is not a separation of the graph~$T$, since the edge~$e$ lies in neither componenent, but it induces a separation of the graph~$G$. In general, given a \gd\ of~$G$ modelled on a graph~$H$, separations of the set~$V(H)$ induce separations of the graph~$G$ in the same way:
	
	\begin{lemma} \label{lem:graphdecompseparator}
		Let $(H,(V_h)_{h \in H})$ be a \gd\ of a graph~$G$.
		Every separation $\{U, W\}$ of the set~$V(H)$ induces a separation $\{A,B\}$ of the graph~$G$ with $A:=\bigcup_{h \in U} V_h$ and~$B:= \bigcup_{h \in W} V_h$.
		The separator of $\{A,B\}$ is equal to 
		$S := \large(\bigcup_{h \in U \cap W} V_h \large)\cup \bigcup_{hh' \in F} (V_h \cap V_{h'}) $,
		where~$F = E_H(U \setminus W, W \setminus U)$.
	\end{lemma}
	
	\begin{proof}
		By \cref{ax:GraphDecomp1}, every edge $e$ of $G$ is contained in some part~$G[V_h]$. 
		Then $h\in U$ or $h\in W$, and $e\in G[A]$ or $e\in G[B]$ respectively.
		In particular, there is no edge in~$G$ between $A \setminus B$ and~$B \setminus A$.
		Thus, $\{A, B\}$ is a separation of~$G$.
		
		It is immediate from the definition of $A$ and $B$ that~$A \cap B \supseteq S$.
		To show $A \cap B \subseteq S$, let $v \in A \cap B$ be any vertex.
		By definition of $A$ and~$B$ there exist $h\in U$ with $v\in V_h$ and $h'\in W$ with $v\in V_{h'}$. Thus, $h,h'\in W_v$. Since $H[W_v]$~is connected~\cref{ax:GraphDecomp2}, it contains an $h$--$h'$ path from $U$ to~$W$; choose $h$ and~$h'$ so that this path has minimum length. Since $\{U,W\}$ is a separation of~$V(H)$, this means that either $h=h'\in U\cap W$ or $hh'\in F$. In either case, $v\in V_h\cap V_{h'}\subset S$ as desired.
	\end{proof}

	Recall that an \emph{end} of a graph~$G$ is an equivalence class of rays in~$G$, where two rays in~$G$ are \emph{equivalent} if for every finite set~$X$ of vertices of~$G$ the two rays have subrays in the same component of~$G-X$.
	As a first reflection of the relationship between the separations of $H$ and~$G$ observed in \cref{lem:graphdecompseparator}, let us show that the ends of a graph~$G$ often correspond naturally%
   \COMMENT{}
   to those of its models~$H$.

	To define this relationship, we use a tangle-like description of ends known as directions.
	A \emph{direction} in a graph~$G$ is a map~$f$, with domain the set of all finite vertex sets of~$G$, that assigns to every finite vertex set $X\subset V(G)$ a component $f(X)$ of~$G-X$ so that $f(X)\supset f(X')$ whenever~$X\subset X'$.
	Every end $\omega$ of~$G$ defines a direction $f_\omega$ in~$G$ by letting $f_\omega(X)$ be the component of $G-X$ that contains a subray of one (equivalently:~every) ray in~$\omega$.
	The map $\omega\mapsto f_\omega$ is a bijection between the ends of~$G$ and its directions~\cite{Ends}*{Theorem 2.2}.
	Therefore, in order to define a bijection between the ends of $H$ and~$G$ for an $H$-decom\-pos\-ition of~$G$, it suffices to define a bijection between the directions of~$H$ and of~$G$.
	
	Let $G$ be any graph, and let $(H,(V_h)_{h\in H})$ be a \gd\ of~$G$ of finite width. Every direction $f$ in~$G$ defines a direction $g_f$ in~$H$, as follows.
	For every finite set $Y\subset V(H)$, the set $X_Y:=\bigcup_{h\in Y}V_h$ of vertices in~$G$ is finite.
	Every $W_v$ with $v \in f(X_Y)$ lies entirely outside~$Y\!$, and hence by \cref{ax:GraphDecomp2} lies in one component of~$H-Y$.%
    \COMMENT{}
   This is the same component of~$H-Y$ for all these~$v$. Indeed, for adjacent $v,v'\in f(X_Y)$ consider any $h\in V(H)$ such that $v,v'\in V_h$, as given by~\cref{ax:GraphDecomp1}. Then $h\in W_v\cap W_{v'}$. In particular, $W_v\cap W_{v'}\ne\emptyset$, so the components of~$H-Y$ that contain~$W_v$ and~$W_{v'}$, respectively, are identical. As $f(X_Y)$ is connected, this implies that all of $\bigcup_{v \in f(X_Y)} W_v$ lies in one component of~$H-Y$;%
    \COMMENT{}
    let $g_f(Y)$ be this component. It is straightforward to verify that $g_f$ is a direction in~$H$.
		
	\begin{lemma}\label{lem:GraphDecDisplayingEnds}
   		For every honest $H$-decom\-pos\-ition of~$G$ of finite width and spread, and with connected parts, the map $f\mapsto g_f$ between the directions of~$G$ and those of~$H$ is bijective.
	\end{lemma}
	
	\noindent The bijection in \cref{lem:GraphDecDisplayingEnds} between the directions of $G$ and~$H$ induces a bijection between their ends, by~~\cite{Ends}*{Theorem 2.2} as mentioned earlier. That bijection is in fact a homeomorphism between the end spaces of $G$ and~$H$; see~\cite{DiestelBook25} for definitions.
	
	\begin{proof}[Proof of \cref{lem:GraphDecDisplayingEnds}]
Let the given decomposition of~$G$ be denoted as~$(H,\cV)$, where $\cV = (V_h)_{h\in H}$ as usual, and denote its (connected) parts by~$G_h$.	
We establish an inverse to the map $f \mapsto g_f$ by constructing a map $g\mapsto f_g$ from the directions of~$H$ to those of~$G$ such that $f_{g_f} = f$ for all~$f$%
		\COMMENT{}
		and $g_{f_g} = g$ for all~$g$.
		
		Let a direction $g$ of~$H$ be given. 
		To define~$f_g$, let a finite set $X\subset V(G)$ be given. Let $Y_X$ be the set of nodes $h \in H$ whose bag~$V_h$ contains some vertex of~$X$.
		Note that $Y_X$ is finite, since all the~$W_v$ are finite.
		The subgraph $\bigcup_{h \in g(Y_X)} G_h$ of~$G-X$ is connected, because $g(Y_X)\subset H-Y$ is connected, the parts $G_h\subset G$ are connected, and $(H, \cV)$ is honest.
		Let $f_g(X)$ be the unique component of $G-X$ that contains~$\bigcup_{h \in g(Y_X)} G_h$.

		To check that $f_g$ is a direction, let $X \subseteq X'$ be two finite sets of vertices of~$G$.
		Then~$Y_X \subseteq Y_{X'}$.
		Since $g$ is a direction, we have $g(Y_X) \supseteq g(Y_{X'})$.
		Hence $f_g(X)\supset \bigcup_{h \in g(Y_X)} G_h \supseteq \bigcup_{h \in g(Y_{X'})} G_h$.
		Now since $X \subseteq X'$, every component of $G - X'$ is contained in a unique component of $G - X$.
		As $f_g(X')$ and $f_g(X)$ both contain $\bigcup_{h \in g(Y_{X'})} G_h\ne\emptyset$,%
		\COMMENT{}
		the unique component of $G - X$ containing $f_g(X')$ is~$f_g(X)$.
		
For a proof that $f = f_{g_f}$ for every direction $f$ of~$G$, we show that for every finite set~$X \subseteq V(G)$ both $f(X)$ and~$f_{g_f}(X)$ contain~$f(X_{Y_X})\ne\emptyset$, and must therefore coincide.%
		\COMMENT{}
	As $X\subset X_{Y_X}$ we have ${f(X)\supseteq f(X_{Y_X})}$, since $f$ is a direction. For a proof of $f_{g_f}(X)\supseteq f(X_{Y_X})$ let $v\in f(X_{Y_X})$ be given. By~\cref{ax:GraphDecomp1} there is an~$h\in H$ such that $v\in V_h$.%
		\COMMENT{}
		Then $h\in W_v\subset g_f(Y_X)$ by definition of~$g_f$ and the choice of~$v$,%
		\COMMENT{}
		so $G_h\subset f_{g_f}(X)$ by definition of~$f_{g_f}$.%
		\COMMENT{}
		In particular, $v\in G_h\subset f_{g_f}(X)$ as desired.
		
To show $g = g_{f_g}$ for every direction $g$ of $H$, we show that for every finite set $Y$ of nodes of $H$ both $g(Y)$ and $g_{f_g}(Y)$ contain $g(Y_{X_Y}) \neq \emptyset$, and must therefore coincide.%
		\COMMENT{}
		Since $Y \subseteq Y_{X_Y}$, we have $g(Y) \supseteq g(Y_{X_Y})$ because $g$ is a direction of $H$.%
		\COMMENT{}
		For a proof of $g_{f_g}(Y) \supseteq g(Y_{X_Y})$, let $h \in g(Y_{X_Y})$ be given. Then $V_h \neq \emptyset$ since $(H, \cV)$ is honest;%
   \COMMENT{}
   pick $v \in V_h$.%
		\COMMENT{}
		Now $v \in V_h \subseteq f_g(X_Y)$ by definition of $f_g$ and the choice of~$h$,%
		\COMMENT{}
		so $h \in W_v \subseteq g_{f_g}(Y)$ by the definition of~$g_{f_g}$, as desired.%
		\COMMENT{}	\end{proof}

Next, let us show that not only do the ends of $H$ and~$G$ correspond bijectively, under the assumptions from \cref{lem:GraphDecDisplayingEnds}, but the sizes of the separators needed to distinguish these ends correspond too.

A finite-order separation $\{A,B\}$ of $G$ \emph{distinguishes} two ends $\omega_1,\omega_2$ of~$G$ if for some (and hence every) ray in~$\omega_1$ some subray lies in~$G[A]$ and for some (and hence every) ray in~$\omega_2$ some subray lies in~$G[B]$.
	Following Thomassen and Woess~\cites{ThomassenWoess}, let us call a graph \emph{\accessible} if there exists an integer $K \in \N$ such that every two of its ends are distinguished by a separation of order at most~$K$. Compare~\cref{subsec:windfall} for the origin of this notion.

\begin{lemma} \label{lem:AccessibleDecomp}
	Let $G$ be a graph with an honest $H$-decom\-pos\-ition of finitely bounded width, finite spread, and with connected parts. If $H$ is \accessible, so is~$G$.
\end{lemma}

\begin{proof}
	Let $(H,(V_h)_{h\in H})$ be the decomposition of~$G$ assumed to exist, let $n\in\N$ be an upper bound on the size of its parts, and let $k\in\N$ be such that every two ends of~$H$ are distinguished by a separation of order at most~$k$. We show that every two ends of~$G$ are distinguished by a separation of order at most~$nk$.%
   \COMMENT{}

Let $\omega_1$ and $\omega_2$ be two distinct ends of $G$, and let $f_1 := f_{\omega_1}$ and $f_2 := f_{\omega_2}$ be the two directions of~$G$ they define.
		Consider the directions $g_1 := g_{f_1}$ and $g_2 := g_{f_2}$ of $H$ as defined just before \cref{lem:GraphDecDisplayingEnds}.
		By~\cite{Ends}*{Theorem 2.2}, there exist ends $\eta_1$ and $\eta_2$ of $H$ which induce these directions $g_1$ and~$g_2$.
		Since $H$ is \accessible, it has a separation $\{U_1, U_2\}$ of order at most~$k$ such that, for $i=1,2$, the end~$\eta_i$ is represented by a ray in~$H[U_i]$, and thus $g_i(U_1\cap U_2)\subset H[U_i]$.

Let us apply \cref{lem:graphdecompseparator} to $\{U_1, U_2\}$ to obtain a separation $\{A_1, A_2 \}$ of $G$ with $A_i = \bigcup_{h \in U_i} V_h$. Note that $A_1 \cap A_2 = \bigcup_{h \in U_1 \cap U_2} V_h$: as $\{U_1,U_2\}$ is a separation not just of~$V(H)$ but of the graph~$H$, the set~$F$ in \cref{lem:graphdecompseparator} is empty.%
	\COMMENT{}
	Hence $\{A_1, A_2\}$ has order at most~$nk$; it remains to show that it distinguishes $\omega_1$ and~$\omega_2$.

By definition of~$f_i$, each~$\omega_i$ has a ray in~$f_i(A_1 \cap A_2)$. So it suffices to show that $V(f_i(A_1 \cap A_2))\subset A_i$, for $i=1,2$. For every $v\in G$ that is not in $A_1\cap A_2 = \bigcup_{h \in U_1 \cap U_2} V_h$, the graph $H[W_v]$ avoids~$U_1\cap U_2$, and hence by~\cref{ax:GraphDecomp2} lies in a component of~$H - (U_1\cap U_2) \ne \emptyset$.%
	\COMMENT{}
	For $v \in f_i(A_1 \cap A_2)$ this component is $g_i(U_1\cap U_2)\subset H[U_i]$, by definition of $g_{f_i} = g_i$. Pick $h\in W_v$ for any such~$v$. Then $v\in V_h$ and $h\in U_i$. This implies $v\in A_i$, as desired, by definition of~$A_i$.
	\end{proof}

\subsection{\boldmath From \td s to \gd s}\label{subsec:treedecomptographdecomp}
	
	In the remainder of this section we show how \gd s of a graph can be obtained from \td s of its normal covers, and how all these can be chosen canonically.
	The main result in this section will be \cref{main:TreeDecToGraphDec}.
	
	A \td~$\cT = (T,\cV)$ of a connected graph $G$ is \emph{regular} if all the separations of~$G$ induced by the edges of~$T$ are proper: if none has the form $\{A,V(G)\}$. In particular, then, $\cT$~cannot have empty adhesion sets and hence no empty bags,%
   \COMMENT{}
    so regular \td s of connected graphs are honest.%
	\COMMENT{}
	 All \td s relevant to us will be regular.

We shall be interested particularly in \td s~$\cT$ on which the automorphisms~$\phi$ of the graph being decomposed act naturally, in that they extend to automorphisms $(\phi,\psi)$ of~$\cT$. These automorphisms~$\psi$ are easily seen to be unique when $\cT$ is \proper:

\begin{lemma}\ifArXiv\else\cite{GraphDecArXiv}\fi\label{lem:atmostoneautom}
	If $\cT = (T,\cV)$ is a \proper\ \td\ of a graph~$G$, then for every automorphism $\phi$ of~$G$ there exists at most one automorphism $\psi$ of~$T$ such that $(\phi,\psi)$ is an automorphism of~$\cT$.\ifArXiv\else\qed\fi
	\end{lemma}%
\ifArXiv
\begin{proof}
Consider automorphisms $\psi,\psi'$ of~$T$ such that both $(\phi,\psi)$ and $(\phi,\psi')$ are automorphisms of~$\cT$. Then so is $(\id,\psi'\circ\psi^{-1})$. If $\psi'\circ\psi^{-1} = \id$, then $\psi'=\psi$. Thus, it is left to show that the only automorphism $\psi$ of $T$ for which $(\id,\psi)$ is an automorphism of $\cT$ is $\psi = \id$.
\endgraf
Let us show first that it is enough to show that for every oriented separation $(A,B)$ induced by $\cT$, the graph $T_{(A,B)}$ which consists of all oriented edges of $T$ that induce $(A,B)$ is a finite directed path.%
\COMMENT{}
Assuming this, let~$(\id, \psi)$ be any automorphism of $\cT$. Let $(t,t')$ be an oriented edge of $T$, and let $(A,B)$ be the separation induced by $(t,t')$. Since $(\id, \psi)$ is an automorphism of $\cT$, the automorphism $\psi$ of $T$ maps every oriented edge of $T$ to an oriented edge of $T$ that induces the same separation. Thus, $\psi$ restricts to an automorphism of $T_{(A,B)}$. As the only automorphism of a finite directed path is the identity map, the restriction of $\psi$ to $T_{(A,B)}$ is the identity map. Hence, $(\psi(t), \psi(t')) = (t,t')$. Since the edge oriented $(t,t')$ of $T$ was chosen arbitrarily, $\psi$ is the identity map on the whole decomposition tree $T$.
\endgraf
It remains to show, then, that $T_{(A,B)}$ is a finite directed path, for every separation $(A,B)$ induced by $\cT$. Let~$(A,B)$ be a separation induced by $\cT$. Since $\cT$ is regular, $T_{(A,B)}$ does not contain $(t,t')$ and $(t',t)$ for any edge $tt'$ of~$T$; otherwise, $(A,B)=(B,A)$, and thus $(A,B)=(V(G),V(G))$ - contradicting the fact that $\cT$ is regular. Let $(t_0,t_1)$ and $(t_2,t_3)$ be two oriented edges of $T$ which both induce~$(A,B)$, then either~$(t_0,t_1) \leq (t_2,t_3)$ or~$(t_2,t_3) \leq (t_0,t_1)$, i.e., the unique $\{t_0,t_1\}$--$\{t_2,t_3\}$ path in $T$ is $t_1Tt_2$ or $t_3Tt_0$, respectively. Indeed, if not it follows from the compatibility of the partial order on the oriented edges of $\cT$ and its induced separations%
	\COMMENT{}
	that $(A,B) \leq (B,A)$ or $(B,A) \leq (A,B)$; in particular, $A = V(G)$ or $B = V(G)$~-- a contradiction to $\cT$ being regular. In both cases this compatibility also shows that all oriented edges of $T$ between these two edges induce $(A,B)$ as well.%
\COMMENT{}
This shows that the underlying graph of $T_{(A,B)}$ is connected and also that at each node $t$ of $T_{(A,B)}$ there is at most one outgoing edge which induces $(A,B)$ and at most one outgoing edge which induces $(B,A)$. Using~\cref{ax:TreeDecomp1}, it is not hard to show that for every ray $t_0t_1t_2\dots$ in $T$ we have $V(G) = \bigcup_{i \in \N} A_i$, where $(A_i,B_i)$ is the separation induced by $(t_i,t_{i+1})$.%
   \COMMENT{}
   Since $\cT$ is regular, $A,B \neq V(G)$, and thus the underlying graph of $T_{(A,B)}$ does not contain a ray. Hence, $T_{(A,B)}$ is finite.%
\COMMENT{}
So, $T_{(A,B)}$ is a finite directed path.
\end{proof}

We believe that \cref{lem:atmostoneautom} extends to $H$-decompositions for graphs~$H$ other than trees, with an appropriate adaptation of the regularity requirement. For example, the extended version is easy to prove when $H$ is 2-connected and its bonds induce proper separations of~$G$, since this implies that $V_h\ne V_{h'}$ for distinct nodes $h,h'\in H$.%
   \COMMENT{}
   \medbreak
	\fi

	Let $\Gamma$ be a group of automorphisms of~$G$. A \gd\ $(H,\cV)$ or~$(H,\cG)$ of~$G$ is \emph{$\Gamma$-canonical} if $\Gamma$ acts on~$(H,\cV)$ or~$(H,\cG)$, respectively, via $\phi\mapsto (\phi, \psi)$%
	\COMMENT{}
	for suitable automorphisms $\psi =: \psi_\phi$ of~$H$.%
	\COMMENT{}
	Once the action of~$\Gamma$ on~$(H,\cV)$ or~$(H,\cG)$ is fixed, or if it is unique by \cref{lem:atmostoneautom}, we usually write~$\phi$ instead of~$\psi_\phi$ to reduce clutter.
If $(H,\cV)$ or~$(H,\cG)$ is $\Aut(G)$-canonical, we just call it a \emph{canonical} \gd\ of~$G$.
	When $(H,\cV)$ is a \td, this definition formalises the notion of canonicity used informally for \td s in the literature~\cite{confing,CDHH13CanonicalAlg}.%
	\COMMENT{}

\medbreak\penalty-200
	
	The \gd s given by the following construction will be the key objects in \cref{main:TreeDecToGraphDec}.
	
\begin{construction}[Graph-decom\-pos\-itions defined by \td s via coverings] \label{def:GraphDecompViaCovering}
		\hfill\break
   Let $p\colon C\to G$ be any normal covering of a connected graph~$G$. Let $\cD\subset \Aut(C)$ be the group of deck transformations, and let $\cT=(T,\cV_T)$ be any regular \dcan\ \td\ of~$C$.
	By \cref{lem:atmostoneautom},%
   \COMMENT{}
   every $\phi\in\cD$ acts uniquely on~$\cT$, and in particular on~$T$, so that $\phi(C[V_t])=C[V_{\phiT(t)}]$%
	\COMMENT{}
   for all $t\in T$.%
	\COMMENT{}
		The map $t\mapsto p(C[V_t])$ is constant on the orbits in~$V(T)$ under the action of~$\cD$ on~$T$, since for every $\phi \in \cD$ we have 
		\begin{equation}\label{eq:constant}
			p(C[V_t]) = p(\phi(C[\,V_t\,])) = p(C[V_{\phiT(t)}]).
		\end{equation}
		Let $H$ be the orbit graph%
  \COMMENT{}
   $T/\cD$ of the action of $\cD$ on~$T$.
		Let $\cV_G := (V_h)_{h\in H}$, where $V_h = p(V_t)$ for any~$t\in h$. These~$V_h$ are well-defined%
	\COMMENT{}
	by~\cref{eq:constant}. We shall prove in a moment that $(H,\cV_G)$ is a \gd\ of~$G$.

Let $\cG:=(G_h)_{h\in H}$, where $G_h = p(C[V_t])$ for any~$t\in h$. These~$G_h$ are well-defined%
	\COMMENT{}
	by~\cref{eq:constant}; we choose them as the parts of~$(H,\cV_G)$. Note that, while the parts~$C[V_t]$ of~$(T,\cV_T)$ were induced subgraphs of~$C$, their images~$G_h$ under~$p$ need not be induced in~$G$.

Let $\cH = (H_v)_{v\in G}$ be the family of subgraphs $H_v$ of~$H$ whose nodes and edges are the $\cD$-orbits in~$T$ of the nodes and edges of~$T_{\hat v} := T[\{\,t\in T\mid \hat v\in V_t\}]$,%
   \COMMENT{}
   where $\hat v$ is any vertex of~$C$ in~$p^{-1}(v)$. These $H_v$ are independent of the choice of~$\hat v$ since, as $p$ is normal, $\cD$~acts transitively on the fibres of~$p$.%
   \COMMENT{}
   Note that these~$H_v$ are spanning subgraphs of~$H[\{\,h\mid v\in V_h\}]$; we may thus choose these~$H_v$ as the co-parts of~$(H,\cV_G)$. Once more, the~$H_v$ need not be induced subgraphs of~$H$, although the co-parts~$T_{\hat v}$ of~$(T,\cV_T)$ are induced in~$T$.%
   \COMMENT{}

We say that $\cT$ \emph{defines, via~$p$,} the \gd\ $(H,\cV_G)$ of~$G$ with parts~$G_h$ and co-parts~$H_v$.%
   \COMMENT{}
	\end{construction}

\cref{def:GraphDecompViaCovering} can also be applied to \gd s, rather than \td s, of the normal cover~$C$. But we shall not need that level of generality here.
	
\begin{lemma}\label{lem:tdcDefinesHdecomp}
   Every $(H,\cV_G)$ constructed as in \cref{def:GraphDecompViaCovering} is an honest \gd\ of~$G$. Its parts~$G_h$ and co-parts~$H_v$ satisfy \cref{ax:GraphDecomp1parts} and~\cref{ax:GraphDecomp2coparts}.%
   \COMMENT{}
	\end{lemma}
	
\begin{proof}
With its induced parts~$C[V_t]$ and co-parts $T_{\hat v}$, the given \td\ $(T,\cV_T)$ of~$C$ satisfies not only \cref{ax:GraphDecomp1} and~\cref{ax:GraphDecomp2} but even \cref{ax:GraphDecomp1parts} and~\cref{ax:GraphDecomp2coparts}. It is straightforward to check that $(H,\cV_G)$ inherits \cref{ax:GraphDecomp1parts} and~\cref{ax:GraphDecomp2coparts} from~$(T,\cV_T)$ for our choice of parts~$G_h$ and co-parts~$H_v$.
As $(T,\cV_T)$ is regular, and hence honest, $(H,\cV_G)$~is honest too.%
   \COMMENT{}
	\end{proof}

\ifArXiv
	The \gd s $(H, \cG)$ from \cref{def:GraphDecompViaCovering} also satisfy an analogue of \cref{ax:GraphDecomp2} for edges:
	\begin{enumerate}[label=(H\arabic*)]
		\setcounter{enumi}{2}
		\item \label{ax:GraphDecomp3} For every edge~$e\in G$, the graph $H[\{\,h\in H\mid e\in G_h\,\}]$ is connected.
	\end{enumerate}%
   \COMMENT{}
	
	\noindent 
Note that \cref{ax:GraphDecomp3} refers to a given choice of parts~$G_h$. This is deliberate, and we are always free to choose these as~$G[V_h]$. But in the context of \cref{def:GraphDecompViaCovering} that would be unnatural when $G_h$ is not an induced subgraph of~$G$, because $G$ has `external' edges on~$V_h$ that are not $p$-images of edges of some~$C[T_t]$ with $t\in h$.\looseness=-1

The decompositions $(H, \cG)$ from \cref{def:GraphDecompViaCovering} also satisfy the analogue

	\begin{enumerate}[label={\rm(H\arabic*$'$)}]
		\setcounter{enumi}{2}
		\item \label{ax:GraphDecomp3edgecobags} For every edge~$e\in G$, the subgraph~$H_e$ of~$H$ is connected
	\end{enumerate}

\noindent
   of \cref{ax:GraphDecomp2coparts} for edges if we define the~$H_e$ in the natural way: as the subgraphs of~$H$ whose nodes and edges are the $\cD$-orbits in~$T$ of the subtree $T_{\hat u}\cap T_{\hat v}$ of~$T$, where $\hat e = \hat u\hat v$ is any lift of~$e=uv$ under~$p$.
The proofs of \cref{ax:GraphDecomp3} with the~$G_h$ from \cref{def:GraphDecompViaCovering}, and of~\cref{ax:GraphDecomp3edgecobags} with the~$H_e$ chosen as above, are analogous to the proofs for vertices: they reduce these properties to the analogous properties of the \td~$(T,\cV_T)$ of~$C$.\looseness=-1

Condition~\cref{ax:GraphDecomp3} is a familiar property of \td s. It hinges on the fact that any non-empty intersection of two subtrees of a (decomposition) tree is connected. Connected subgraphs of other graphs~$H$ can have disconnected intersections, though, which is why \cref{ax:GraphDecomp3} can fail in \gd s not induced by \cref{def:GraphDecompViaCovering}:
	\begin{example*}
		Let $G:=K^2$ with vertices~$u,v$ and let~$H:=C_4=abcda$. Let $G_a:=G_c:=G$, let $G_b:=(\{u\},\emptyset)$ and let~$G_d:=(\{v\},\emptyset)$.
		Then $(H,\cG)$ satisfies \cref{ax:GraphDecomp1parts} and \cref{ax:GraphDecomp2}. However, the set $\{\,h\in H\mid uv\in G_h\}=\{a,c\}$ is disconnected in~$H$.\qed
	\end{example*}
	
	\else
The \gd s $(H, \cG)$ from \cref{def:GraphDecompViaCovering} also satisfy the analogues of \cref{ax:GraphDecomp2} and~\cref{ax:GraphDecomp2coparts} for edges~$e$ of~$G$, instead of vertices~$v$. While this is familiar from \td s, it is not the case for \gd s in general. See~\cite{GraphDecArXiv} for more. \medbreak
	\fi

	The graph~$H$ obtained in \cref{def:GraphDecompViaCovering} can in general be infinite even when $G$ is finite.%
	\COMMENT{}
	However this will not happen if $\cT$ has finite spread (which will always be the case):

	
\begin{lemma}\ifArXiv\else\cite{GraphDecArXiv}\fi \label{lem:FiniteDecGraphPointFinite}
	The co-parts~$H_v$%
   \COMMENT{}
   of the \gd\ $(H, \cV_G)$ constructed in \cref{def:GraphDecompViaCovering} are finite if the \td~$\cT\!$ defining it has finite spread. If both the~$H_v$ and $G$ are finite, then $H$ is finite.\looseness=-1\ifArXiv\else\qed\fi
	\end{lemma}%
   \COMMENT{}
	\ifArXiv\begin{proof} The first assertion is immediate from the definition of~$H_v$; note that $H_v$ has only finitely many edges as well as nodes, since all its edges are orbits of edges of some finite tree~$T_{\hat v}\subset T$.%
   \COMMENT{}
	\endgraf
		As~$\cT$, being regular by assumption, is honest, we have~$T \subseteq \bigcup_{\hat{v} \in V(C)} T_{\hat v}$.
		Since $H = T/\cD$, this implies~$H \subseteq \bigcup_{v \in V(G)} H_v$.
		Hence  if both the~$H_v$ and~$G$ are finite, then so is~$H$.
	\end{proof}
	\fi
	
	For our proof of \cref{main:KeyTheorem} we want the \gd s~$(H,\cG)$ of~$G$ obtained via \cref{def:GraphDecompViaCovering} to be canonical. If all the automorphisms of~$G$ are induced via~$p$ by automorphisms of~$C$ then, as we shall see in \cref{lem:canonicity}, the canonicity of the \td~$\cT$ in \cref{def:GraphDecompViaCovering} will ensure this.%
	\COMMENT{}
	In \cref{lem:liftautomorphism} we provide a sufficient condition on~$p$ to ensure this,%
   \COMMENT{}
   that all the automorphisms of~$G$ `lift' to~$C$ in this way.
	
	In general, given a covering $p \colon C \to G$ of a connected graph~$G$, a \emph{lift} of an automorphism~$\phi$ of~$G$ is an automorphism $\hat{\phi}$ of~$C$ such that $p \circ \hat{\phi} = \phi \circ p$.
	Lifts of automorphisms are unique up to composition with deck transformations, since for any two lifts $\hat{\varphi},\hat{\varphi}'$ of~$\varphi$ their composition $\hat{\varphi}'\circ \hat{\varphi}^{-1}$ is a deck transformation.
	
\cref{lem:liftautomorphism} below says that the automorphisms of~$G$ all lift to~$C$ if the covering $p\colon C\to G$ is canonical (see \cref{subsec:coverings}):
	
	\begin{lemma}[\cite{djokovic1974automorphisms}*{Theorem 1}]\label{lem:liftautomorphism}
		Let $G$ be a connected graph, and $p \colon C \to G$ a \can\ normal covering.
		Then every automorphism $\phi$ of $G$ lifts to an automorphism of~$C$.%
	\COMMENT{}
	\end{lemma}

	
	\begin{lemma}\label{lem:canonicity}
		The \gd s $(H,\cG)$ from \cref{def:GraphDecompViaCovering} are canonical if $\cT\!$ and $p$ are canonical.%
	\COMMENT{}
	\end{lemma}
	
\begin{proof}
To prove $(H,\cG)$ to be canonical, we have to find for every automorphism $\phi$ of~$G$ an automorphism $(\phi,\psi)$ of~$(H,\cG)$ such that $\phi\mapsto (\phi,\psi)$ is a group homomorphism.%
\COMMENT{}

Let $\phi$ be an automorphism of $G$. As $p$ is canonical, \cref{lem:liftautomorphism} tells us that $\phi$~lifts to an automorphism~$\hat{\phi}$ of~$C$.%
	\COMMENT{}
	Since $\cT$ is canonical, $\hat\phi$~acts on the decomposition tree~$T$, uniquely by \cref{lem:atmostoneautom},%
	\COMMENT{}
	as of course does~$\cD$.%
	\COMMENT{}
	Since $\hat\phi$ is a lift of~$\phi$, its action on~$T$ is well defined on the orbits of~$\cD$ in~$T$,%
	\COMMENT{}
	i.e., $\psi\colon \cD\act t \mapsto \cD\act \hat{\phiT}(t)$ is a well-defined automorphism of~$H = T/\cD$.%
	\COMMENT{}
	The definition of $\psi$ is also independent of the choice of~$\hat\phi$, given~$\phi$, since $\hat\phi$ is unique up to composition with deck transformations.%
	\COMMENT{}

	\ifArXiv
It remains to check that our choice of~$\psi$ makes $\phi \mapsto (\phi,\psi)$ into a group homomorphism from the automorphisms of~$G$ to those of~$(H,\cG)$.
Let $\phi_0,\phi_1$ be automorphisms of~$G$, with lifts ${\hat\phi_0}$ and~${\hat\phi_1}$. Then ${\hat\phi_0}\circ {\hat\phi_1}$ is a lift of~$\phi_0\circ \phi_1$. So it follows directly from the definition of $\psi$ that the automorphism induced on the orbit graph $H = T/\cD$ by the action of ${\hat\phi_0}\circ {\hat\phi_1}$ on~$T$ is $\psi_0\circ \psi_1$, where $\psi_0$ and $\psi_1$ are the automorphisms induced by the action of ${\hat\phi_0}$ and ${\hat\phi_1}$, respectively. So, $\phi\mapsto (\phi,\psi)$ is a group homomorphism, and thus $(H,\cG)$ is canonical.
	\else
It is routine to check~\cite{GraphDecArXiv} that our choice of~$\psi$ makes $\phi \mapsto (\phi,\psi)$ into a group homomorphism from the automorphisms of~$G$ to those of~$(H,\cG)$.
	\fi
	\end{proof}

	Our main result of this section now follows directly from  \cref{lem:tdcDefinesHdecomp,lem:FiniteDecGraphPointFinite,lem:canonicity}:
	
	\begin{theorem} \label{main:TreeDecToGraphDec}
		Let $G$ be any connected graph, and let $p\colon C \to G$ be a normal covering of~$G$.
		Then every regular \dpcan\ \td~$\cT$ of~$C$ defines an honest \gd\ $(H,\cG)$ of~$G$ via \cref{def:GraphDecompViaCovering}.
		
		If $\cT\!$ and~$p$ are canonical, then $(H,\cG)$ is canonical.
		If $\cT\!$ has finite spread, then so does~$(H,\cG)$. If $(H,\cG)$ has finite spread and $G$ is finite, then $H$ is finite.
		\qed
	\end{theorem}

\ifArXiv

	\vskip 1.5 \smallskipamount\subsection{\boldmath Examples} \label{subsec:LocalCoveringsExamples}
	In this section we illustrate \cref{main:TreeDecToGraphDec} by a number of examples highlighting various aspects of \cref{def:GraphDecompViaCovering}.%
   \COMMENT{}
	All the coverings considered in this section are local coverings as defined in~\cref{sec:LocalCoverings}; formal proofs can be found in~\cite{FundGroupPlane}.
		
		\begin{figure}[ht]
			\center
			\includegraphics[height=11\baselineskip]{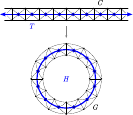}
			\caption{The global structure of $G$ is displayed by a cycle. Its local parts are $K^4$s.}
			\label{fig:Donut}
		\end{figure}
		
		\begin{example}
			Let~$G$ be the graph depicted in \cref{fig:Donut}. It is formed from twelve~$K^4$s by glueing them along pairs of vertices in the shape of a cycle.
			The graph~$C$ at the top of \cref{fig:Donut} is a normal cover of~$G$.
			The deck transformation group~$\cD$ of this covering is isomorphic to the infinite cyclic group~$\Z$; it acts on~$C$ by shifting the~$K^4$s to the left or right by multiples of~12. 
			
			Let $\cT = (T, \cV)$ be the canonical \td\ of $C$ whose decomposition tree~$T$ is a double ray and whose parts are the copies of~$K^4$ in~$C$.
			The model $H = T/\cD$ of the \gd\ $(H ,\cG)$ defined by~$\cT$ (\cref{def:GraphDecompViaCovering}) is a cycle of length~12; its family $\cG$ of parts consists of the twelve $K^4$s which form~$G$.\qed
		\end{example}
		
		Our next example justifies why we do not require the parts of a \gd\ of a graph~$G$ to be induced subgraphs of~$G$.
		The example consists of a \gd\ of~$G$ that is defined by a \td\ of a normal cover of~$G$, just as we need it for \cref{def:GraphDecompViaCovering}. But one of its parts is not an induced subgraph of~$G$.
		
		\begin{example}
			Let $G$ be the graph depicted in \cref{fig:noinducedsubgraphs} which is formed by a $K^6$ and a ladder with contracted first and last rungs such that the two respective contraction vertices are identified with two distinct vertices of the~$K^6$.
			The graph~$C$ at the top of \cref{fig:noinducedsubgraphs} is a normal cover of~$G$.
			The deck trans\-formation group~$\cD$ of this cover is the infinite cyclic group~$\Z$, which acts on~$C$ by shifting.

		\begin{figure}[ht]
			\center
			\includegraphics[height=14\baselineskip]{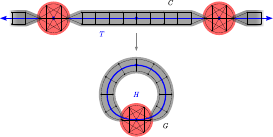}
			\caption{The grey part of~$H$ is not an induced subgraph of~$G$}
			\label{fig:noinducedsubgraphs}
		\end{figure}
		
\cref{fig:noinducedsubgraphs} shows a canonical \td ~$\cT = (T, \cV)$ of~$C$.
			Its decomposition tree~$T$ is a double ray; its parts alternate between red $K^6$s and copies of the modified ladder. 
			All adhesion sets have size one.
			The \gd ~$(H, \cG)$ of~$G$ defined by~$\cT$ as in \cref{def:GraphDecompViaCovering} has a model~$H = T/\cD$ on two vertices, which are joined by two parallel edges.
			It has two parts, the $K^6$ and the modified ladder.
			The latter is not an induced subgraph of~$G$: its two endpoints are joined by an edge of the~$K^6$, which is an edge of~$G$ but not an edge of the ladder part.\qed
		\end{example}
		
		Our final example shows that if $(H, \cG)$ is a \gd\ defined by $(T, \cV)$ via a normal covering $p\colon C \to G$, then the decomposition tree $T$ of~$C$ need not be a covering space of the model~$H$ of~$G$.
		Indeed, one can show that $T$~is a cover of $H$ if and only if the deck transformation group $\cD(p)$ acts properly discontinuously on $T$, i.e., if and only if every deck transformation of $p$ either fixes all the parts of $(T, \cV)$ or none (cf.~\cite{Hatcher}*{Proposition 1.40}).

		\begin{example}
			Let $G$ be a graph obtained from two copies of some quadrangulation of the torus by glueing them together at a single vertex.
			Note that every quadrangulation of the torus is covered by the infinite $(\Z \times \Z)$-grid.
			Thus, a normal cover~$C$ of~$G$ is obtained by identifying infinitely many $(\Z\times\Z)$-grids in single vertices to form an $\aleph_0$-regular `tree of grids'.
			The deck transformation group~$\cD$ of this cover is isomorphic to the free product~$(\Z \times \Z) * (\Z \times \Z)$.%
   \COMMENT{}
			Let~$(T,\cV)$ be the canonical \td\ of~$C$ into its $(\Z\times\Z)$-grids.
			The \gd ~$(H,\cG)$ defined by~$(T,\cV)$ then has~$H = K^2$ as its model, one node for each of two the quadrangulated tori of~$G$.
			Since~$H$ has nodes of degree~1 but $T$ does~not, $T$~cannot cover~$H$.\qed
		\end{example}

\else

In~\cites{GraphDecArXiv, FundGroupPlane}%
   \COMMENT{}
   we give a number of examples highlighting different aspects of \cref{def:GraphDecompViaCovering}.

\fi

\bigbreak\section{\boldmath Local coverings} \label{sec:LocalCoverings}

\noindent
	Our aim in this section is to construct \gd s $(H,\cG)$ of a given graph~$G$ that display its global structure only, leaving the local aspects of the structure of~$G$ to the parts of~$\cG$. We shall use the framework of \cref{def:GraphDecompViaCovering} to achieve this, but will have to choose the covering $p\colon C\to G$ on which it is based specifically to suit our aim. Our choice of this covering will depend only on our intended level of `locality', a~parameter~$r$ we shall be free to choose in order to set a threshold between `local' and `global'.

After defining such `local coverings' formally in \cref{subsec:ClosedWalksCycles,subsec:DefLocalCover}, we show in \cref{subsec:ballpres} that these coverings admit a natural hierarchy as their locality parameter~$r$ grows.

In \cref{subsec:finitegraphs} we show that these local covers~$G_r$ of finite graphs~$G$ always resemble Cayley graphs of finitely presented groups: every $G_r$~has a \gd\ of finitely bounded width modelled on a Cayley graph of~$\cD(p)$, the group of deck transformations of~$p$, which we show is finitely presented.

In \cref{subsec:cayley} we translate the entire relationship between a graph and its local covers to groups: we show that when $G$ is a locally finite Cayley graph of a finitely generated group~$\Gamma$, then its local covers~$G_r$ are Cayley graphs of finitely presented groups that extend~$\Gamma$.%
 	\COMMENT{}

\vskip 1.5 \smallskipamount\subsection{\boldmath Closed walks which stem from cycles} \label{subsec:ClosedWalksCycles}

Let $G$ be any connected graph, and let $x_0\in G$ be any vertex.
We say that a closed walk $W$ in $G$ based at $x_0$ \emph{stems} from a closed walk $Q$ in~$G$ if $W$ can be written as $W = W_0 Q W_0^{-}$ where the \emph{base walk} $W_0$ starts at~$x_0$ and ends at the base vertex of~$Q$.
A closed walk \emph{once around} a cycle $\O$ is a closed walk in~$\O$ which traverses every edge of~$\O$ exactly once.
If $W$ stems from a closed walk once around a cycle $\O$ in~$G$, then $W$ \emph{stems} from~$\O$.

Let $\cO$ be some fixed set of cycles in~$G$.
Given a vertex $x_0 \in G$, denote by $\pi_1^{\cO}(G,x_0)$ the subgroup of $\pi_1(G, x_0)$ generated by the closed walks%
	\COMMENT{}
	at~$x_0$ that stem from a cycle in~$\cO$.
For such a walk $W = W_0 Q W_0^{-}$ its conjugation $U W U^- = (U W_0) Q (U W_0)^{-}$ with any closed walk $U$ based at $x_0$ stems from the same cycle as~$W$, as only the base walk has changed.
Thus, $\pi_1^{\cO}(G,x_0)$ is normal in~$\pi_1(G, x_0)$.%
   \COMMENT{}

We say that a closed walk $W$ in~$G$ is \emph{generated} by~$\cO$ if the following equivalent%
	\ifArXiv\else~\cite{GraphDecArXiv}\fi\ assertions hold:

\begin{enumerate}
	\item there exists a vertex $x_0 \in G$ such that some walk in $\pi_1^{\cO}(G,x_0)$ stems from $W$;
	\item for every vertex $x_0 \in G$, every closed walk at~$x_0$ in~$G$ which stems from $W$ is in~$\pi_1^{\cO}(G,x_0)$.
\end{enumerate}

\noindent
	In particular, the closed walks at any vertex~$x_0$ that are generated by~$\cO$ are precisely those in~$\pi_1^{\cO}(G,x_0)$. Moreover,\penalty-200\ $S^\cO := \pi_1^\cO(G, x_0)$ is \can\ as a subgroup of~$\pi_1(G, x_0)$ if and only if $\cO$ is invariant under the automorphisms~$\phi$ of~$G$: for every vertex $x\in G$ we have $\phi_\ast(S^\cO_x) = S^{\phi(\cO)}_{\phi(x)} = S^\cO_{\phi(x)}$.%
   \COMMENT{}

	\ifArXiv
	While the implication (ii)$\to$(i) is trivial, let us show (i)$\to$(ii).
    To see (ii), let $x_0\in G$ be any vertex and let $W_0 W W_0^-$ be any closed walk at~$x_0$ that stems from~$W$.
    By~(i), there are a vertex $x_0'\in G$ and a closed walk $W':= W_0' W (W'_0)^-\in\pi_1^{\cO}(G,x_0')$ that stems from~$W$.

As $G$ is connected, we may assume that $x_0=x_0'$. Indeed, let $U$ be any walk in~$G$ from~$x_0$ to~$x'_0$. Then $UW'U^-$ clearly stems from~$W$;%
   \COMMENT{}
   let us show that it lies in~$\pi_1^\cO(G,x_0)$ and hence satisfies~(i).%
   \COMMENT{}
   Write $W'\in \pi_1^{\cO}(G,x_0')$ in terms of generators, say $W' = W_1\dots W_n$ where each~$W_i$ is a walk from~$x'_0$ to a cycle in~$\cO$, round it, and back. Then every $UW_i U^-$ lies in~$\pi_1^\cO (G,x_0)$, and hence so does ${UW_1U^-\!\dots UW_nU^- = UW'U^-}$.%
   \COMMENT{}

Thus, $x'_0 = x_0$. Since $W_0 W W_0^-$ can be obtained from $W_0' W (W'_0)^-\in \pi_1^{\cO}(G,x_0)$ by conjugation with $W_0 (W_0')^-$, and since $\pi_1^{\cO}(G,x_0)$ is normal, it follows that $W_0 W W_0^-\in\pi_1^{\cO}(G,x_0)$.
	\medbreak
	\fi
We shall need the following well-known relationship between the fundamental group of a graph and the fundamental cycles of its spanning trees (see~\cite{Hatcher}%
	\COMMENT{}
	for a proof).
Let~$T$ be a spanning tree of a graph~$G$, rooted at a vertex~$x_0$. Given an edge $e=uv$ not on~$T$, there is a unique cycle in the graph~$T+e$, the {\em fundamental cycle\/} of~$e$ with respect to~$T$. The walk $x_0TuvTx_0$ and its reverse $x_0TvuTx_0$ stem from this cycle; they are the only two reduced closed walks at~$x_0$ in $T + e$ that traverse~$e$ exactly once.
For every such edge~$e$ we pick one of these two closed walks and call it the \emph{fundamental closed walk} through~$e$ with respect to~$T$.\looseness=-1

\begin{lemma}\label{lem:FundGroupGenByFundCycles}
	Let $G$ be a connected graph, and let $T$ be a spanning tree of $G$ rooted at a vertex~$x_0$.
	Then the fundamental group $\pi_1(G, x_0)$ is freely generated by the fundamental closed walks with respect to~$T$.\qed
\end{lemma}

\noindent
	In particular,%
	\COMMENT{}
	every closed walk in $G$ is generated by the set of fundamental cycles with respect to any fixed spanning tree of~$G$.	If $G$ is planar, its closed walks are also freely generated by walks%
	\COMMENT{}
	around its inner faces~\cite{FundGroupPlane}, which need not arise as the fundamental cycles with respect to any spanning tree. We shall not make use of this fact here.

\vskip 1.5 \smallskipamount\subsection{\boldmath Definition of local coverings} \label{subsec:DefLocalCover}

\noindent Let $G$ be any connected graph, let $x_0\in G$ be any vertex, and let~$r \in\N$.
The \emph{$r$-local subgroup} of $\pi_1(G,x_0)$ is the group
\[
\pi_1^r(G,x_0):=\pi_1^{\cO_r}(G,x_0)
\]
where $\cO_r$ denotes the set of all cycles in $G$ of length at most~$r$.
When the base point is given implicitly or does not matter, we refer to it simply as~$\pi_1^r(G)$.

The \emph{$r$-local covering} of~$G$, denoted by $p_r\colon\loc\to G$, is the unique covering of~$G$ with characteristic subgroup~$\pi_1^r(G)$.
As $\pi_1^r(G)$~is normal and \can\ (\cref{subsec:ClosedWalksCycles}), so is~$p_r$.%
	\COMMENT{}
We call $\loc$ the \emph{$r$-local covering graph} of~$G$, or its \emph{$r$-local cover} for short.
In the context of $r$-local coverings, we may think of $r$ as the \emph{locality parameter}.
We call a covering $p \colon C \to G$ \emph{local} if $p$ is equivalent to $p_r$ for some~$r \in \N$.

If the value of~$r$ is clear from the context, we refer to the cycles in~$\cO_r$ as the \emph{short} cycles in~$G$, and say that a~closed walk in~$G$ is \emph{generated by short cycles} if it is generated by~$\cO_r$.
Note that, by standard covering space theory, a walk in~$\loc$ is closed if and only if its projection to~$G$ is a closed walk in~$G$ generated by short cycles.%
	\COMMENT{}

\vskip 1.5 \smallskipamount\subsection{\boldmath Local coverings are ball-preserving} \label{subsec:ballpres}

All coverings, by definition, are local homeomorphisms. For $r$-local coverings of graphs we can extend these local homeomorphisms to larger neighbourhoods. How large exactly depends on~$r$, and our next goal is to quantify this.

Let $G$ be any connected graph, and~$\rho \in \N$.
As usual, we write $d_G(x,y)$ for the distance in~$G$ between two vertices $x$ and~$y$.
Given a vertex $v \in G$, let us write $B_G(v,\rhoh)$ for the (combinatorial) \emph{$(\rhoh)$-ball} in~$G$ {\em around~$v$}, the subgraph of~$G$ formed by the vertices at distance at most $\rhoh$ from~$v$ and all the edges $xy\in G$ on these vertices that satisfy
\[
d_G(v,x)+1+d_G(y,v)\le \rho.
\]
Note that $B_G(v,\rhoh)$ lacks all the edges $xy$ where $x$ and~$y$ have distance exactly~$\rho/2$ from~$v$~-- which, of course, can occur only when $\rho$ is even.%
   \COMMENT{}
   Instead, it consists of precisely those vertices and edges of~$G$ that lie on closed walks at~$v$ of length at most~$\rho$. In particular, $B_G(v, \rhoh)$ contains all the cycles of length at most~$\rho$ in~$G$ that meet~$v$.
Given a set $X$ of vertices of~$G$, we write $B_G(X, \rhoh):= \bigcup_{x \in X} B_G(x, \rhoh)$ for the \emph{$(\rhoh)$-ball} in $G$ {\em around~$X$.}

A covering $p \colon C \to G$ is said to \emph{preserve $(\rhoh)$-balls\/} if, for every lift $\hat{v}$ of a vertex $v\in G$, the covering~$p$ maps $B_{C}(\hat{v},\rhoh)$ isomorphically to~$B_G(v,\rhoh)$.%
	\footnote{A notion that could be regarded as a precursor of the preservation of $(\rhoh)$-balls has been studied by Georgakopoulos~\cite{georgakopoulos2017covers}.}
	In particular, then, $p$~projects all the cycles of length at most~$\rho$ in~$C$ isomorphically to cycles of length at most~$\rho$ in~$G$. Note that every covering $C \to G$ preserves 1-balls.%
   \COMMENT{}


The preservation of $(\rhoh)$-balls can be reduced to local injectivity of the covering map\ifArXiv\else~\cite{GraphDecArXiv}\fi:

\begin{lemma}\label{rem:ballpresdistanceincover}
	The following two statements are equivalent for coverings $p \colon C \to G$ and~$\rho \in \N$:
	\begin{enumerate}
		\item $p$ preserves $(\rhoh)$-balls;
		\item for every vertex $v \in G$, the distance in~$C$ between two distinct lifts of $v$ is greater than~$\rho$. \ifArXiv\else\qed\fi
	\end{enumerate} 
\end{lemma}

	\ifArXiv
	\begin{proof}
For~(i)$\to$(ii), suppose for a contradiction that there exists a vertex~$v \in G$ that has distinct lifts~$\hat{v}_1$ and~$\hat{v}_2$ at a distance of at most~$\rho$.
Choose~$v$, $\hat{v}_1$ and~$\hat{v}_2$ so that~$\hat{v}_1$ and~$\hat{v}_2$ have minimum distance in~$C$.
Let~$P$ be a shortest~$\hat{v}_1$--$\hat{v}_2$~path in~$C$; its internal vertices then lie in distinct fibres of~$p$. Hence $p(P) =:O$ is a cycle of length $d_C(\hat{v}_1, \hat{v}_2) \le \rho$, which thus lies in the~$(\rho/2)$-ball around~$v$.%
   \COMMENT{}
   By~(i), the $(\rho/2)$-ball around $\hat{v}_1$ contains a cycle~$Q$ such that $p(Q)=O$. Now $Q$ and~$P$, viewed as walks from~$\hat v_1$, are distinct lifts of the same closed walk at~$v$ in~$G$, which contradicts the uniqueness of path-lifting.

   For (ii)$\to$(i), consider a lift~$\hat{v}$ of a vertex~$v$ of~$G$. Recall that $B_G(v,\rhoh)$ consists of precisely those vertices and edges of~$G$ that lie on closed walks at~$v$ of length at most~$\rho$, and similarly for~$\hat v$ in~$C$. Since closed walks of length at most~$\rho$ at~$\hat v$ project to closed walks of length at most~$\rho$ at~$v$, our $p$~maps $B_C(\hat{v},\rhoh)$ to $B_{G}(v,\rhoh)$. It does so surjectively, since such walks at~$v$ lift to such walks at~$\hat v$ by~(ii). It does so injectively on the vertices of~$B_C(\hat{v},\rhoh)$ by (ii).%
   \COMMENT{}
   This implies that $p$ is injective also on non-parallel edges of~$B_C(\hat{v},\rhoh)$. But parallel edges cannot project to the same edge either, because $p$ is $1$-\ballpres.
	\end{proof}
	\fi

Local coverings, by definition, preserve local structure as encoded in short cycles. Our next lemma says that they preserve local structure also in the more usual sense of copying small neighbourhoods:

\begin{lemma}\label{lem:BallLifting}
	For every $r\in\N$, the $r$-local covering $p_r \colon \loc \to G$ preserves $(\rh)$-balls.
\end{lemma}

\begin{proof}
	By \cref{rem:ballpresdistanceincover} it is enough to show that no two lifts $\hat{v}_1,\hat{v}_2$ of a vertex $v \in G$ are connected in~$\loc$ by a path of length at most~$r$. The projection of any such path~$P$ is a closed walk~$W$ at~$v$ of length at most~$r$. By a straightforward direct argument, $W$~is generated by closed walks at~$v$ stemming from short cycles,%
	\COMMENT{}
	and thus lies in~$\pi_1^r(G,v)$. Hence $W$ lifts to a closed walk at~$\hat v_1$ in~$G_r$, by definition of~$p_r$,%
   \COMMENT{}
   which contradicts the fact that its unique such lift~$P$ is a path.
\end{proof}

Our $r$-local covering~$p_r$ is universal amongst all the $(\rh)$\ballpres\ coverings of~$G$:%
	\COMMENT{}

\begin{lemma} \label{lem:MaxBallPreservingCover}
	For every $(\rh)$\ballpres\ covering $p\colon C\to G$ there is a covering $q\colon \loc\to C$ such that $p_r = p\circ q$, and all such coverings $q$ are equivalent.
	
\end{lemma}

\begin{proof}
	By the Galois correspondence of coverings we only have to show that the characteristic subgroup of~$p$ contains that of~$p_r$.%
	\COMMENT{}
	The latter is generated by the closed walks~$W$ in~$G$ that stem from a short cycle~$O$, so it suffices to show that such~$W$ lift, via~$p$, to closed walks in~$C$.

Our $W$ has the form $W_0 Q W_0^-$, where $Q$ is a walk once around~$O$. Pick a lift~$\hat W$ of~$W$ to~$C$. This includes, as subwalks, lifts $\hat W_0$, $\hat{Q}$ and~$\hat W_0^-$ of $W_0$, $Q$ and~$W_0^-$. Since $O$ is short and $p$ preserves $(\rh)$-balls, $\hat{Q}$~is~closed.%
	\COMMENT{}
	Therefore $\hat W_0^-$ is the reverse of~$\hat W_0$, and $\hat{W} = \hat{W}_0 \hat{Q} \hat{W}_0^-$ is a closed walk in~$C$.
\end{proof}

Let us now see how the $r$-local coverings of~$G$ are related for various values of~$r$. When $r\le r'$, the $(r'/2)$-balls in~$G$, which the $r'$-local covering $p_{r'}\colon G_{r'}\to G$ preserves by \cref{lem:BallLifting}, include the $(\rh)$-balls in~$G$. We may thus apply \cref{lem:MaxBallPreservingCover} with $p:= p_{r'}$ to obtain a covering $q_{r,r'}\colon\! G_r \to G_{r'}$ such that $p_r = p_{r'}\circ q_{r,r'}$, as shown in \cref{fig:r-hierarchy}.

	\vskip-6pt
\begin{figure}[ht]
		\center
\begin{tikzcd}
    G_r \arrow[rdddd, "p_r"'] \arrow[rrdd, swap, "{q_{r,r'}}"] \arrow[rr, "q", bend left=10] &    & (G_{r'})_r \arrow[dd, "{p_{r,r'}}"] \arrow[ll, "q'", bend left=10] \\
        &    &    \\
        &    &    G_{r'} \arrow[ddl, "p_{r'}"]     \\
        &    &     \\
        &    G     &                                    
\end{tikzcd}
	\caption{The interaction of local coverings for different values of~$r$}
		\label{fig:r-hierarchy}
	\end{figure}

Moreover, $G_{r'}$ has its own $r$-local covering $p_{r,r'}\colon\! (G_{r'})_r \to G_{r'}$. We can compose this with~$p_{r'}$ to obtain a covering $p_{r'} \circ p_{r,r'}\colon\! (G_{r'})_r \to G$. This composition of two universal local coverings is itself universal,%
	\COMMENT{}
	as it is equivalent to the $r$-local covering~$p_r$ of~$G$:

\begin{lemma} \label{lem:rlocalCoveringInBetween}
    Given $r \le r'$, the coverings $p_{r,r'}$ and~$q_{r,r'}$ of~$G_{r'}$ are equivalent, while $p_r$ and $p_{r'} \circ p_{r,r'}$ are equivalent as coverings of~$G$. In particular, $(G_{r'})_r$ and~$\loc$ are isomorphic graphs.
\end{lemma}

\begin{proof}
    Since both $p_{r'}$ and $p_{r,r'}$ preserve $(\rh)$-balls, so does their composition~$p_{r'} \circ p_{r,r'}$. Choosing this as $p$ in \cref{lem:MaxBallPreservingCover} yields a covering $q\colon G_r\to (G_{r'})_r$ such that $p_r = p_{r'} \circ p_{r,r'}\circ q$. Since $p_r = p_{r'} \circ q_{r,r'}$ and both $p_r$ and $p_{r'}$ preserve $(\rh)$-balls, so does $q_{r,r'}$. So by \cref{lem:MaxBallPreservingCover} with $p:= q_{r,r'}$, we obtain a covering $q'\colon (G_{r'})_r \to G_r$ with $p_{r,r'} = q_{r,r'} \circ q'$. Combining all the above, we have
    \begin{equation*}
        p_r = p_{r'} \circ p_{r,r'} \circ q = p_{r'} \circ q_{r,r'} \circ q' \circ q = p_r \circ q' \circ q.
    \end{equation*}

    As a concatenation of two coverings, $q'\circ q$ is a covering (of~$G_r$ by itself), as of course is the identity~$\id_{G_r}$. Since $p_r \circ \id_{G_r} = p_r = p_r \circ (q' \circ q)$, as shown above,%
   \COMMENT{}
   \cref{lem:MaxBallPreservingCover} applied with $p=p_r$ yields that $\id_{G_r}$ and $q' \circ q$ are equivalent coverings of~$G_r$.
    In particular, $q' \circ q$ is a homeomorphism.%
   \COMMENT{}
    As $q$ is surjective, being a covering, this implies that $q$ is a homeomorphism, which in turn implies that $q'$ is a homeomorphism. Hence $p_{r,r'}$ and~$q_{r,r'}$ are equivalent as coverings of~$G_{r'}$ as witnessed by~$q'$, while $p_r$ and $p_{r'} \circ p_{r,r'}$ are equivalent as coverings of~$G$ as witnessed by~$q$.
	\end{proof}

\noindent
   Choosing $r=r'$ in \cref{lem:rlocalCoveringInBetween} shows that taking $r$-local covers is idempotent: $(\loc)_r=\loc$. 

\ifArXiv\goodbreak\fi\medbreak

By and large, the connectivity of~$\loc$ increases as the locality parameter~$r$ grows and $\loc$ changes from a tree, which is minimally connected, to the original graph~$G$. How exactly the connectivity of~$\loc$ evolves along the hierarchy of local coverings given by~\cref{lem:rlocalCoveringInBetween} is studied in~\cite{LocalSeps}.

\vskip 1.5 \smallskipamount\subsection{\boldmath Cycle spaces} \label{subsec:IntegralCycleSpace}

As is customary in graph theory, we call the first homology group of a graph its {\em cycle space\/}. Taken over the integers it is its {\em integral\/} cycle space; taken over $\Z/2\Z$ it is its {\em binary\/} cycle space.

To formalise this, let $G$ be any connected graph and $r\in\N$. For every edge $e$ of~$G$ pick an arbitrary fixed orientation~$\ve$; this makes~$G$ into a 1-complex. The {\em integral\/} or {\em binary cycle space\/} of~$G$, then, is the $R$-module~$\CS(G)$%
	\COMMENT{}
	for $R=\Z$ or~$R=\Z/2\Z$ that consists of all $E(G)\to R$ functions with finite support whose sum of the values of the incoming edges at any given vertex equals the sum of the values of its outgoing edges.\looseness=-1

Every closed walk $W = v_0 e_0 v_1 \dots v_{k-1} e_{k-1} v_k$ once around a cycle in~$G$ induces an element $z_W$ of $\CS(G)$ by letting $z_W(e_i) := 1$ if $\ve_i = (e_i,v_i, v_{i+1})$ and $z_W(e_i) := -1$ if $\ve_i = (e_i,v_{i+1}, v_i)$ and $z_W(e_i) := 0$ elsewhere.
Note that the reverse $W^{-}$ of~$W$ induces~$z_{W^{-}} = -z_W$.
As is well known and easy to see, the (functions~$z_W$ of walks~$W$ once around) cycles generate all of~$\CS(G)$.
We write~$\CSr(G)$ for the submodule of~$\CS(G)$  generated by the cycles of length at most~$r$.

\begin{lemma}\label{thm:integralCycleSpace}
	For the $r$-local cover~$\loc$ of~$G$ we have $\CSr(\loc)=\CS(\loc)$.%
	\COMMENT{}
\end{lemma}
\begin{proof}%
   \COMMENT{}
	Let us show first that $\pi_1^r(\loc) = \pi_1(\loc)$. By definition of~$q_{r,r}$ we have $p_r = p_r \circ q_{r,r}$. 
	By \cref{lem:MaxBallPreservingCover} applied with $p = p_r$, it follows from~$p_{r} \circ \id_{G_r} = p_r = p_r \circ q_{r,r}$
   \COMMENT{}
   that $\id_{G_{r}}$ and~$q_{r,r}$ are equivalent coverings of~$G_r$.
	But $q_{r,r}$ is equivalent also to~$p_{r,r}$, by \cref{lem:rlocalCoveringInBetween}. Hence $p_{r,r}$ and $\id_{G_r}$  are equivalent coverings of~$G_r$, and thus have the same characteristic subgroups. These are $\pi_1^r(\loc)$ and~$\pi_1(\loc)$, respectively.

	Since the first homology group of a path-connected space is the abelianisation of its fundamental group,%
	\COMMENT{}
	the fact that $\pi_1^r(\loc) = \pi_1(\loc)$ implies~$\CSr(\loc) = \CS(\loc)$.
\end{proof}

As pointed out for~$\loc$ in the last line of the above proof, we have $\CS(G) = \CSr(G)$ if the fundamental group of~$G$ is generated by short cycles. When $G$ is planar the converse holds too; this is shown in \cite{FundGroupPlane}.

\vskip 1.5 \smallskipamount\subsection{\boldmath Local coverings of finite graphs}\label{subsec:finitegraphs}


\ifArXiv
		While local coverings of finite graphs are often infinite, this is not necessarily the case. The following example due to Bowler~\cite{NathanComm2021} exhibits finite graphs with $n$-sheeted coverings for any~$n\in\N$.
	\endgraf
	
		\begin{figure}[ht]
		\center
		\includegraphics[height=14.5\baselineskip]{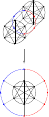}
		\caption{A graph whose $4$-local covering is finite.}
		\label{fig:finitelocalcovering}
        \vskip-6pt
	\end{figure}
	
\begin{example*}
	Let~$G'$ be the graph depicted in \cref{fig:finitelocalcovering}.
	We construct the desired graph~$G$ from~$G'$ in that we identify the two paths of length $3$ drawn in red and blue which partition the outer cycle of $G'$, and we do this identification according to the indicated direction of the two paths.
	In other words, if we view~$G'$ embedded in the plane, symmetrically as drawn here, with the outer cycle equal to $S^1$, then we construct~$G$ from~$G'$ by taking the quotient of this outer cycle under the antipodal map $x \mapsto -x$ on~$S^1$. 
	Now taking two copies of~$G'$ and glueing them together by identifying their outer cycles yields a $2$-sheeted covering of~$G$, and one can check that this is indeed the $4$-local covering of~$G$ by \cref{lem:MaxBallPreservingCover}.
	\endgraf
	The above construction can easily be adapted to yield $n$-sheeted local coverings by partitioning the directed outer cycle of (a possibly larger version of) $G'$ into $n$ paths of the same length at least $3$ which we then identify as above to obtain~$G$.
	The respective $4$-local coverings then arise by glueing $n$ copies of this larger $G'$ along their outer cycles.\qed
	\end{example*}

\else

While local coverings of finite graphs are often infinite, this is not necessarily the case; in~\cite{GraphDecArXiv} we present finite graphs with $n$-sheeted coverings for any~$n\in\N$ constructed by Bowler~\cite{NathanComm2021}.

\fi

Our aim in this section is to prove that the local covers of finite graphs resemble Cayley graphs of finitely presented groups. We first show that every normal cover of a finite graph has a \gd\ of finitely bounded width, modelled on a Cayley graph of its group of deck transformations. This group is always finitely generated.%
	\COMMENT{}
	In the case of our local coverings we shall see that it is even finitely presented.

Our first result, the~structure theorem for arbitrary normal covers of finite graphs, can be seen as a \v{S}varc-Milnor lemma~\cites{vsvarc1955volume, milnor1968note} that provides additional structural information:

\begin{theorem}\label{thm:graphdecompoffinitewidthalongcayley}
	Let $p \colon C \to G$ be a normal covering of a finite connected graph~$G$.
	Then the group $\cD$ of all deck transformations of~$p$ has a finite generating set~$S$ such that $C$ has a \gd\ modelled on~$D:=\cayley{\cD}{S}$ with connected parts of finitely bounded size. If $\hat{x}_0\in C$ is a~lift of any vertex~$x$ of~$G$ then these parts~$C_\phi$, one for each~$\phi\in\cD$, can be chosen as~$C_\phi = B_C(\phi(\hat{x}_0),|G|)$.
	
	With these~$C_\phi$ the \gd\ $(D, (C_\phi)_{\phi \in \cD})$ of~$C$ is honest, and the map $f \mapsto g_f$ defined in \cref{subsec:GraphDecProperties} induces a bijection between the ends of~$C$ and those of~$\cD$.
	If $\cD$ is an accessible group, then $C$ is an accessible graph.
\end{theorem}

\begin{proof}
	Choose $(C_\phi)_{\phi \in \cD}$ as in the statement of the theorem. As these~$C_\phi$ depend only on~$\cD$, not on the choice of generators for our Cayley graph~$\cayley{\cD}{S}$, we can first verify~\cref{ax:GraphDecomp1parts} and choose~$S$ later.
	
Since $C$ is connected, every vertex is incident with some edge. To verify~\cref{ax:GraphDecomp1parts} it is therefore enough to show that every edge of $C$ is contained in some part~$C_\phi$. Every edge of~$C$ is a lift $\hat{e} = \hat{u}\hat{v}$ of an edge $e=uv$ of~$G$. These have distance at most $|G|-1$ from~$x$ in~$G$.
	Pick a shortest $u$--$x$ path~$P$ in~$G$, lift it to a path~$\hat{P}$ starting at~$\hat{u}$, and let $\hat{x}_1$ be the last vertex of~$\hat{P}$.
	As $\hat P$ has length at most~$|G|-1$, this shows that $\hat{e}$ lies in~$B_{C}(\hat{x}_1, |G|)$.
	Since $p$ is normal, there exists a deck transformation $\phi \in \cD$ such that $\phi(\hat{x}_0) = \hat{x}_1$. Thus~$\hat{e} \in C_{\phi}$, completing the proof of~\cref{ax:GraphDecomp1parts}.
	
	Let us now find a finite generating set $S$ of $\cD$ such that, for $D = \cayley{\cD}S$, the pair $(D, (C_\phi)_{\phi \in \cD})$ will satisfy~\cref{ax:GraphDecomp2} and thus be a \gd\ of~$C$. We shall%
	\COMMENT{}
	have to choose~$S$ so that, for every lift~$\hat v\in C$ of a vertex~$v$ of~$G$, the subgraph~$D_{\hat v}$ of~$D$ induced by~$\{ \phi \in \cD \colon \hat{v} \in C_\phi \}$ is connected.

To define~$S$, fix of every $v\in G$ one lift~$\hat v_0$. We shall pick enough generators for~$S$ to make every~$D_{\hat v_0}$ connected, and then deduce the same for the other lifts~$\hat v$ of each~$v$. In fact, we shall add to~$S$ for every two nodes of~$D_{\hat v_0}$ a generator that will make them adjacent in~$\cayley{\cD}S$, and hence in~$D_{\hat v_0}$.

The pairs of nodes of~$D_{\hat v_0}$ are the pairs $(\phi,\phi')\in\cD^2$ such that $\hat{v}_0 \in C_\phi \cap C_{\phi'}$. To give~$D_{\hat v_0}$ an edge $(\phi,s)$%
	\COMMENT{}
	from $\phi$ to~$\phi'$ we need to add $s = \phi^{-1}\phi'$ to~$S$. Let $S_v$ be the set of all these~$s$, setting
 $$S_v := \{\, \phi^{-1}\phi'\mid \hat{v}_0 \in C_\phi \cap C_{\phi'}\}\subset \cD,$$
 and let~$S := \bigcup_{v \in V(G)} S_v$. Then the subgraphs $D_{\hat v_0}$ of $D = \cayley{\cD}S$ are complete for all~$v\in G$.

For a proof that each~$S_v$, and hence~$S$, is finite, it suffices to show that $\hat v_0\in C_\phi$ for only finitely many~$\phi$.%
	\COMMENT{}
	For every such~$\phi$ there is a path in~$C$ of length at most~$|G|$ from~$\hat v_0$ to~$\phi(x_0)$. But $B_C(\hat{v}_0,|G|)$ is finite, and the~$\phi(x_0)$ are distinct for different~$\phi$, by the uniqueness of path lifting.%
	\COMMENT{}
	So there are only finitely many such~$\phi$, as required.

To prove~\cref{ax:GraphDecomp2} for lifts~$\hat v$ other than~$\hat v_0$, recall from \cref{sec:interplayofgraphsandgroups} that $\cD$ acts on~$D$ by left-multiplication. Since our covering~$p$ is normal, there exists $\psi\in\cD$ such that $\hat v = \psi(\hat v_0)$. Then $\psi\act D_{\hat v_0} = \{\,\psi\phi\mid \phi\in D_{\hat v_0}\}$, and any two of these nodes $\psi\phi$ and~$\psi\phi'$%
   \COMMENT{}
   are joined by the edge $(\psi\phi,s)$ for $s = \phi^{-1}\phi'\in S_v$. So it suffices to check that the node set of $\psi\act D_{\hat v_0}$ equals that of~$D_{\psi(\hat v_0)} = D_{\hat v}$.

By definition, $D_{\psi(\hat v_0)}$%
	\COMMENT{}
	consists of those $\chi\in\cD$ for which $\psi(\hat v_0)\in C_\chi$, i.e., for which $\psi(\hat v_0)$ has distance at most~$|G|$ from~$\chi(\hat x_0)$. With $\chi =: \psi\phi$%
	\COMMENT{}
	the latter is equivalent to%
	\COMMENT{}
	$\hat v_0$ having distance at most~$|G|$ from~$\phi(\hat x_0)$, i.e.\ to $\phi\in D_{\hat v_0}$. Hence
	$$\chi\in D_{\psi(\hat v_0)} \Leftrightarrow \phi\in D_{\hat v_0}\Leftrightarrow \psi\phi\in \psi\act D_{\hat v_0}\Leftrightarrow \chi\in \psi\act D_{\hat v_0}$$
	as desired. This completes the proof that $(D, (C_\phi)_{\phi \in \cD})$ is a \gd.
	
For a proof that $S$ generates~$\cD$, we have to show that $D = \cayley{\cD}{S}$ is connected. This follows from the connectedness of~$C$ by \cref{lem:graphdecompseparator}, since the parts~$C_\phi$ of our \gd\ are non-empty.

	Next, we check that our \gd\ $(D, (C_\phi)_{\phi \in \cD})$ is honest. As the $C_\phi$ are non-empty, we only have to show that $C_\phi\cap C_{\phi'}\ne\emptyset$ for any adjacent nodes $\phi,\phi'$ of~$D$. As $D = \cayley{\cD}{S}$ there exists $s\in S$ such that $\phi' = \phi s$.%
	\COMMENT{}
	This~$s$ lies in some~$S_v$.%
	\COMMENT{}
	This means, by definition of~$S_v$, that there exist $\phi_0,\phi'_0\in\cD$ such that $s = \phi_0^{-1}\phi'_0$ and $\hat v_0\in C_{\phi_0}\cap C_{\phi'_0}$. Let $\psi := \phi\phi_0^{-1}$. Then $\phi = \psi\phi_0$ and $\phi' = \phi s = \psi\phi_0 s = \psi\phi'_0$.%
	\COMMENT{}
	Hence $\hat v_0\in C_{\phi_0}\cap C_{\phi'_0}$ implies that $\psi(\hat v_0)\in C_{\psi\phi_0}\cap C_{\psi\phi'_0} = C_\phi\cap C_{\phi'}$, since the action of~$\cD$ on itself by left-multiplication commutes with its natural action on~$C$.%
	\COMMENT{}
	In particular, $C_\phi\cap C_{\phi'}\ne\emptyset$, as required for our honesty proof.
	
	We have shown that our \gd\ $(D, (C_\phi)_{\phi \in \cD})$ is honest and has finite connected parts.
	In our proof that the sets~$S_v$ are finite we saw that each vertex of~$C$ lies in~$C_\phi$ for only finitely many~$\phi \in \cD$: that $(D, (C_\phi)_{\phi \in \cD})$ has finite spread.
	By \cref{lem:GraphDecDisplayingEnds}, therefore, the directions of~$C$ and of~$D = \cayley{\cD}{S}$ are in bijective correspondence, and hence so are their ends~\cite{Ends}*{Theorem 2.2}. The ends of~$\cD$, by definition, are those of~$D$.%
   \COMMENT{}
	
	Finally, since $C$ covers the finite graph~$G$ it has finitely bounded degrees. The parts~$C_\phi = B_C(\phi(\hat{x}_0), |G|)$ thus have finitely bounded size. If the group~$\cD$ is accessible then so is its Cayley graph~$D$~\cite{ThomassenWoess}, and hence by \cref{lem:AccessibleDecomp} so is~$C$.
	\end{proof}

The groups of deck transformations of our local coverings of finite graphs are not only finitely generated but even finitely presented:

\begin{proposition}\label{lem:D_risfinitelypresented}
	Let $G$ be a finite graph, and let~$r \in \N$.
	Then the group $\deckr$ of deck transformations of the $r$-local covering of $G$ is finitely presented.
\end{proposition}

\begin{proof}
    Fix a vertex~$x_0 \in G$. The fundamental group~$\pi_1(G, x_0)$ of~$G$ is finitely presented by \cref{lem:FundGroupGenByFundCycles}.%
	\COMMENT{}
    To show that~$\deckr = \pi_1(G, x_0)/\pi_1^r(G, x_0)$ is finitely presented, it remains to prove that~$\pi_1^r(G, x_0)$ is the normal closure of a finite set of closed walks in~$G$.
    
Recall that $\pi_1^r(G, x_0)$ is generated by the closed walks at~$x_0$ stemming from short cycles in~$G$. There are only finitely many such cycles, but infinitely many such walks. However, the walks stemming from the same short cycle~$O$ are all conjugates of each other in~$\pi_1(G,x_0)$.%
	\COMMENT{}
	Hence we may pick one of them for each~$O$, finitely many in total, and their normal closure in~$\pi_1(G,x_0)$ will be exactly~$\pi_1^r(G,x_0)$.%
	\COMMENT{}
	\end{proof}

\begin{corollary} \label{cor:LocCovFinGraphAreAccessible}
	Local covers of finite graphs are accessible. They have \gd s into connected parts of finitely bounded size modelled on Cayley graphs of finitely presented groups.
\end{corollary}

\begin{proof}
	Finitely presented groups are accessible~\cite{FinitelyPresentedAccessible}. Apply \cref{thm:graphdecompoffinitewidthalongcayley} and \cref{lem:D_risfinitelypresented}.
	\end{proof}

\ifArXiv
 In contrast to local covers, arbitrary normal covers of finite graphs can be inaccessible:

\begin{example*}
    Let $\Gamma = \pres{S}{R}$ be an inaccessible group finitely generated by~$S$,%
	\COMMENT{}
    and let $G$ be any connected graph with ${|V(G)| - 1 + |S|}$ edges.
    Then $\pi_1(G)$ is isomorphic to the free group~$\free{S}$  (\cref{lem:FundGroupGenByFundCycles}), so we may view the normal closure~$N = \normcl{R}{S}$ of~$R$ in~$S$ as a normal subgroup of~$\pi_1(G)$.

    Let~$p\colon C \mapsto G$ be the normal covering of~$G$ with characteristic subgroup~$N$.
    Then the group~$\cD$ of all deck transformations of~$p$ is isomorphic to~$\pi_1(G)/N$ and hence to~$\Gamma$.
    Thus, $\cD$~is inaccessible, and hence so are all its locally finite Cayley graphs~\cite{ThomassenWoess}.

    Now consider the \gd\ $(D, (C_\phi)_{\phi \in \cD})$ of~$C$ modelled on a locally finite Cayley graph~$D$ of~$\cD$ as given by \cref{thm:graphdecompoffinitewidthalongcayley}.
    This decomposition is honest, and its parts~$C_\phi = B_C(\phi(\hat{x}_0), |G|)$ are connected. Hence by \cref{lem:DualGraphDecomp} there is a dual decomposition $(C,(D_v)_{v\in C})$ of~$D$, again with connected parts~$D_v$.%
   \COMMENT{}
   These~$D_v$ have finitely bounded size by the proof of~\cref{thm:graphdecompoffinitewidthalongcayley}.%
    \COMMENT{}
    Since $D$ is inaccessible, \cref{lem:AccessibleDecomp} applied to its decomposition $(C,(D_v)_{v\in C})$ yields that $C$ is inaccessible too.\qed
\end{example*}

\else

We remark that arbitrary normal covers of finite graphs can be inaccessible~\cite{GraphDecArXiv}.%
	\COMMENT{}

\fi

\vskip 1.5 \smallskipamount\subsection{\boldmath Local coverings of Cayley graphs}\label{subsec:cayley}

When our graph~$G$ is itself a Cayley graph, of some finitely generated (but not necessarily finite) group~$\Gamma$, say, we can say even more about its local covers than we could in~\cref{subsec:finitegraphs}. These covers are now Cayley graphs too, of groups extending~$\Gamma$:%
 	\COMMENT{}

\begin{theorem} \label{thm:localcoveringofcayleyiscayley}
	Let $G$ be a connected locally finite Cayley graph of some group~$\Gamma$, and~$r \in \N$.
	Then $\loc$ is a connected locally finite Cayley graph of a finitely presented group $\Gamma_r$ of which $\Gamma$ is a quotient.
	\end{theorem}%
	\COMMENT{}

\noindent
	For our proof of \cref{thm:localcoveringofcayleyiscayley} we need a tool from combinatorial group theory.

\medbreak

Let $G$ be a Cayley graph of a group $\Gamma$ with respect to a finite generating set~$S$.
The fundamental group $\pi_1(G,x_0)$ of $G$ is then isomorphic to the kernel of the canonical homomorphism $\phi_{\Gamma, S}\colon\free{S}\to\Gamma$, where $F(S)$ is the free group generated by~$S$. Let us define such an isomorphism explicitly.

The \emph{edge-labelling} $\ell\colon E(G) \to S$ of a Cayley graph $G = \cayley{\Gamma}{S}$ is defined as $\ell(e) := s$ when $e = (g,s)$.%
	\COMMENT{}
	We extend $\ell$ to a map from the walks in~$G$ to~$\free{S}$ by mapping a walk $W = g_0 e_0 g_1 \dots g_{k-1} e_{k-1} g_k$ to the group element $\ell(e_0)^{\epsilon_0} \cdots \ell(e_{k-1})^{\epsilon_{k-1}}$ of~$\free{S}$, where $\epsilon_i$ is $1$ if $e_i = (g_i, s)$ for some $s \in S$ and $-1$ otherwise.%
	\COMMENT{}
Since two walks map to the same element of $\free{S}$ under $\ell$ if their reductions coincide, $\ell$~defines a group homomorphism $\ell_\ast\colon\pi_1(G,x_0)\to\free{S}$.
This $\ell_\ast$ is our explicit isomorphism between $\pi_1(G, x_0)$ and $\ker(\phi_{\Gamma, S}) \leq \free{S}$:

\begin{lemma}[\cite{magnus2004combinatorial}*{Theorem 1.6}] \label{lem:fundamentalgroupandrelations}
	Let $G$ be a Cayley graph of a group $\Gamma$ with respect to a finite generating set~$S$, and let $\ell$ be its edge-labelling.
	Fix any vertex $x_0$ of~$G$.
	Then $\ell_\ast$ is an isomorphism from the fundamental group $\pi_1(G,x_0)$ to the kernel of the canonical homomorphism $\phi_{\Gamma, S}\colon\free{S}\to\Gamma$.
\end{lemma}

\begin{proof}[Proof of \cref{thm:localcoveringofcayleyiscayley}]
	Let $G$ be the Cayley graph of~$\Gamma$ with respect to the finite generating set~$S$.
	Then $\Gamma$ has the group presentation $\pres{S}{R}$ with~$R = \ker(\phi_{\Gamma, S})$.
	We aim to show that $\loc$ is the Cayley graph of a group $\Gamma'$ with the same set $S$ of generators as~$G$, but with a different and finite set $R' \subseteq R$ of relators.
	We shall first construct $\Gamma'$ and~$R'$, then check that $\cayley{\Gamma'}{S}$ is a covering graph of~$G$, and finally show that this covering is equivalent to our $r$-local covering of~$G$.
	
	In $G = \cayley{\Gamma}{S}$ let us fix the identity $1 \in \Gamma$ as our basepoint~$x_0$.
	Since $G$ is locally finite, there are only finitely many cycles of length at most~$r$ that contain~$x_0$. Hence the set of closed walks at~$x_0$ that only go once around such a cycle is finite too;%
	\COMMENT{}
	let $R'\subset F(S)$ be its image under~$\ell$.
	By \cref{lem:fundamentalgroupandrelations} we have~$R' \subseteq  \ell_\ast(\pi_1(G,x_0)) = \ker(\phi_{\Gamma, S}) \subseteq \free{S}$.
	As $R'$ is finite, the group $\Gamma' := \pres{S}{R'}$ is finitely presented.
	
	Let us construct a covering $p'\colon G' \to G$ for $G' = \cayley{\Gamma'}{S}$.
	Since $R' \subseteq \ker(\phi_{\Gamma, S})$, we have $\ker(\phi_{\Gamma', S}) = \normcl{R'}{F(S)} \subseteq \ker(\phi_{\Gamma, S})$.%
	\COMMENT{}
	Hence $\Gamma = \free{S}/\ker(\phi_{\Gamma, S})$ is a quotient of $\Gamma' = \free{S}/\ker(\phi_{\Gamma', S})$. The quotient map $p'\colon\Gamma'\to\Gamma$ between these groups, the vertex sets of their Cayley graphs $G'$ and~$G$, extends to a covering $p'\colon G'\to G$ by $(g',s)\mapsto (p'(g'),s)$ for edges, since both Cayley graphs use the same set~$S$ as edge labels.
Explicitly, we have $\ell \circ p' = \ell'$ for the edge-labelling~$\ell'$ of~$G'$.%
	\COMMENT{}
	
	Let us show that $p'$ is equivalent to our $r$-local covering~$p_r \colon \loc \to G$.
	We do so by showing that their characteristic subgroups coincide.%
	\COMMENT{}
	By \cref{lem:fundamentalgroupandrelations}, this will be the case if $\ell'_\ast(\pi_1(G', \hat{x}_0))$ and $\ell_\ast(\pi_1^r(G, x_0))$ are the same subgroup of~$\free{S}$, where $\hat{x}_0$ is any element in~$p'^{-1}(x_0)$.%
	\COMMENT{}
	Let $R_r$ be the image under~$\ell$ of the set of closed walks at~$x_0$ which stem from a cycle in~$G$ of length at most~$r$. Note that $R_r$ is usually infinite, whereas $R'\subset R_r$ is finite.%
   \COMMENT{}
	Again by \cref{lem:fundamentalgroupandrelations} we have $\normcl{R'}{F(S)} = \ell'_\ast(\pi_1(G',\hat{x}_0))$ and $\normcl{R_r}{F(S)} = \ell_\ast(\pi_1^r(G,x_0))$.%
   \COMMENT{}
	It thus remains to show that~$\normcl{R_r}{F(S)} = \normcl{R'}{F(S)}$.
	
	Since $R' \subseteq R_r$, we immediately have $\normcl{R'}{F(S)} \subseteq \normcl{R_r}{F(S)}$. For the converse inclusion let $t\in R_r$ be given.
	Then there exists a reduced closed walk $W$ at $x_0$ which stems from a cycle $\O$ of length at most $r$ such that~$\ell(W) = t$. This $W$ consists of a base walk~$W_0$ and a closed walk~$Q$ once around~$\O$, so that~$W = W_0 Q W_0^-$.
	Let $g \in V(G) = \Gamma$ be the base vertex of~$Q$.
	Left-multiplication by~$g^{-1}$ defines an automorphism $\psi\colon h\mapsto g^{-1} h$ of~$G$. Now $\psi(Q)$ is a closed walk, based at $\psi(g) = 1 = x_0$, once around the cycle $\psi(\O)$ of length at most~$r$. Since $\psi$ is an automorphism of a Cayley graph, the labellings on the edges of $\O$ and $\psi(\O)$ agree, so that $\ell(Q)  = \ell(\psi(Q)) \in R'$.
	Thus, $t = \ell(W) = \ell(W_0)\ell(Q)\ell(W_0^-) = \ell(W_0)\ell(Q)\ell(W_0)^{-1}$ is contained in the normal subgroup~$\normcl{R'}{F(S)}$.
	This yields $R_r \subseteq \normcl{R'}{F(S)}$ and completes the proof.
\end{proof}

\begin{corollary} \label{cor:LocCovCayleyAccessible}
	Local covers of connected locally finite Cayley graphs are accessible.
\end{corollary}
\begin{proof}
	Finitely presented groups are accessible~\cite{FinitelyPresentedAccessible}, and hence so are their locally finite Cayley graphs~\cite{ThomassenWoess}.
	Apply \cref{thm:localcoveringofcayleyiscayley}.
\end{proof}

\bigbreak\section{\boldmath Displaying global structure} \label{sec:OverviewProof}

\noindent
	In this section we make the statements of Theorems \ref{main:KeyTheorem} and~\ref{main:CayleyThm} precise and give an overview of their proof. This will include a proof of \cref{maincor:TangleAccessibility}.

In \cref{subsec:TanglesBackground} we recall how a tree of tangles can describe the tree-like aspects of the global structure of any graph.
In \cref{subsec:rGlobalStructure} we apply this to the typically infinite graphs~$\loc$ that reflect the $r$-local structure of the finite graphs~$G$ given in \cref{main:KeyTheorem}, to describe their (tree-like) $r$-global structure. Projected back to~$G$, this $r$-global structure of~$\loc$ ceases to be infinite and tree-like, but defines what we call the (finite) $r$-global structure of~$G$.

\cref{subsec:rGlobalStructure} ends with a precise definition of the term of `$r$-global structure' used in Theorems~\ref{main:KeyTheorem} and~\ref{main:CayleyThm}, and of what it means for another graph to `display' that structure. We proceed from there to give an outline of the proof of \cref{main:KeyTheorem}. This includes a common generalisation, \cref{main:KeyTheoremGen}, of Theorems~\ref{main:KeyTheorem} and~\ref{main:CayleyThm}.
The outline will serve as our road map when we prove Theorems~\ref{main:KeyTheorem} and~\ref{main:CayleyThm} (by way of proving \cref{main:KeyTheoremGen}) in \cref{sec:LocalCoveringsAreAccessible,sec:TangleAccessibility,sec:ProofMainTheorem}.

\vskip 1.5 \smallskipamount\subsection{\boldmath Separations, tangles, and trees of tangles} \label{subsec:TanglesBackground}

In this section we recall basic facts and terminology about tangles; for a more detailed account see~\cite{DiestelBook25}*{Ch.12.5}.

Let $G$ be a graph.
Given a set $S$ of separations of~$G$, we write~$\vS$ for the set of orientations of separations in~$S$, two for every $s\in S$.
An \emph{orientation} $O$ of~$S$ contains for every separation in~$S$ exactly one of its two orientations.
Recall that $(A,B) \leq (C,D)$ for two oriented separations $(A,B)$ and $(C,D)$ of $G$ if $A \subseteq C$ and~$B \supseteq D$.
If $O$ does not contain an oriented separation $(B,A)$ of~$G$ whenever $(A,B) \leq (C,D)\in O$, then $O$ is \emph{consistent}.
A set $N$ of separations of $G$ is \emph{nested} if every two of its elements, $\{A,B\}$~and~$\{C,D\}$ say, have orientations $(A,B) \leq (C,D)$.
Separations that are not nested are said to \emph{cross}.

Let $\Gamma$ be a group acting on a graph~$G$.
A set $S$ of separations of $G$ is \emph{$\Gamma$-canonical} if for every $g \in \Gamma$ and every $\{A, B\} \in S$ we have~$g \act \{A, B\} := \{g \act A, g \act B\} \in S$.
If $S$ is $\Aut(G)$-canonical, we call it \emph{canonical}.

Let us now define tangles. Given any $k \in \N \cup \{\aleph_0\}$, let $S_k$ denote the set of all separations of $G$ of order~$< k$. A~\emph{$k$-tangle} in $G$ is an orientation~$\tau$ of~$S_k$ such that for all $(A_1,B_1), (A_2,B_2), (A_3,B_3) \in \tau$, not necessarily distinct, we have~$G[A_1] \cup G[A_2] \cup G[A_3] \neq G$.
A \emph{tangle} is a $k$-tangle for some~$k$.
Note that every tangle is consistent.

Examples of tangles include blocks and ends. A~{\em $k$-block} in~$G$ is a maximal set of at least~$k$ vertices that is included in one of the two sides of each separation in~$S_k$. It {\em orients\/} that separation {\em towards\/} this side: as $(A,B)$ if it lies in~$B$. Every $k$-block~$X$ thus defines an orientation~$\tau_X$ of~$S_k$. If $|X| > 3(k-1)/2$, then $\tau_X$ is a tangle~\cite{ForcingBlocks}*{Theorem 6.1}. By the maximality in the definition of a $k$-block, distinct $k$-blocks define different tangles.

Given an end $\omega$ of $G$, let $f_\omega$ be the corresponding direction of $G$ (see \cref{subsec:GraphDecProperties}). Orienting every separation $\{A, B\} \in S_{\aleph_0}$ towards the side that contains~$f_\omega(A \cap B)$, we obtain a tangle $\tau_\omega$ of~$S_{\aleph_0}$~\cites{EndsAndTangles}. As in the case of blocks, distinct ends define different tangles.

A separation of $G$ \emph{distinguishes} two tangles in $G$ if they orient it differently.
The separation distinguishes these two tangles \emph{efficiently} if no separation of smaller order distinguishes them.
Two tangles are \emph{distinguishable} if some separation distinguishes them; note that any two $k$-tangles for the same~$k$ are distinguishable.
A set $N$ of separations \emph{distinguishes} a set of tangles (efficiently) if every two distinguishable tangles in this set are (efficiently) distinguished by some separation in~$N$.
A \emph{tree of tangles} of~$G$, in this paper,%
	\COMMENT{}
	is a nested set of separations of~$G$ which efficiently distinguishes the entire set of tangles in~$G$.

A~separation of~$G$ is \emph{relevant} if it efficiently distinguishes some two tangles $\tau$ and~$\tau'$ in~$G$. One can show%
   \COMMENT{}
   that such separations $\{A,B\}$ are \emph{tight}: for $X:=A\cap B$ the graph $G - X$ has components $C\subset G[A]$ and $D\subset G[B]$ such that every vertex in~$X$ sends edges to both $C$ and~$D$. Indeed, these are the unique such components for which $(V(G-C), V(C)\cup X)\in\tau$ and $(V(G-D),V(D)\cup X)\in\tau'$ or vice versa~\cite[Ex.\,43(iii)]{DiestelBook25}.

Note that a finite-order separation of~$G$ distinguishes the tangles induced by two ends of~$G$ if and only if it distinguishes these ends in the usual sense that they contain rays on different sides of that separation. It distinguishes these two ends {\em efficiently\/} if no separation of lower order distinguishes them.

Let us remember all the above properties of trees of tangles for later use in \cref{sec:ProofMainTheorem}:

\begin{lemma} \label{lem:ToTProperties}
	Let $N$ be a tree of tangles of~$G$ in which every separation is relevant.
	Then $N$ efficiently distinguishes every two ends of~$G$. Every separation in~$N$ has finite order and is tight.\qed
\end{lemma}

\vskip 1.5 \smallskipamount\subsection{\boldmath The $r$-global structure of $G$} \label{subsec:rGlobalStructure}

In general, a locally finite graph~$G$ may have several canonical trees of tangles.
However, there is one specific canonical tree of tangles~$N(G)$ that is particularly natural, which we therefore fix as an invariant of those locally finite graphs in which it exists. In this section we describe the construction of~$N(G)$, and then show in which sense it is particularly natural.

Recall~\cites{ThomassenWoess, hamann2018stallings} that a graph $G$ is {\em accessible\/} if there exists a number $K \in \N$ such that every two ends of~$G$ are distinguished by a separation of order at most~$K$.
Extending this notion from ends to tangles, we call a graph~$G$ \emph{tangle-accessible} if there exists a number $K \in \N$ such that every two distinguishable tangles in~$G$ are distinguished by a separation of order at most~$K$.

Let $G$ be a connected, locally finite, and tangle-accessible graph. Denote by~$D$ the set of all finite-order separations of~$G$ that distinguish some two tangles in~$G$ efficiently.
Since $G$ is tangle-accessible, all the separations in~$D$ have order at most some fixed $K \in \N$, and hence they each cross only finitely many other separations in~$D$~\cite{InfiniteSplinters}*{proof of Proposition~6.2}.%
	\COMMENT{}
	
\begin{definition}\label{def:N(G)}
	For every pair $\tau,\tau'$ of distinguishable tangles in~$G$ let $N_{\tau, \tau'}$ be the set of all separations that efficiently distinguish $\tau$ from~$\tau'$ and, among these,%
	\COMMENT{}
	cross as few separations in~$D$ as possible. Let $N(G)$ denote the union of all these sets~$N_{\tau, \tau'}$.
	\end{definition}%
   \COMMENT{}

\noindent
   Note that the separations in~$N_{\tau,\tau'}$ in \cref{def:N(G)} are relevant.%
	\COMMENT{}
The set~$N(G)$ is clearly canonical, it distinguishes all the tangles in~$G$, and its elements have finitely bounded order. It~is also still nested:

\begin{lemma}[\cite{carmesin2022entanglements}*{Theorem~4.2}] \label{thm:Entanglements}
	For every connected, locally finite, and tangle-accessible graph $G$, the set $N(G)$ of separations defined in \cref{def:N(G)} is a canonical tree of tangles. All the separations in~$N(G)$ are relevant; in particular,%
	\COMMENT{}
	they have finitely bounded order.%
	\COMMENT{}
\end{lemma}

\noindent
   We sometimes refer to $N(G)$ informally as {\it the\/} canonical tree of tangles of~$G$.

\medbreak

   We shall prove in \cref{sec:TangleAccessibility} that the $r$-local covers~$\loc$ of the graphs~$G$ we consider in Theorems~\ref{main:KeyTheorem}~and~\ref{main:CayleyThm} satisfy the premise of \cref{thm:Entanglements}. We shall see that they have unique canonical \td s $\cT = (T,\cV)$ such that $\alpha_\cT\colon E(T)\to N(\loc)$ is a bijection.
We say that $N(\loc)$ {\it induces\/} this \td~$\cT$ of~$\loc$.%
   \COMMENT{}\looseness=-1

To see in which sense $N(G)$ is particularly natural, consider any tree of tangles~$N$ in~$G$. By definition,%
	\COMMENT{}
	$N$~has to contain for every pair $\tau,\tau'$ of tangles a separation $s(\tau,\tau')$ that distinguishes them efficiently. If we want $N$ to be canonical, however, we cannot simply pick for~$N$ just any such separation~$s(\tau,\tau')$ for each pair $\tau,\tau'$ of tangles: our choice has to be definable in terms of invariants of~$G$. This may force us to pick not just one such separation~$s(\tau,\tau')$ for each pair~$\tau,\tau'$, but a set $N(\tau,\tau')$ of more than one separation. As $N(\tau,\tau')$ is to be included in~$N$, it will be a nested subset of the set $D(\tau,\tau')$ of separations that distinguish $\tau$ from~$\tau'$ efficiently. Our set~$N(G)$ now chooses these $N(\tau,\tau')$ with reference to~$D(\tau,\tau')$ alone: they consist of {\em all\/} the elements of~$D(\tau,\tau')$ that are the most promising candidates for inclusion in~$N(\tau,\tau')$, in that they cross as few other elements of~$D(\tau,\tau')$ as possible.

As described in \cref{sec:LocalCoverings}, the $r$-local cover~$\loc$ of a graph $G$ agrees with $G$ on its $r$-local structure, the structure encoded in the $r$-local subgroup~$\pi_1^r(G)$ of~$\pi_1(G)$. The remaining structure of~$G$ is unfolded in~$\loc$ in a tree-like way, which encodes it%
	\COMMENT{}
	in conjunction with the group $\deckr = \pi_1(G)/\pi_1^r(G)$ of deck transformations of~$\loc$ over~$G$. This tree-like global structure of $\loc$ is captured by its tree of tangles~$N(\loc)$.
We refer to $N(\loc)$ and~$\deckr$ together as the $r$-global structure of~$G$:

\begin{definition}\label{def:r-globalstructure}
	The \emph{$r$-global structure} of a graph~$G$ whose $r$-local cover~$G_r$ is tangle-accessible is the pair $(N(\loc), \deckr)$, where $\deckr$ is the group of deck transformations of $p_r\colon\loc\to G$ and $N(\loc)$ is our canonical tree of tangles of~$\loc$ from \cref{def:N(G)}.

The \emph{$r$-global structure} of a group~$\Gamma$ presented with a finite set~$S$ of generators is the $r$-global structure of its Cayley graph~$\cayley{\Gamma}{S}$.

	The unique canonical \gd\ of~$G$ that {\em displays\/} its $r$-global structure $(N(\loc), \deckr)$ is that defined via~$p_r$ by the canonical \td\ which $N(\loc)$ induces on~$\loc$,%
   \COMMENT{}
   as in \cref{def:GraphDecompViaCovering}.\qed
	\end{definition}

\noindent
	\cref{def:r-globalstructure} rests on the assumption that $\loc$ is tangle-accessible. We shall prove that this is the case whenever $G$ is finite, as in \cref{main:KeyTheorem}, or a Cayley graph as in \cref{main:CayleyThm}.

\medbreak

 Note that when a \gd\ $(H,\cG)$ displays the $r$-global structure of a graph~$G$, then $H$ comes with an edge labelling that reveals additional internal structure of~$H$, which in turn displays additional connectivity aspects of~$G$ not reflected in the graph structure of~$H$ itself.

Indeed, recall that every edge~$e$ of the tree~$T$ from the \td\ $\cT = (T,\cV)$ that defines~$(H,\cG)$ as in \cref{def:GraphDecompViaCovering} maps to a separation $\alpha(e)$ of~$\loc$. This separation has some order~$k_e$. When one analyses trees of tangles in graph minor theory, one often contracts in~$T$ all those edges~$e$ whose label~$k_e$ is at least some chosen threshold~$k$. This results in a tree~$T_k$ that is a minor of~$T$, and whose edges~$e$ map to separations $\alpha(e)$ of order~$<k$. These still distinguish all the $k$-tangles of the graph decomposed, in our case~$\loc$. These trees~$T_k$ form a chain of minors of~$T$ as $k$ grows: $T_1\minor T_2\minor\dots\minor T$.

This structure is now inherited by~$H = T/\cD$. Its edges arise from ($\cD$-orbits~of) edges~$e$ of~$T$, which carry labels~$k_e$ that specify the orders of separations~$\alpha(e)$ of~$\loc$. We may interpret these as local separations of~$G$: they tell us how the parts~$G_h$ of~$\cG$ are separated locally from each other in~$G$. (In our original example from the Introduction, these local separators are the 2-sets in which two adjacent~$K^5$s overlap.) And as with those tree minors~$T_k$ of~$T$, contracting in~$H$ the edges with labels at least some chosen threshold~$k$ gives us a minor~$H_k$ of~$H$. We thus obtain a hierarchy $H_1\minor H_2\dots$ of minors of~$H$ that defines an internal structure of~$H$, which in turn displays more details of the  connectivity structure of~$G$ than $H$ does alone. This internal structure is explored further in~\cite{LocalSeps}.%
   \COMMENT{}

It is sometimes desirable that the parts~$G_h$ of the \gd\ that displays the $r$-global structure of~$G$ should inherit more properties of the parts~$\loc[V_t]$ of~$\cT$ whose projections they are than the projection map~$p_r$ can transmit. Such additional information is encoded in what we call a {\em blueprint\/} of~$G$: a family whose members are not themselves subgraphs of~$G$ but which map homomorphically to the parts~$G_h$.
We explore the notion of blueprints and the properties  of the parts of~$\cT$ that they preserve in \cite{Blueprints}. This yields, among other things, a new proof of Carmesin's local $2$-separator theorem~\cite{Local2seps} via local coverings.

While the $r$-global structure of a finite graph~$G$ is displayed by a \gd\ $(H,\cG)$ of~$G$ where $H$ is another finite graph, this decomposition of~$G$ is obtained indirectly via its usually infinite $r$-local covering. One can ask, then, whether it is possible to characterise $(H,\cG)$ directly in terms of~$G$. As a proof of concept for this type of problem, we obtain in~\cite{LocalChordal} such a direct characterisation of the $r$-global structure of the locally finite graphs that are `$r$-locally chordal', those whose parts~$G_h$ as above are complete.

\vskip 1.5 \smallskipamount\subsection{\boldmath \boldmath Road map for the proof of main result}
\label{subsec:ProofOverview}

The remainder of this paper is devoted to the proof of \cref{main:KeyTheorem}, for which we now give an outline. We shall prove \cref{main:KeyTheorem} by way of proving a stronger result, \cref{main:KeyTheoremGen}, which will include \cref{main:CayleyThm} at no extra cost. A~proof of \cref{maincor:TangleAccessibility} will also be included on the way.

As we will see in \cref{subsubsec:ToolsTreeDecNested}, every nested set of separations of a graph, all of finitely bounded order, induces a \td\ of that graph.
As seen in \cref{def:GraphDecompViaCovering} and \cref{main:TreeDecToGraphDec}, canonical \td s of local covers define canonical \gd s of the graph they cover.
So once we know that our canonical set $N(\loc)$ from \cref{def:N(G)} is defined, for which we only need that $\loc$ is tangle-accessible (\cref{thm:Entanglements}), it will induce the \td\ of~$\loc$ that defines the canonical \gd\ of~$G$ which displays its $r$-global structure.

Let us summarise these observations as the point of departure for the proofs of all our main results:%
   \COMMENT{}

\begin{reduction}\label{red1}
	It suffices to show that $\loc$ is tangle-accessible.
\end{reduction}

In order to show that $\loc$ is tangle-accessible we study another canonical%
   \COMMENT{}
   tree of tangles of~$\loc$, one whose existence was proved by Elbracht, Kneip and Teegen~\cite{InfiniteSplinters} (\cref{thm:Splinters}). In contrast to~$N(\loc)$ this tree of tangles~$N'(\loc)$ is defined for all locally finite graphs, without any accessibility requirements.
Correspondingly, the separations in~$N'(\loc)$ need not, prima facie, have finitely bounded order. However if they do, then $\loc$ is tangle-accessible. Hence:

\begin{reduction}\label{red2}
	It suffices to show that the separations in~$N'(\loc)$ have finitely bounded order.
\end{reduction}

Instead of trying to bound the order of the separations in~$N'(\loc)$ directly, we will prove that $N'(\loc)$ exhibits a certain structural property: the property of being {\em exhaustive\/}.%
	\COMMENT{}
	This means that for every infinite sequence $(A_0,B_0)<(A_1,B_1)<\dots$ in ${\vNdash}(\loc)$ we have ${\bigcap_{i\in\N}B_i=\emptyset}$. Sets of tight separations of finitely bounded order in locally finite graphs are easily seen to be exhaustive. The converse fails in general. However we will show that if $N'(\loc)$ is exhaustive, then $\deckr$ acts on~$N'(\loc)$ with only finitely many orbits. In particular, then, the separations in~$N'(\loc)$ will have finitely bounded order. Thus:

\begin{reduction}\label{red3}
	It suffices to show that $N'(\loc)$ is exhaustive.
\end{reduction}

\noindent We will prove that $N'(\loc)$ is exhaustive in two steps, as follows.

In the first step we show that $\loc$ is accessible in the traditional sense:%
	\COMMENT{}
	that the separations in~$N'(\loc)$ that efficiently distinguish two ends of~$\loc$ have finitely bounded order. We denote this set as~$N'_{\rm end}\subset N'(\loc)$. The fact that the separations in~$N'_{\rm end}$ have finitely bounded order, and hence are exhaustive because they are clearly tight,%
	\COMMENT{}
	will be the key ingredient of the remainder of our proof.

In the second step we show that $N'(\loc)$ is exhaustive.
For this let $(A_0,B_0)<(A_1,B_1)<\dots$ be any infinite sequence in~${\vNdash}(\loc)$; we have to show $\bigcap_{i\in\N}B_i=\emptyset$. Since $N'_{\rm end}$ is exhaustive, it suffices to find in~${\vNdash}_{\!\rm end}$ an infinite sequence $(C_0,D_0)<(C_1,D_1)<\dots$ which our earlier sequence dominates: one such that for every $i\in\N$ there exists $j\in\N$ with $(C_i,D_i)\le (A_j,B_j)$.%
	\COMMENT{}
	If there is no such sequence we find a part~$V'_t$ of the \td\ of~$\loc$ induced by~$N'_{\rm end}$ such that $A_j\cap B_j\subset V'_t$ for all large enough~$j$.%
\COMMENT{}
   We then enlarge $V'_t$ to a certain set~$V_t$ such that, for every $j_0$~large enough,%
	\COMMENT{}
	the subsequence $(A_j,B_j)_{j\ge j_0}$ induces an infinite sequence of tight separations of~$G[V_t]$. Moreover, $G[V_t]$~will have the following properties.

The graph~$G[V_t]$ will be connected; it will have exactly one end; and the stabiliser of~$t$ in~$\deckr$,%
   \COMMENT{}
   the group of the deck transformations of~$\loc$, will act quasi-transitively on~$G[V_t]$.
These properties combined will imply that~$G[V_t]$ has no infinite sequence of separations such as that induced by our $(A_j,B_j)_{j\ge j_0}$, a~contradiction.
Hence~$N'(\loc)$ will be exhaustive, and our proof will be complete.

Following this road map we will be able to prove Theorems~\ref{main:KeyTheorem} and~\ref{main:CayleyThm} together, by proving their following common strengthening:

\begin{theorem} \label{main:KeyTheoremGen}
	Let $G$ be a connected, locally finite, and quasi-transitive graph, and let $r > 0$ be an integer.
	Then $G$ has a unique canonical \gd\ of finite spread which displays its $r$-global structure.
\end{theorem}

\noindent Note that finite graphs, such as the $G$ in \cref{main:KeyTheorem}, are trivially quasi-transitive, and so are locally finite Cayley graphs such as the~$\cayley{\Gamma}{S}$ in \cref{main:CayleyThm}.

\medbreak

Our proof of \cref{main:KeyTheoremGen} will implement the strategy outlined above by reversing the order of its steps, as follows. We start in \cref{sec:LocalCoveringsAreAccessible} with a proof that if $G$ is locally finite and quasi-transitive, then~$\loc$ is accessible. We deduce from this in \cref{sec:TangleAccessibility} that $N'(\loc)$ is exhaustive (cf.~\cref{red3}).
In \cref{sec:ProofMainTheorem} we infer that the separations in~$N'(\loc)$ have finitely bounded order (cf.~\cref{red2}). This implies that $\loc$ is tangle-accessible (cf.~\cref{red1}), which will allow us to prove \cref{main:KeyTheoremGen}.

\bigbreak\section{\boldmath Local covers are accessible} \label{sec:LocalCoveringsAreAccessible}

\noindent
	Our aim in this section is to establish \cref{thm:G_risAccessible} below. Recall the definition of accessibility in graphs:

\begin{definition}[\cites{ThomassenWoess}]
	A graph $G$ is \emph{\accessible} if there exists an integer~$K$ such that every two ends of~$G$ are distinguished by a separation of order at most~$K$.
\end{definition}

\begin{theorem}\label{thm:G_risAccessible}
	Let $G$ be a connected, locally finite, quasi-transitive graph, and let $r > 0$ be an integer.
	Then the $r$-local cover~$\loc$ of $G$ is \accessible.
\end{theorem}

We already proved~\cref{thm:G_risAccessible} for the special cases of Cayley graphs (\cref{cor:LocCovCayleyAccessible}), and of arbitrary finite graphs by way of \gd s (\cref{cor:LocCovFinGraphAreAccessible}). These are the two cases that we would need for direct proofs of Theorems \ref{main:KeyTheorem} and~\ref{main:CayleyThm}, proceeding  exactly as we shall below. However, \cref{thm:G_risAccessible} is interesting in its own right, is considerably more general, and enables us to also prove \cref{main:KeyTheoremGen}.

For a comprehensive proof of \cref{thm:G_risAccessible}, let us observe first that $\loc$ is connected, locally finite, and quasi-transitive as soon as $G$~is:

\begin{lemma} \label{lem:G_rIsQuasitransitive}
	Let $G$ be a connected, locally finite and quasi-transitive graph, and let $r > 0$ be an integer.
	Then the $r$-local cover~$\loc$ of $G$ is connected, locally finite, and quasi-transitive.
\end{lemma}

\begin{proof}
	The $r$-local cover~$\loc$ is connected and locally finite by definition. To prove that $\loc$ is quasi-transitive, let us show that lifts $\hat u$ of~$u$ and $\hat v$ of~$v$ lie in the same orbit of~$\Aut(\loc)$ as soon as $v=\phi(u)$ for some $\phi\in\Aut(G)$. By \cref{lem:liftautomorphism}, and since local coverings are canonical, $\phi$~lifts to an automorphism~$\hat\phi$ of~$\loc$. Then $\hat v$ differs from~$\hat\phi(\hat u)$ only by a deck transformation,%
	\COMMENT{}
	so $\hat u$ and~$\hat v$ lie in the same orbit of~$\Aut(\loc)$.
\end{proof}

Our key ingredient for the proof of \cref{thm:G_risAccessible} is the following result of Hamann:

\begin{lemma}[\cite{hamann2018accessibility}*{Corollary 3.3}]\label{thm:HamannAccessible}
	Every locally finite quasi-transitive graph whose binary cycle space is generated by cycles of bounded length is accessible. 
	\end{lemma}

\begin{proof}[Proof of \cref{thm:G_risAccessible}]
	By \cref{thm:integralCycleSpace}, the binary cycle space of~$\loc$ is generated by its cycles of length at most~$r$. The theorem now follows from \cref{thm:HamannAccessible}, which we may apply to~$\loc$ by \cref{lem:G_rIsQuasitransitive}.
	\end{proof}

\bigbreak\section{\boldmath Tools and lemmas for the proof of \cref{main:KeyTheorem}} \label{sec:TangleAccessibility}

\noindent
	Throughout this section, $G$~will be a connected locally finite infinite graph. Our main goal is to prove the following key lemma. Note that if $G$ satisfies its premise, it must be accessible.

\begin{lemma} \label{thm:NestedSetIsTDGen}
	Let $N$ be a nested set of tight finite-order separations of~$G$, and suppose that $N$ is $\Gamma$-canonical for some group~$\Gamma$ that acts quasi-transitively on~$G$.
	Suppose further that there exists an integer~$K$ such that some $\Gamma$-canonical set $N_{\rm end} \subseteq N$ of separations of order $<K$ distinguishes all the ends of~$G$.
	Then $N$ is exhaustive and induces a $\Gamma$-canonical \td\ of~$G$.\looseness=-1
\end{lemma}

\vskip 1.5 \smallskipamount\subsection{\boldmath Tools: Tree-decompositions and nested sets of separations} \label{subsubsec:ToolsTreeDecNested}

Let $\cT = (T, \cV)$ be a \td\ of~$G$, and let $\Gamma$ be a group acting on~$G$.
Recall from \cref{subsec:GraphDecProperties} that every edge $t_1 t_2$ of~$T$ induces a separation $\{A_{t_1}, A_{t_2}\}$ of~$G$, and that each of its two orientations $(t_1,t_2)$ induces the oriented separation $\alpha_\cT(t_1, t_2) = (A_{t_1}, A_{t_2})$ of~$G$.%
	\COMMENT{}
If we orient distinct edges $t_1 t_2$ and $t_2 t_3$ of~$T$ each from $t_i$ to~$t_{i+1}$, then $\alpha_\cT(t_1, t_2) \leq \alpha_\cT(t_2, t_3)$ in the partial ordering~$\leq$ of oriented separations (see \cref{subsec:TanglesBackground}).

By definition of the sets $A_{t_1}$ and~$A_{t_2}$, the separations of~$G$ that $\cT$ induces in this way are pairwise nested.\penalty-200\ If $\cT$ is regular, they are proper separations. If $\cT$ is $\Gamma$-canonical for some group~$\Gamma$ as defined in \cref{subsec:treedecomptographdecomp}, then the nested set~$N$ of separations of~$G$ which $\cT$ induces is $\Gamma$-canonical as defined in \cref{subsec:TanglesBackground}.


Conversely, let us try to construct from a given nested set~$N$ of proper separations of~$G$ a~regular \td\ of~$G$ that induces~$N$. We begin by constructing, from~$N$ alone, the nodes of a potential  decomposition tree. A~\emph{splitting star} of~$N$ is any set $\sigma\subset \vN$ that consists of all the maximal separations of some consistent orientation~$O$ of~$N$ such that for every $\vr \in O$ there exists $\vs \in \sigma$ with $\vr \le \vs$. Every separation in $\vN$ lies in at most one splitting star of~$N$. If $N$ is induced by a \td\ $\cT=(T, \cV)$ of~$G$, as is our aim, then every $\vs\in\vN$ does lie in a splitting star: it will have the form $\vs = \alpha_\cT(t',t)$ for some oriented edge $(t',t)$ of~$T$, and $\vs$ will lie in splitting star of~$N$ formed by the separations $\alpha_\cT (t'',t)$ where $t''$ varies over the neighbours of~$t$ in~$T$.%
	\footnote{In general, when $N$ is infinite, the elements of~$\vN$ need not lie in a splitting star of~$N$;%
	\COMMENT{}
	see~\cites{ASS,TreelikeSpaces}.}

Taking our cue from the above observation, let us define a graph~$T$ from~$N$ as follows.
Take the splitting stars of~$N$ as the nodes of~$T$, and make two nodes adjacent if they contain, as splitting stars of~$N$, distinct orientations of some element of~$N$.
For each splitting star $t \in V(T)$ let $V_t := \bigcap_{(A,B) \in t} B$, and put~$\cV = (V_t)_{t \in T}$. Let us say that $N$ \emph{induces} this pair~$(T, \cV)$.

As indicated, $(T, \cV)$ need not be a \td\ of~$G$ when $N$ is infinite. Indeed, while $T$ is always acyclic it need not be connected, and the family $\cV$ might violate \cref{ax:TreeDecomp1}~\cite{InfiniteSplinters}*{Example 4.9}. We shall mention some sufficient conditions for $(T,\cV)$ to be a \td\ later.

If $\cT := (T,\cV)$ is a \td, it will induce the set~$N$ of separations we started with, in that $N = \{\,\alpha_\cT(e) : e\in E(T)\,\}$ (as was our aim), with $\alpha_\cT$ an order isomorphism between the oriented edges of~$T$ and~$\vN$~\cite{InfiniteSplinters}*{Proof of Lemma 2.7}.%
	\COMMENT{}
	If $N$ is $\Gamma$-canonical for some group~$\Gamma$, then $\cT$ will be a $\Gamma$-canonical \td. Indeed, if $N$ is $\Gamma$-canonical, then every $g \in \Gamma$ maps splitting stars of~$N$ to splitting stars of~$N$.%
   \COMMENT{}
   Since $g \act (A, B) = (g \act A\>,\, g \act B)$ for $g \in \Gamma$ and $(A, B)\in\vN$, this action of $\Gamma$ on~$T$ agrees with that in our definition of $\Gamma$-canonicity of $(T, \cV)$ in \cref{subsec:treedecomptographdecomp}.

A~strictly increasing infinite sequence $(A_i, B_i)_{i \in \N}$ of oriented separations of~$G$ is \emph{exhaustive} if $\bigcap_{i \in \N} B_i = \emptyset$. A set~$S$ of separations is \emph{exhaustive} if every such sequence in~$\vS$ is exhaustive. Any nested set of separations induced by a \td\ of~$G$ of finite spread is clearly exhaustive. We shall need the following converse:

\begin{lemma}[\cite{InfiniteSplinters}*{Lemma 2.7}] \label{lem:ExhaustiveImpliesTreeDecomp}
	Let $N$ be a nested set of proper separations of~$G$.
	If $N$ is exhaustive, then the pair $(T, \cV)$ induced by $N$ is a \proper\ \td\ of~$G$.
	If a group~$\Gamma$ acts on~$G$ so that $N$ is $\Gamma$-canonical, then this \td\ $(T, \cV)$ is $\Gamma$-canonical.\footnote{The last sentence, which we already verified, is not explicitly stated in~\cite{InfiniteSplinters}.}
\end{lemma}

The task of constructing \td s thus reduces to that of constructing exhaustive nested sets of separations. Let us collect a few lemmas that will help with this task. When $(A_i, B_i)_{i \in \N}$ and $(C_j, D_j)_{j \in \N}$ are sequences of separations such that for every $i$ there exists~$j$ such that $(A_i, B_i) \leq (C_j, D_j)$, we say that $(C_j, D_j)_{j \in \N}$ \emph{dominates}~$(A_i, B_i)_{i \in \N}$.

\begin{lemma}\label{lem:DominationPrinciple}
	If $(A_i, B_i)_{i \in \N}$ and $(C_j, D_j)_{j \in \N}$ are sequences of separations such that $(C_j, D_j)_{j \in \N}$ dominates $(A_i, B_i)_{i \in \N}$, then $\bigcap_{i \in \N} B_i \supseteq \bigcap_{j \in \N} D_j$. In particular, if $(A_i, B_i)_{i \in \N}$ is exhaustive then so is $(C_j, D_j)_{j \in \N}$. \qed
\end{lemma}

All the separations from which we shall construct \td s will be tight, and hence proper.
\cref{lem:BoundedOrder} below, which is implicit in~\cite{InfiniteSplinters},%
	\COMMENT{}
	shows that a nested set of tight separations is exhaustive if these have finitely bounded order. For its proof we need the following well-known fact about tight separations:

\begin{lemma}[\cite{ThomassenWoess}*{Proposition 4.2}] \label{lem:TightSepFinite}
	Every vertex of~$G$ lies in only finitely many separators of tight separations of order less than~$K$, for each~$K \in \N$.
\end{lemma}

\begin{lemma}[\cite{InfiniteSplinters}] \label{lem:BoundedOrder}
	Let $K$ be an integer, and let $(A_i, B_i)_{i \in \N}$ be a strictly increasing sequence of tight separations of~$G$, all of order $<K$. Then $(A_i, B_i)_{i \in \N}$ is exhaustive.
\end{lemma}
\begin{proof}
	By \cref{lem:TightSepFinite}, every vertex of $G$ lies in $A_i \cap B_i$ for only finitely many~$i \in \N$.
	These sets $A_i \cap B_i$ are finite, so applying this to every vertex $v \in A_i \cap B_i$ shows that $(A_i \cap B_i) \cap (A_j \cap B_j) = \emptyset$ for all large enough~$j$. We can thus find an infinite subsequence of $(A_i, B_i)_{i \in \N}$ whose separators are disjoint.%
	\COMMENT{}
	But this implies that $\bigcap_i B_i = \emptyset$: otherwise $G$ contains an $A_0\,$--$\,\bigcap_i B_i$ path that meets all these infinitely many disjoint separators, a contradiction.
	Hence $(\vsi)_{i\in\N}$ is exhaustive.
\end{proof}

\vskip 1.5 \smallskipamount\subsection{\boldmath Proof of \cref{thm:NestedSetIsTDGen} for 1-ended graphs} \label{sec:OneEndedCase}

Let us now prove \cref{thm:NestedSetIsTDGen} for 1-ended graphs. In fact, we show the following stronger result:

\begin{lemma} \label{thm:NestedIsTDOneEnded}
	Assume that $G$ has only one end. Let $N$ be a nested set of tight separations of~$G$, all of finite order. Suppose that $N$ is $\Gamma$-canonical for some group~$\Gamma$ acting quasi-transitively on~$G$. Then $\vN$ contains no strictly increasing infinite sequence.
\end{lemma}

We shall need the following fact%
   \ifArXiv\ \else, proved in~\cite{GraphDecArXiv}, \fi
   about 1-ended, connected, locally finite graphs:

\begin{lemma} \label{lem:OneEndedFiniteOrder}
	If $G$ has only one end, then every finite-order separation of~$G$ has a finite and an infinite side.\looseness=-1
\end{lemma}

\ifArXiv
\begin{proof}
	Let $\{A,B\}$ be any finite-order separation of~$G$.
	At least one of $A,B$ must be infinite.
	Since $A \cap B$ is finite, and $G$ is locally finite and connected, $G-A$~and $G-B$ each have only finitely many components. If they are~both infinite, they each contain an infinite component. By K\"onig's infinity lemma~\cite{DiestelBook25}*{Proposition~8.2.1}, each of these components contains a ray. These rays, then, are separated by the finite set $A\cap B$, which contradicts our assumption that $G$ has only one end.
\end{proof}
\fi

Our next lemma shows that, in a strictly increasing infinite sequence of separations, every separation is oriented towards its infinite side:
\begin{lemma} \label{lem:OmegaChainFiniteSmall}%
	\COMMENT{}
	In any strictly increasing sequence $(A_i, B_i)_{i \in \N}$ of separations of~$G$, all $B_i$ are infinite.
\end{lemma}
\begin{proof}
	Suppose some~$B_i$ is finite; we may assume that $i = 0$.
	Since $B_0 \supseteq B_i$ for all~$i$, all the $B_i$ are finite.
	Since $(A_i, B_i)_{i \in \N}$ is a strictly increasing sequence, we have for each~$i$ either $A_i\subsetneq A_{i+1}$ or $B_i\supsetneq B_{i+1}$.
	As $B_0$ is finite, we have $B_i\supsetneq B_{i+1}$ for only finitely many~$i$, so $A_i\subsetneq A_{i+1}$ for infinitely many~$i$. For all these~$i$ the non-empty sets $A_{i+1} \setminus A_i$ are disjoint and contained in the finite set~$B_0$,%
	\COMMENT{}
	a~contradiction.
\end{proof}

Recall that the set $N$ in \cref{thm:NestedIsTDOneEnded}, which we are seeking to analyse, is $\Gamma$-canonical. The following lemma shows how separations with a finite side interact with their images under certain automorphisms:

\begin{lemma} \label{lem:FiniteAutomCross}%
	\COMMENT{}
	Let $\{A, B\}$ be a separation of $G$ with $A$~finite. Consider any automorphism $\phi$ of~$G$ and a vertex $v \in A \cap B$ with a neighbour $w\in B\setminus A$ such that $\phi(v) \in A \setminus B$. Then $\{A, B\}$ and $\{\phi(A), \phi(B)\}$ cross.\looseness=-1
\end{lemma}

\begin{proof}
	Let us show first that neither $A \subseteq \phi(A)$ nor~$\phi(A) \subseteq A$.
	To see this, it is enough to show $\phi(A) \neq A$, since $A$ is finite. If $\phi(A) = A$, then $\phi(v) \in A \setminus B$ is a neighbour of $\phi(w) \in \phi(B\setminus A) = B \setminus A$, which contradicts the fact that $\{A,B\}$ is a separation.
	
As $B$ is infinite, we also have $\phi(B) \nsubseteq A$.
	So if $\{A, B\}$ and $\{\phi(A), \phi(B)\}$ were nested, we would have $(A, B) \leq (\phi(B), \phi(A))$, and thus $B \supseteq \phi(A)$.
	But $\phi(v) \in \phi(A) \setminus B$ by assumption, a~contradiction.
\end{proof}

\noindent
We remark that \cref{lem:FiniteAutomCross} has no analogue for finite graphs, or for separations with two infinite sides%
	\ifArXiv
. Indeed, for finite graphs consider a cycle of odd length at least~$5$ and an arbitrary tight separation of order $2$ together with an automorphism which fixes precisely one of the vertices in the separator.
For separations with two infinite sides, consider a double ray and an arbitrary separation of order $1$ together with any automorphism which shifts the double ray.
	\else
	~\cite{GraphDecArXiv}.\medbreak
	\fi

With these three lemmas at hand, we are now ready to prove \cref{thm:NestedIsTDOneEnded}.

\begin{proof}[Proof of \cref{thm:NestedIsTDOneEnded}]
	Suppose, for a contradiction, that $(A_i, B_i)_{i \in \N}$ is a strictly increasing sequence in~$\vN$. Then $A_i\setminus B_i\subset A_j\setminus B_j$ for all $i<j$, and by \cref{lem:OneEndedFiniteOrder,lem:OmegaChainFiniteSmall} all the $A_i$ are finite.
	
	We first show that there exists $n \in \N$ such that $A_i \cap B_i \subseteq A_{i+1} \cap B_{i+1}$ for all~$i \ge n$. If not, there are infinitely many~$i$ such that $A_i \cap B_i$ contains a vertex $v_i\in A_{i+1}\setminus B_{i+1}$.%
	\COMMENT{}
   By our initial observation,%
   \COMMENT{}
   these~$v_i$ are distinct for different~$i$. Since $\Gamma$ acts quasi-transitively on~$G$, there exist $i<j$ among these with $v_i = g \act v_j$ for some~$g \in \Gamma$.
	Then $\{A_j, B_j\}$, together with $v_j \in A_j \cap B_j$%
	\COMMENT{}
	and the automorphism~$\phi$ of~$G$ defined by~$g$, satisfies the premise of \cref{lem:FiniteAutomCross}, since $g \act v_j = v_i \in A_j \setminus B_j$ and $\{A_j, B_j\}$ is tight.
	Since $N$ is $\Gamma$-canonical, applying \cref{lem:FiniteAutomCross} yields a contradiction to $N$ being nested.
	
	We next show that there exists $m\ge n$ such that $A_i \cap B_i = A_{i+1} \cap B_{i+1}$ for all~$i \ge m$. If not, then $A_i \cap B_i \subsetneq A_{i+1} \cap B_{i+1}$ for infinitely many $i\ge n$; pick a subsequence $(A_{i_j}, B_{i_j})_{j \in \N}$ with $i_0 = n$ witnessing this.
	Since the separations $(A_{i_j}, B_{i_j})$ are tight, every $G[A_{i_j} \setminus B_{i_j}]$ has a component $K_{i_j}$ with neighbourhood~$A_{i_j} \cap B_{i_j}$. Then $K_{i_j} \cap K_{i_{j+1}}= \emptyset$ for all~$j$, since $K_{i_j}$ is also a component of~$G[A_{i_{j+1}}\setminus B_{i_{j+1}}]$%
   \COMMENT{}
   but distinct from~$K_{i_{j+1}}$, since the two have different neighbourhoods.%
	\COMMENT{}
   As $G$ is connected, there exists $v \in A_n \cap B_n$. This~$v$ lies in $A_{i_j}\cap B_{i_j}$ for every~$j$, and hence has a neighbour in~$K_{i_j}$. As these neighbours are distinct, this contradicts our assumption that $v\in G$ has finite degree.%
   \COMMENT{}

The~$(A_i,B_i)$ with $i\ge m$ form a strictly increasing sequence of separations with the same separator~$X$. This is impossible since, as $G$ is connected and locally finite, the graph $G-X$ has only finitely many components.%
	\COMMENT{}
	\end{proof}

\vskip 1.5 \smallskipamount\subsection{\boldmath Proof of \cref{thm:NestedSetIsTDGen}} \label{sec:NestedSetisTDArbitraryGraphs}

In the remainder of this section we prove \cref{thm:NestedSetIsTDGen}. Let $N$, $N_{\rm end}$, $\Gamma$ and~$K>0$ be as stated in the lemma. Our goal is to show that $N$ is exhaustive: that every strictly increasing sequence $(A_i,B_i)_{i\in\N}$ of separations in~$\vN$ satisfies $\bigcap_i B_i = \emptyset$.%
	\COMMENT{}

We already know from \cref{lem:BoundedOrder} that $N_{\rm end}$ is exhaustive, since by assumption the separations in~$N_{\rm end}$ have bounded order. If we can show that our sequence $(A_i,B_i)_{i\in\N}$ dominates a sequence in~$\vN_{\rm end}$, it~will also be exhaustive (\cref{lem:DominationPrinciple}). Essentially, this amounts to showing that our sequence is interlaced with a sequence in~$\vN_{\rm end}$: that at most finitely many $(A_i,B_i)$ have the same $\le$-relationships to all the elements of~$\vN_{\rm end}$.%
   \COMMENT{}
   Our plan is to show that if there were infinitely many such~$(A_i,B_i)$ then these would, essentially, disect the same 1-ended subgraph of~$G$,%
   \COMMENT{}
   which would contradict \cref{thm:NestedIsTDOneEnded}.

In graph terms this means that, essentially,%
	\COMMENT{}
	for the \td\ $\cT' = (T,\cV')$ of~$G$ that $N_{\rm end}$ induces (see below), where $\cV' = {(V'_t)}_{t\in T}$ say, there is no $t\in T$ such that infinitely many $(A_i,B_i)$ differ only in how they cut up~$V'_t$, because this would contradict \cref{thm:NestedIsTDOneEnded}. This will not work directly, since the graphs~$G[V'_t]$ need not satisfy the conditions on~$G$ laid down in the premise of \cref{thm:NestedIsTDOneEnded}.%
	\COMMENT{}
	However we shall be able to extend the~$V'_t$ to sets~$V_t$ that do satisfy these conditions, and then apply \cref{thm:NestedIsTDOneEnded} to the graphs~$G[V_t]$ to complete our proof.

To get started, recall from \cref{lem:BoundedOrder,lem:ExhaustiveImpliesTreeDecomp} that $N_{\rm end}$ does induce the required \td:

\begin{lemma} \label{lem:NendInducesTreeDec}
	$N_{\rm end}$ is exhaustive, and the \td\ $\cT'\!$ of~$G$ it induces is $\Gamma$-canonical.\qed%
	\COMMENT{}
\end{lemma}

\noindent
	As $N_{\rm end}$ distinguishes all the ends of~$G$, each part of $\cT'$ meets rays from at most one end of~$G$ infinitely.%
	\COMMENT{}

\medbreak

Since $N_{\rm end}$ is $\Gamma$-canonical, the group $\Gamma$ acts on the decomposition tree $T$ of~$\cT'$ (see \cref{subsubsec:ToolsTreeDecNested}).
Since $G$ is quasi-transitive under the action of~$\Gamma$, there are only finitely many $\Gamma$-orbits of tight separations of~$G$ of order less than~$K$ \cite{ThomassenWoess}*{Proposition 4.2}. The action of~$\Gamma$ on~$T$,%
	\COMMENT{}
	therefore, is quasi-transitive too:%
	\COMMENT{}

\begin{lemma} \label{lem:DecTreeFiniteOrbits}
	$\Gamma$ acts on $E(T)$ with finitely many orbits. \qed
\end{lemma}

Let $st$ be any edge of~$T$.
Since $G$ is locally finite and $V'_{st} := V'_s \cap V'_{t}$ is finite,%
	\COMMENT{}
	the graph $G-V'_{st}$ has only finitely many components. Hence there exists an integer $d_{st} \geq 1$ such that the $d_{st}$-ball $B_{G}(V'_{st}, d_{st})$ around~$V'_{st}$ (see \cref{subsec:ballpres}) meets every component $C$ of $G-V'_{st}$ in a set that includes all the shortest paths in~$C$ between neighbours of~$V'_{st}$ in~$C$.%
	\COMMENT{}
	 (There are only finitely many such paths, since $G$ is locally finite.) By \cref{lem:DecTreeFiniteOrbits}, $\Gamma$ has finitely many orbits on~$E(T)$. Let us assume that the integers~$d_{st}$ were all chosen as small as possible; then they coincide for any two edges~$st$ in the same orbit.
Thus, the maximum~$d$ over all $d_{st}$ for $st \in E(T)$ is well-defined; let $V_t := V(B_G(V'_t,d))$ and $\cV := (V_t)_{t\in T}$.

\begin{lemma} \label{lem:TransformedTreeDec}
	$\cT = (T,\cV)$ is a $\Gamma$-canonical \td\ of $G$ of finitely bounded adhesion such that the following statements hold for every $t \in T$:
	\begin{enumerate}
		\item $G[V_t]$ is connected;
		\item the stabiliser~$\Gamma_t$ of~$t$ under the action of~$\Gamma$ on~$T$ acts quasi-transitively on~$G[V_t]$;
		\item $G[V_t]$ is either finite or 1-ended.
	\end{enumerate}
\end{lemma}

\noindent
	The proof of \cref{lem:TransformedTreeDec} closely follows~\cite{hamann2018stallings};%
	\COMMENT{}
	details are given in~%
\ifArXiv
	\cref{app:ProofsTreeDecHamann} below.
\else
	\cite{GraphDecArXiv}.
\fi

\medbreak

Recall that our plan is to show that our sequence $(A_i,B_i)_{i\in\N}$ in $\vN$ dominates a suitable sequence in~$\vN_{\rm end}$.
To achieve this, it will suffice to show that no part~$V'_t$ of~$\cT'$ contains all the separators from a final segment $(A_i,B_i)_{i\ge i_0}$ of our sequence.%
\COMMENT{}
We will prove this by  applying \cref{thm:NestedIsTDOneEnded} to the corresponding part~$G[V_t]$ of~$\cT$.\looseness=-1

To apply \cref{thm:NestedIsTDOneEnded} in this way, we need to show that a separation of~$G$ whose separator is contained in~$V_t'$ induces a tight separation of~$G[V_t]$. It does, because the graph~$G[V_t]$ behaves like a torso of~$G[V'_t]$ (see~\cite{DiestelBook25}*{Ch.12.3} if desired):

\begin{lemma} \label{lem:torso}
	Let $t$ be a node of~$T$ and $X\subset V'_t$.
	Then the map $C\mapsto C\cap G[V_t]$ is a bijection between the components of $G-X$ and those of~$G[V_t]-X$.
	Moreover, $N_{G[V_t]}(C \cap G[V_t]) = N_G(C)$ for all such~$C$.
\end{lemma} 


\begin{proof}
	Let $C$ be any component of~$G-X$. Since $G$ is connected, $C$~contains a neighbour of~$X$. All the neighbours of~$X$ in~$C$ lie in~$V_t$, since $N_G(X) \subseteq B_G(V'_t,1) \subseteq B_G(V'_t, d) = V_t$. Hence $C$ meets some com\-po\-nent~$D$ of~$G[V_t]-X$. We shall prove that $D$ is the only component of~$G[V_t]-X$ that $C$ meets; then clearly $D = C\cap G[V_t]$,%
	\COMMENT{}
	which establishes the desired bijection.

Since $G[V_t]$ is connected (\cref{lem:TransformedTreeDec}), every component of~$G[V_t]-X$ contains a neighbour of~$X$. It thus suffices to show that all the neighbours $u,v$ of~$X$ in~$C$ lie in the same component of~$G[V_t]-X$ (our~$D$).%
	\COMMENT{}
	We prove this by showing that every shortest $u$--$v$ path~$P$ in~$G-X$ lies in~$G[V_t]$. Note that $P\subset C$.%
	\COMMENT{}

Let $x$ be a neighbour of~$u$ in~$X$, and $y$ a neighbour of~$v$ in~$X$, possibly $x=y$. Then $xuPvy$ is a concatenation of walks in~$G$%
	\COMMENT{}
	that either lie in~$G[V'_t]$ or have their first and last vertex in~$V'_t$ and all their inner vertices (at least one) in the same component of~$G-V'_t$. Let $Q = x'\dots y'$ be any subwalk of~$xuPvy$ of the latter type if one exists (which will be a path or the entire walk~$xuPvy$, in which case that is a cycle),%
	\COMMENT{}
	and let~$C'$ be the component of $G-V'_t$ that contains $Q-x'-y'$. It will suffice for our proof of $P\subset G[V_t]$ to show that $Q\subset G[V_t]$.%
	\COMMENT{}

As $X\subset V'_t$ and $C'$ meets~$C$ (in~$Q - x' - y'\ne\emptyset$, which lies in~$P\subset C$),%
	\COMMENT{}
	we have $C'\subset C$. Since $\cT'$ is a \td, $V(C')$~is contained in~$\bigcup_{t'\in T'} V'_{t'}$ for some component $T'$ of~$T-t$. Let $s\in T'$ be the neighbour of~$t$ in that component. Then all the neighbours of~$C'$ in~$V'_t$ lie in~$V'_{st}\subset V'_t$, and so the component of $G-V'_{st}$ that contains~$C'$ (which exists because $V'_{st}\subset V'_t$) in fact equals~$C'$.

Thus, $Q - x' - y'$ is a path in a component of $G-V'_{st}$%
	\COMMENT{}
	between neighbours of~$V'_{st}$ (possibly identical).%
	\COMMENT{}
	Since it is a subpath of~$P$, it is also a shortest such path. Hence $Q\subset G[V_t]$ by the definition of~$V_t$.
\end{proof}

To complete our proof of \cref{thm:NestedSetIsTDGen} we need one more definition. Let $\vr,\vs$ be orientations of proper%
	\COMMENT{}
	separations of~$G$, such as those in~$\vN$. If $r=s$ we say that $\vr$ {\em points towards\/}~$\vs$ if $\vr=\vs$. If $r\ne s$ we say that $\vr$ \emph{points towards}~$\vs$ (and~$s$) if either $\vr\le\vs$ or $\vr\le\sv$. Note that if $r\ne s$ are nested, they have unique orientations $\vr$ and~$\sv$ that point towards each other, with $\vr\le\vs$ and $\rv\ge\sv$. Also, since $r$ and~$s$ are proper, we cannot have both $\vr\le\vs$ and $\vr\le\sv$.

\begin{proof}[Completion of proof of \cref{thm:NestedSetIsTDGen}]
	Recall that we have a sequence $(A_i,B_i)_{i\in\N}$ of separations in~$\vN$, which we want to show to be exhaustive.%
	\COMMENT{}
	Let us abbreviate $(A_i,B_i)$ as~$\vsi$. The tools at our disposal are that $N_{\rm end}\subset N$ is exhaustive (\cref{lem:NendInducesTreeDec}), our \td\ of~$G$ into finite or 1-ended parts provided by \cref{lem:TransformedTreeDec}, and \cref{thm:NestedIsTDOneEnded} by which 1-ended graphs contain no infinite sequences such as $(\vsi)_{i\in\N}$ at all.\looseness=-1

We need three more steps.
	We first establish a partition of $\vN$ defined by the splitting stars~$\sigma$ of $N_{\rm end}$. Recall that these correspond to the nodes $t$ of~$T$ (see \cref{subsubsec:ToolsTreeDecNested}); the corresponding partition class~$P_\sigma$ of~$\vN$ will consist, very roughly, of those separations in~$\vN$ that lie in~$\sigma$ or separate~$V_t$.
	Next we show that for our proof that $(\vsi)_{i\in\N}$ is exhaustive it suffices to show that it contains separations from infinitely many of these partition classes~$P_\sigma$.
	Finally, we check that $(\vsi)_{i\in\N}$ does meet infinitely many~$P_\sigma$.
	
Since $N_{\rm end}$ is exhaustive, its splitting stars form a partition of~$\vN_{\rm end}$: every element of~$\vN_{\rm end}$ lies in exactly one splitting star \cite{TreelikeSpaces}*{Lemma~2.4}.%
	\COMMENT{}
	Let us extend this partition of~$\vN_{\rm end}$ to one of~$\vN$, as follows. For every splitting star~$\sigma$ of~$N_{\rm end}$ let
	$$P_\sigma :=\ \big\{\vs \in \vN : \text{all the elements of } \sigma \text{ point towards } \vs\big\}.$$%
	\COMMENT{}
	Note that $P_\sigma\cap\vN_{\rm end} = \sigma$.%
	\COMMENT{}

Let us show that these~$P_\sigma$ form a partition of~$\vN$. As the separations in~$N$ are nested, every $\vs\in \vN$ defines a consistent orientation of~$N_{\rm end}$:%
	\COMMENT{}
	$$O\!_\vs := \big\{\vr\in\vN_{\rm end} : \vr \text{ points towards } \vs\big\};$$%
	\COMMENT{}
	let $\sigma\!_\vs$ denote the set of its maximal elements. Since $N_{\rm end}$ is exhaustive and $s$ is proper, $O\!_\vs$ contains no infinite strictly increasing sequence. Hence every $\vr\in O\!_\vs$ lies below some maximal element of~$O\!_\vs$;%
	\COMMENT{}
	 thus, $\sigma\!_\vs$~is a splitting star of~$N_{\rm end}$. So $P_{\sigma_{\vec s}}$ is defined, and it clearly contains~$\vs$. But $\vs$ cannot lie in any other~$P_\sigma$. Indeed, consider a splitting star $\sigma\ne\sigma\!_\vs$ of~$N_{\rm end}$. The orientation $O$ of~$N_{\rm end}$ of which $\sigma$ is the set of maximal elements then orients some separation $r\in N_{\rm end}$ differently from the way $O\!_\vs$ does.%
	\COMMENT{}
	Hence if $\vs\in P_\sigma\cap P_{\sigma_{\vec s}}$ then both orientations of~$r$ lie below an element of~$\sigma$ or~$\sigma\!_\vs$,%
	\COMMENT{}
	and hence both point towards~$\vs$.%
	\COMMENT{}
	This is impossible, since $s$ is proper.%
	\COMMENT{}

Next we show that $(\vsi)_{i \in \N}$ is exhaustive if it meets infinitely many~$P_\sigma$.
	By \cref{lem:DominationPrinciple} we may assume that the $\vsi$ lie in distinct~$P_\sigma$.%
	\COMMENT{}
	For each $i \in \N$ let $\sigma_i$ be such that $\vsi\in P_{\sigma_i}$. Then $O_i := O\!_{\vsi}$ is the consistent orientation of~$N_{\rm end}$ of which $\sigma_i$ is the set of maximal elements.%
	\COMMENT{}
	Let $r_i\in N_{\rm end}$ be a separation oriented differently by $O_i$ and~$O_{i+1}$.%
	\COMMENT{}
	Then one of these, $\vri$~say, points towards~$\vsiplus$,%
	\COMMENT{}
	while the other, $\rvi$, points towards~$\vsi$. Since $\vsi < \vsiplus$, this is possible only if $\vsi < \vri \leq \vsiplus$.%
	\COMMENT{}
	(The first inequality is strict, because otherwise $\rvi$ would point towards its own inverse~$\vri=\vsi$, which is impossible.)%
	\COMMENT{}
	Hence the sequence $(\vri)_{i \in \N}$ is strictly increasing too, and is dominated by~$(\vsi)_{i \in \N}$. Since $(\vri)_{i \in \N}$ consists of separations from $\vN_{\rm end}$ and is thus exhaustive by \cref{lem:NendInducesTreeDec}, our sequence $(\vsi)_{i \in \N}$ is exhaustive by \cref{lem:DominationPrinciple}.
	
	It remains to check that $(\vsi)_{i \in \N}$ does meet infinitely many~$P_\sigma$. To prove this we show that each $P_\sigma$ contains only finitely many~$\vsi$. Recall that $\cT' = (T,\cV')$ was the \td\ of~$G$ induced by~$N_{\rm end}$, and that our group~$\Gamma$ of automorphisms of~$G$ acts on~$T$. As shown in \cref{subsubsec:ToolsTreeDecNested}, the splitting stars~$\sigma$ of~$N_{\rm end}$ correspond to the nodes~$t$ of~$T$. Given such a pair $\sigma$ and~$t$, let $N_t \subseteq N$ be the set of all separations with an orientation in~$P_\sigma$.%
	\COMMENT{}
	Note that the stabiliser~$\Gamma_t$ of~$t$ under the action of $\Gamma$ on~$T$ consists precisely of those elements of~$\Gamma$ that fix~$\sigma$ as a set.

Let us show that $N_t$ is $\Gamma_t$-canonical:%
	\COMMENT{}
	that for every $g\in\Gamma_t$ and every $s\in N_t$ we have $g\act s\in N_t$. By definition of $N_t$ and~$P_\sigma$ the latter means that every $\vrdash\in\sigma$ points towards~$g\act s$: that $s' := g\act s$ has an orientation $\vsdash\ge\vrdash$.%
	\COMMENT{}
	Let $\vr := g^{-1}\act\vrdash$. As $g\in\Gamma_t$ we have $\vr\in\sigma$; thus, $\vr$ also points towards~$s$, say $\vr\le\vs$. But then%
	\COMMENT{}
	$\vrdash = g\act\vr\le g\act\vs =: \vsdash$%
	\COMMENT{}
	as required.

Let us turn $N_t$, a set of separations of~$G$, into a $\Gamma_t$-canonical set~$M_t$ of tight separations of~$G[V_t]$. By the definition of $\cT'$, the part $V_t'$ is equal to~$\bigcap_{(C,D) \in \sigma} D$.
	As for every $\{A, B\} \in N_t$ and every $(C,D)\in\sigma$ either $A$ or~$B$ is a subset of~$D$,%
   \COMMENT{}
   we thus have $A \cap B \subseteq V'_t$. Every oriented separation $(A, B) \in \vN_t$%
	\COMMENT{}
	induces an oriented separation $(A_t, B_t) := (A \cap V_t, B \cap V_t)$ of~$G[V_t]$.
	By \cref{lem:torso} one easily shows that these are distinct for different~$(A,B)$, that they are tight because the $(A,B)$ are tight, and that for $(A, B) < (C, D)$ in~$\vN_t$ we have ${(A_t, B_t) < (C_t, D_t)}$.%
	\COMMENT{}
	Let
	$$M_t := \big\{ \{A_t, B_t\} \colon \{A, B\}\in N_t \big\}.$$

$M_t$~is a $\Gamma_t$-canonical set of tight separations of~$G[V_t]$.
	If $G[V_t]$ is finite then so is~$M_t$, and in particular $\vM_t$ contains no strictly increasing sequence.
	Otherwise $G[V_t]$ is a connected locally finite 1-ended graph that is quasi-transitive under the action of~$\Gamma_t$, by \cref{lem:TransformedTreeDec}.
	Hence, by \cref{thm:NestedIsTDOneEnded}, $\vM_t$~contains no strictly increasing infinite sequence.
	Since any strictly increasing sequence in $\vN_t$ induces a strictly increasing sequence in~$\vM_t$, by \cref{lem:torso} as noted earlier, there is thus no strictly increasing infinite sequence in~$\vN_t$ either. This completes our proof that $P_\sigma\subset\vN_t$ contains only finitely many~$\vsi$, and hence that our sequence $(\vsi)_{i\in\N}$ is exhaustive.
	By \cref{lem:ExhaustiveImpliesTreeDecomp}, $N$~induces a $\Gamma$-canonical \td\ of~$G$, as required in the statement of \cref{thm:NestedSetIsTDGen}.
	\end{proof}

\bigbreak\section{\boldmath Proof of the main result}\label{sec:ProofMainTheorem}

\noindent
	In this section we combine all our lemmas to prove \cref{main:KeyTheoremGen}~-- the more general version of \cref{main:KeyTheorem} that also includes \cref{main:CayleyThm}~-- and derive some corollaries of the proof.
The unique \gd\ of the given graph~$G$ whose existence is asserted in \cref{main:KeyTheoremGen}, the decomposition that displays its $r$-global structure (if it exists), was specified in \cref{def:r-globalstructure}. It is defined (via the projection $p_r\colon\loc\to G$) by the \td~$\cT$ of~$\loc$ that is induced by its canonical tree of tangles~$N(\loc)$ from \cref{def:N(G)} if this exists.%
	\COMMENT{}

We shall derive the existence of~$N(\loc)$ from \cref{thm:Entanglements}, which requires $\loc$ to be tangle-accessible. \cref{thm:NestedSetsInducesTreeDecTechnical} helps ensure this when $\loc$ is accessible, which we already know from \cref{thm:G_risAccessible}.

Let us say that a set $S$ of separations of a graph $G$ is \emph{point-finite} if for every vertex $v \in V(G)$ the set $S_v := \big\{ \{A, B\} \in S \colon v \in A \cap B \big\}$ is finite.
It is easy to see that a regular \td\ has finite spread if (and only if) its induced nested set of separations is point-finite.%
	\COMMENT{}
	Indeed, if a vertex of~$G$ lies in infinitely many parts then, by \cref{ax:TreeDecomp2} and \cite{DiestelBook25}*{Prop.\,8.2.1}, the decomposition tree has an infinite star or a ray such that $v$ lies in every part~$V_t$ corresponding to a node~$t$ in that star or ray. Then $v$ also lies in all the separators that correspond to edges of that star or ray.\looseness=-1

\begin{lemma} \label{thm:NestedSetsInducesTreeDecTechnical}
	Let $G$ be a connected locally finite graph, and let $N$ be a nested set of tight finite-order separations of~$G$.
	Let $\Gamma$ be a group acting quasi-transitively on~$G$ so that $N$ is $\Gamma$-canonical.
	Suppose that there exists a number $K \in \N$ such that some $\Gamma$-canonical set $N_{\rm end} \subseteq N$ of separations of order $<K$ distinguishes all the ends of~$G$.
	
	Then $N$ is exhaustive and point-finite, and the order of the separations in $N$ is finitely bounded.%
	\COMMENT{}
	The \td\ $\cT = (T, \cV)$ of~$G$ which $N$~induces is regular, $\Gamma$-canonical, of finitely bounded adhesion, and has finite spread.
	If $G$ is 1-ended, then $T$ is rayless.
\end{lemma}

For a proof of \cref{thm:NestedSetsInducesTreeDecTechnical} we need two  observations.

\begin{lemma} \label{lem:ExhaustiveImpliesPointFinite}
	Let $G$ be a connected locally finite graph, and let $N$ be a nested set of tight%
   \COMMENT{}
   separations of~$G$.
		If $N$ is exhaustive, then $N$ is point-finite.
\end{lemma}
\begin{proof}
	Let $v \in V(G)$, and suppose the set $N_v$ of separations in~$N$ whose separator contains~$v$ is infinite.
	Since $N$ is exhaustive, so is~$N_v$.
	By \cref{lem:ExhaustiveImpliesTreeDecomp}, $N_v$~induces a regular \td\ $\cT = (T,\cV)$ of~$G$. If $T$ contains a ray, then the forward-oriented edges of this ray define an infinite strictly increasing sequence of separations in $N_v$ via~$\alpha_\cT$.%
   \COMMENT{}
   All these separations have $v$ in their separator,%
	\COMMENT{}
	which contradicts the fact that $N_v$ is exhaustive.

So $T$ is rayless, and hence has a vertex $t$ of infinite degree \cite{DiestelBook25}*{Prop.\,8.2.1}.
	Applying $\alpha_\cT$ to the edges incident with~$t$ and oriented towards~$t$ yields an infinite set $\{(A_i, B_i) \colon i \in \N\}\subset\vN_v$ such that $(A_i, B_i) < (B_j, A_j)$ for all~$i \neq j$.
	All the sets $A_i \setminus B_i$ are disjoint, and $v$ has a neighbour in every $A_i \setminus B_i$ since $v \in A_i \cap B_i$ and $\{A_i, B_i\}$ is tight.
	This contradicts the local finiteness of~$G$.
\end{proof}

In quasi-transitive graphs, canonical and point-finite sets of separations are again quasi-transitive under the action of the same group:

\begin{lemma} \label{thm:CanonicalQuasiTransFinitelyBoundedOrder}
	Let $G$ be a connected graph, not necessarily locally finite, and let $S$ be a point-finite set of finite-order separations of~$G$.
	Let $\Gamma$ be a group acting quasi-transitively on $G$ so that $S$ is $\Gamma$-canonical.
	Then $\Gamma$ acts on $S$ with finitely many orbits.
	In particular, the order of the separations in $S$ is finitely bounded.
\end{lemma}

\begin{proof}
	As $G$ is quasi-transitive under the action of~$\Gamma$, there exists a finite set $V'$ of vertices of $G$ such that $\Gamma \act V' = V(G)$.
	Then $S\setminus\big\{\{\emptyset,V(G)\}\big\} = \Gamma \act S'$%
	\COMMENT{}
	where~$S' = \bigcup_{v' \in V'} S_{v'}$.
	Since $S$ is point-finite, all these $S_{v'}$ are finite, and hence so is~$S'$.
	Thus, $\Gamma$ acts on~$S$ with finitely many orbits, while $\{\emptyset,V(G)\}$, if it lies in~$S$, forms an additional orbit in~$S$.
Since $S'$ is finite, the order of the separations in~$S$ is finitely bounded.
\end{proof}

\begin{proof}[Proof of \cref{thm:NestedSetsInducesTreeDecTechnical}]
	By \cref{thm:NestedSetIsTDGen}, $N$~is exhaustive. By \cref{lem:ExhaustiveImpliesPointFinite} it is point-finite, and by \cref{thm:CanonicalQuasiTransFinitelyBoundedOrder} its separations have finitely bounded order.
	By \cref{lem:ExhaustiveImpliesTreeDecomp}, $N$~induces a regular $\Gamma$-canonical \td\ $\cT = (T, \cV)$ of~$G$. Its adhesion sets are the separators of the separations in~$N$, so it has finitely bounded adhesion.%
	\COMMENT{}
	As remarked just before \cref{thm:NestedSetsInducesTreeDecTechnical}, the point-finiteness of~$N$ implies that $\cT$ has finite spread.

If $G$ is 1-ended then, by \cref{thm:NestedIsTDOneEnded}, $\vN$~contains no infinite strictly increasing sequence. Any ray in~$T$ would define such a sequence%
   \COMMENT{}
   via~$\alpha_\cT$, so $T$ is rayless.
	\end{proof}

\cref{thm:NestedSetsInducesTreeDecTechnical} essentially strengthens accessibility to tangle-accessibility in our context. We need this for our proof of Theorems \ref{main:KeyTheorem} and~\ref{main:CayleyThm}, as outlined at the start of this section. More precisely, what we shall need is \cref{maincor:TangleAccessibility},%
	\COMMENT{}
	which we will thus prove first. Our input~$N$ for \cref{thm:NestedSetsInducesTreeDecTechnical} in the proof of \cref{maincor:TangleAccessibility} will come from the following non-trivial result:

\begin{lemma}[\cite{InfiniteSplinters}*{Theorem 6.6}] \label{thm:Splinters}
	Every connected, locally finite graph has a canonical tree of tangles consisting of relevant separations.\footnote{The relevance of the separations is not explicitly stated in~\cite{InfiniteSplinters}*{Theorem 6.6} but immediate from its proof.}
\end{lemma}

We can now prove \cref{maincor:TangleAccessibility}, which we restate:%
	\COMMENT{}

\begin{customthm}{\cref{maincor:TangleAccessibility}}
	Accessible locally finite, quasi-transitive graphs are even%
	\COMMENT{}
	tangle-accessible.
\end{customthm}

\begin{proof}
	Let $G$ be a graph as stated. Since it is quasi-transitive, it has only finitely many isomorphism types of components, so we may assume $G$ to be connected.%
	\COMMENT{}
	Let $N := N'(G)$ be the tree of tangles given by \cref{thm:Splinters}.

By \cref{lem:ToTProperties}, $N$~is a set of tight finite-order separations of~$G$. Since ends define tangles, and these are distinguished by exactly the separations that distinguish those ends (\cref{subsec:TanglesBackground}), $N$~distinguishes the ends of~$G$ efficiently.%
	\COMMENT{}

Since $G$ is accessible, the order of the separations of~$G$ needed to distinguish its ends is finitely bounded. As $N$ is canonical and distinguishes the ends efficiently, it thus has a canonical subset~$N_{\rm end}$ of separations of finitely bounded order that achieves this too~-- for example, its subset of all efficient end-distinguishers.
	By \cref{thm:NestedSetsInducesTreeDecTechnical}, then, the separations in~$N$ have finitely bounded order too. Since $N$ distinguishes the set of all tangles of~$G$,%
	\COMMENT{}
	it thus witnesses that $G$ is tangle-accessible.
\end{proof}

Note that the tree of tangles $N'(G)$ given by \cref{thm:Splinters}, which we just used in the proof of \cref{maincor:TangleAccessibility}, exists for arbitrary connected locally finite graphs, not just quasi-transitive ones. In general it may have limit points, in which case it will not define a \td. But  even when $G$ is tangle-accessible and $N'(G)$ does define a \td, it need not coincide with our canonical tree of tangles~$N(G)$ from \cref{def:N(G)}.

Turning now to our canonical tree of tangles~$N(G)$ from \cref{def:N(G)}, recall that in order to ensure its existence we needed that $G$ is tangle-accessible (\cref{thm:Entanglements}). By \cref{maincor:TangleAccessibility}, this now follows already from classical accessibility, which we shall obtain from \cref{thm:G_risAccessible}.

Our next lemma records that, in this case, $N(G)$~induces the desired \td:

\begin{lemma}\label{thm:N(G)isexhaustive}
	Let $G$ be a connected, locally finite, quasi-transitive,  accessible graph.
	Then $N(G)$ is defined and induces a canonical, regular \td\ of~$G$ of finitely bounded adhesion and with finite spread. If~$G$ is 1-ended, the decomposition tree is rayless.
\end{lemma}

\begin{proof}
	By \cref{maincor:TangleAccessibility} and \cref{thm:Entanglements}, $N(G)$~exists, is canonical, and consists of relevant (and therefore tight)%
	\COMMENT{}
	separations of finitely bounded order. Since $N(G)$ distinguishes all the ends of~$G$, it induces the desired \td\ by \cref{thm:NestedSetsInducesTreeDecTechnical}.
	\end{proof}

Summing everything up, we can now prove \cref{main:KeyTheoremGen}:

\begin{proof}[Proof of \cref{main:KeyTheoremGen}]
	Let $G$ and $r$ be as stated in the theorem. The $r$-local cover~$\loc$ of~$G$ is connected, locally finite, and quasi-transitive by \cref{lem:G_rIsQuasitransitive}, and accessible by \cref{thm:G_risAccessible}.
	By \cref{thm:N(G)isexhaustive}, $N(\loc)$~is defined and induces a canonical regular \td~$\cT$ of~$\loc$ of finite spread. As $p_r\colon\loc\to G$ is also canonical (see \cref{subsec:DefLocalCover}), $\cT$~defines a canonical \gd\ of~$G$ of finite spread (\cref{main:TreeDecToGraphDec}). By \cref{def:r-globalstructure}, this is the unique \gd\ of~$G$ that displays its $r$-global structure.
\end{proof}

\begin{proof}[Proof of~\cref{main:KeyTheorem}] \cref{main:KeyTheorem} is the special case of \cref{main:KeyTheoremGen} where $G$ is finite.
\end{proof}

\begin{proof}[Proof of~\cref{main:CayleyThm}]
	Since $\cayley{\Gamma}{S}$ is locally finite and transitive, \cref{main:KeyTheoremGen} implies \cref{main:CayleyThm}.
\end{proof}

\ifArXiv


\bigbreak

\appendix

\section{\boldmath Tree-Decompositions in Quasi-Transitive Graphs} \label{app:ProofsTreeDecHamann}

In this appendix we give a detailed proof of \cref{lem:TransformedTreeDec}.
As our construction is related to that in~\cite{hamann2018stallings}*{Proposition 4.1}, most of the arguments in the proof are essentially adaptations of the respective proofs in~\cite{hamann2018stallings}*{Section 4}.
For better readability of this appendix, we split up \cref{lem:TransformedTreeDec} into four separate lemmas, \cref{lem:ConstructionYieldsTreeDec,lem:StabiliserActsQuasiTrans,lem:NewTreeDecDistEnds,lem:PartsOneEnded}, which we prove in turn.

Let us briefly recall the setup we are working in: $G$~is a connected, locally finite graph, and $\Gamma$ is a group acting quasi-transitively on~$G$.
Given the $\Gamma$-canonical \td\ $\cT' = (T, \cV')$ of $G$ of finitely bounded adhesion from \cref{lem:NendInducesTreeDec}, we found in \cref{sec:NestedSetisTDArbitraryGraphs} an integer $d > 0$ such that for every edge $tt' \in T$ and every component $C$ of $G - V'_{tt'}$, where $V'_{tt'} := V'_t \cap V'_{t'}$, the graph $C \cap B_{G}(V'_{tt'}, d)$ contains all the shortest paths in~$C$ between neighbours of~$V'_{tt'}$. Clearly,%
	\COMMENT{}
	all these graphs $C \cap B_{G}(V'_{tt'}, d)$ are connected.
We then defined $\cT = (T, \cV)$ by letting $V_t := V(B_G(V'_t,d))$ for each node~$t \in T$.

Let us begin by showing \cref{lem:TransformedTreeDec}~(i), that $(T, \cV)$ is a \td\ into connected parts:%
	\COMMENT{}

\begin{lemma} \label{lem:ConstructionYieldsTreeDec}
	$\cT = (T, \cV)$ is a $\Gamma$-canonical \td\ of finitely bounded adhesion whose parts are connected. 
\end{lemma}

\begin{proof}
	We first check that $\cT$ is indeed a \td.
	The $\Gamma$-canonicity of $\cT$ then follows immediately from the construction and the $\Gamma$-canonicity of~$\cT'$.
	\cref{ax:TreeDecomp1} is clear. 
	For \cref{ax:TreeDecomp2}, let $t_1,t_2, t_3$ be nodes of~$T$ with $t_3 \in t_1Tt_2$, and $v$ any vertex in~$V_{t_1} \cap V_{t_2}$.
	We consider two cases.

	Assume first  that~$v \in V'_{t_1} \cap V_{t_2}$.
	Let $w$ be a vertex of $V'_{t_2}$ such that $d_G(v,w) \leq d$, and fix a shortest $v$--$w$ path $P$ in~$G$.
	Then $P$ meets~$V'_{t_3}$, by \cref{ax:TreeDecomp2} for~$\cT' = (T, \cV')$.
	Hence,~$v \in V(B_G(V'_{t_3},d)) = V_{t_3}$. 
	Note that the case of $v \in V_{t_1} \cap V'_{t_2}$ is analogous. This completes the first case.

	For the second case, assume that~$v \in (V_{t_1} \setminus V'_{t_1}) \cap (V_{t_2} \setminus V'_{t_2})$.
	By \cref{ax:TreeDecomp1} for $\cT' = (T, \cV')$, there exists a node $t_4$ of~$T$ such that~$v \in V'_{t_4}$.
	Since the $t_1$--$t_2$ path in~$T$ is unique, the node~$t_3$ lies on the $t_1$--$t_4$ path in~$T$ or on the $t_2$--$t_4$ path in~$T$.
	So we are done by applying the first case to the pair $t_1,t_4$ or the pair~$t_2,t_4$.
	This completes the proof of \cref{ax:TreeDecomp2}, so $\cT$ is a \td.
	
	For any edge $tt'$ of~$T$, the corresponding adhesion set in~$\cT$ has the form~$V(B_G(V'_{tt'}, d))$.%
	\COMMENT{}
	Thus, the adhesion sets of $\cT$ are of finitely bounded size, since $G$ is locally finite and $\cT'$ has finitely bounded adhesion.
	Moreover, our choice of $d$ guarantees that all the adhesion sets of $\cT$ are connected.
	It remains to deduce from this that~$G[V_t]$ is connected for every node~$t$ of~$T$.

	Recall that~$G$ is connected.
	So consider any two vertices~$u, v \in V_t$ and a~$u$--$v$ path~$P$ in~$G$ with as few vertices outside~$V_t$ as possible.
	Now for any maximal subpath~$P'$ of~$P$ whose internal vertices are not in~$V_t$, both endvertices are in~$G[V_t]$ and more precisely in the same adhesion set~$V_{tt'}$ for some~$t' \in N_T(t)$.
	But since~$V_{tt'}$ is connected, we may replace~$P'$ by a path in~$G[V_{tt'}]$, hence obtaining a~$u$--$v$ walk in~$G$ with strictly fewer vertices outside~$V_t$ than~$P$, and thus such a~$u$--$w$ path in~$G$.
	So the minimal choice of~$P$ yields that~$P$ must be completely contained in~$G[V_t]$, which completes the proof.
\end{proof}

Next we prove \cref{lem:TransformedTreeDec}~(ii): that a subgroup of $\Gamma$ acts quasi-transitively on each part~$G[V_t]$ of~$\cT$.%
	\COMMENT{}

\begin{lemma} \label{lem:StabiliserActsQuasiTrans}
	For every node $t \in T$, the stabiliser $\Gamma_t$ of~$t$ under the action of~$\,\Gamma$ on~$T$ acts quasi-transitively on~$G[V_t]$.
\end{lemma}

\begin{proof}
	Since $\cT$ is $\Gamma$-canonical, we have $\phi(V_t) = V_t$ for every $\phi \in \Gamma_t$; thus, the action of $\Gamma$ on $G$ induces an action of $\Gamma_t$ on $G[V_t]$. Moreover, $\phi(V_{st}) = V_{\phi(s)\phi(t)}$ for every $\phi \in \Gamma$. Hence, $\Gamma_t$ acts on the set $W_t := \{v \in V_t \mid \forall s \in N_T(t) \colon v \notin V_{st} \}$ of `inner vertices' of~$V_t$ and on the set $V_t \setminus W_t$. It is enough to show that $\Gamma_t$ does so quasi-transitively.
	
	To show that $\Gamma_t$ acts quasi-transitively on $W_t$, let $u,v\in W_t$ be such that $\phi(u) = v$ for some~$\phi \in \Gamma$.
	Since $\cT$ is a $\Gamma$-canonical \td, $\phi(W_t) = W_{\phi(t)}$.
	By~\cref{ax:TreeDecomp2}, $t$~is the unique node of~$T$ containing $v$ as inner vertex, and hence $t = \phi(t)$. Thus,~$\phi \in \Gamma_t$. So, $\Gamma_t$ has finitely many orbits on $W_t$, since $\Gamma$ acts quasi-transitively on $G$.
	
	Second, let us consider the action of $\Gamma_t$ on $V_t \setminus W_t$: the union of all adhesion sets associated with edges of~$T$ at~$t$.
	Since each adhesion set $V_{st}$ is finite, the action of $\Gamma_t$ has only finitely many orbits on this adhesion set: at most~$|V_{st}|$.
	Thus, it is enough to show that~$\Gamma_t$ acts quasi-transitively on the set~$\{st\mid s \in N_T(t)\}$ of edges of~$T$.
	By \cref{lem:DecTreeFiniteOrbits}, there are only finitely many $\Gamma$-orbits of~$E(T)$.
	We now claim that each $\Gamma$-orbit $O$ of~$E(T)$ has at most two $\Gamma_t$-orbits.
	For this, let $s_0, s_1, s_2$ be neighbours of~$t$ in~$T$ such that $ts_0, ts_1, ts_2 \in O$. Then there exist $\phi_1, \phi_2 \in \Gamma$ such that $\phi_1(ts_0) = ts_1$ and~$\phi_2(ts_0) = ts_2$.
	So if $\phi_1$ and $\phi_2$ are not in $\Gamma_t$, then~$\phi_2\phi_1^{-1} \in \Gamma_t$.
	This proves the claim and thus completes the proof.
\end{proof}

Towards a proof of \cref{lem:TransformedTreeDec}~(iii), we first show that $\cT$ distinguishes all ends of~$G$, just as $\cT'$ does by definition:

\begin{lemma} \label{lem:NewTreeDecDistEnds}
	$\cT$ distinguishes all ends of~$G$.
\end{lemma}

\begin{proof}
	The \td\ $\cT'$ distinguishes all the ends of~$G$, because it is induced by~$N_{\rm end}$ (\cref{lem:NendInducesTreeDec}).
	Let $tt'$ be an edge of~$T$, and let $\alpha_\cT(t,t') = (A,B)$ be the oriented separation of~$G$ associated with $tt'$ in~$\cT$.
	By construction of $\cT$, we have $(A,B) = (B_G(A',d),B_G(B',d))$, where $(A',B')$ is the oriented separation $\alpha_{\cT'}(t,t')$ associated with~$tt'$ in~$\cT'$.
	Thus, the separator $A \cap B$ of $(A,B)$ is~$B_G(A'\cap B', d)$.%
	\COMMENT{}
	Hence it is finite and contains $A'\cap B'$.
	So $(A,B)$ distinguishes every two ends of~$G$ that are distinguished by~$(A',B')$.
	Altogether, the \td~$\cT$ distinguishes all the ends of~$G$ because~$\cT'$ does.
	\end{proof}

Our last lemma deduces \cref{lem:TransformedTreeDec}~(iii) from \cref{lem:NewTreeDecDistEnds}: that the parts of $\cT$ have at most one end:

\begin{lemma} \label{lem:PartsOneEnded}
	Let $t$ be a node of~$T$.
	Then $G[V_t]$ is either finite or 1-ended. 
\end{lemma}

\begin{proof}
	If $V_t$ is infinite, then $G[V_t]$ is an infinite and locally finite graph, which is connected by \cref{lem:ConstructionYieldsTreeDec}.
	By K\"onig's infinity lemma~\cite{DiestelBook25}*{Proposition 8.2.1}, this graph $G[V_t]$ has at least one end~$\omega$.
	Since $\cT$ has finitely bounded adhesion, the end $\omega_t$ of $G$ with $\omega \subseteq \omega_t$ has the property that every ray $R \in \omega_t$ meets~$V_t$ infinitely often.
	We claim that $\omega_t$ is the unique end of $G$ with this property. 

Suppose there exists another such end $\omega'_t$ of~$G$. Since $\cT$ distinguishes all ends of~$G$, by \cref{lem:NewTreeDecDistEnds}, $T$~has an edge whose associated separation $\{A, B\}$ of~$G$ distinguishes $\omega_t$ from~$\omega'_t$.
	Then the part $V_t$ lies on one side of this separation, say in~$A$.
	Any ray $R$ in $\omega_t$ or $\omega'_t$ meets the finite set $A \cap B$ only finitely often, and hence has a subray of $R$ on one side of $\{A, B\}$. This side must be~$A$, since $R$ meets $V_t$ infinitely often by our assumption about $\omega_t$ and~$\omega'_t$.
	But this contradicts the fact that $\{A,B\}$ distinguishes $\omega_t$ from~$\omega'_t$.
	
	So $\omega_t$ is the unique end of $G$ with the property that every ray $R \in \omega_t$ meets $V_t$ infinitely often.
	The statement now follows directly from~\cite{hamann2018stallings}*{Proposition 4.8 (iv)}.
\end{proof}

\fi

\bibliographystyle{amsplain}
\bibliography{collective.bib}

\end{document}